\documentclass[11pt]{amsart}
\usepackage{amsmath,amssymb,amsthm}
\usepackage{bm}
\usepackage{pinlabel}
\usepackage[hypcap=true,font=small,margin=10pt]{caption}
\usepackage[hypcap=true,format=hang,labelfont=up]{subcaption}
\usepackage[colorlinks=true,linkcolor=blue,citecolor=blue,urlcolor=blue]{hyperref}

\title[Bordered Floer homology for sutured manifolds]
	{Bordered Floer homology for sutured manifolds}
\author[Rumen Zarev]{Rumen Zarev}
\address{Department  of Mathematics\\Columbia University\\
	New York\\NY 10027}
\email{rzarev@math.columbia.edu}
\thanks{The author was partially supported by NSF grant number DMS-0706815.}

\def\co{\colon\thinspace}
\newcommand\lu[1]{\mbox{}^{#1}\!}
\newcommand\li[1]{\mbox{}#1\!}

\DeclareMathOperator{\ddiv}{div}
\DeclareMathOperator{\emb}{emb}

\DeclareMathOperator{\Int}{Int}
\DeclareMathOperator{\Inv}{Inv}
\DeclareMathOperator{\inv}{inv}
\DeclareMathOperator{\ind}{ind}
\DeclareMathOperator{\id}{id}
\DeclareMathOperator{\im}{im}
\DeclareMathOperator{\spin}{Spin}
\DeclareMathOperator{\rel}{rel}
\DeclareMathOperator{\ev}{ev}

\DeclareMathOperator{\Gr}{Gr}
\DeclareMathOperator{\PD}{PD}

\newcommand{\Grdn}{\underline{\Gr}}
\DeclareMathOperator{\gr}{gr}

\newcommand{\grdn}{\underline{\gr}}
\DeclareMathOperator{\Hom}{Hom}

\DeclareMathOperator{\Mod}{Mod}

\newcommand{\spinc}{\spin^{c}}
\newcommand{\F}{\mathcal{F}}
\newcommand{\A}{\mathcal{A}}
\newcommand{\Z}{\mathcal{Z}}
\newcommand{\HH}{\mathcal{H}}

\newcommand{\M}{\mathcal{M}}

\newcommand{\I}{\mathcal{I}}
\newcommand{\G}{\mathcal{G}}
\newcommand{\V}{\mathcal{V}}
\newcommand{\W}{\mathcal{W}}
\newcommand{\s}{\mathfrak{s}}

\newcommand{\C}{\mathcal{C}}
\newcommand{\D}{\mathcal{D}}
\newcommand{\B}{\mathcal{B}}
\newcommand{\PP}{\mathcal{P}}
\newcommand{\scat}{\mathcal{S}}
\newcommand{\sdcat}{\mathcal{SD}}

\newcommand{\ZZ}{\mathbb{Z}}
\newcommand{\II}{\mathbb{I}}
\newcommand{\RR}{\mathbb{R}}
\newcommand{\CC}{\mathbb{C}}
\newcommand{\DD}{\mathbb{D}}
\newcommand{\QQ}{\mathbb{Q}}
\newcommand{\wdd}{\omega_{\DD}}
\newcommand{\Zmid}{\mathbf{Z}}
\newcommand{\amid}{\mathbf{a}}
\newcommand{\wsig}{\omega_{\Sigma}}
\newcommand{\jd}{j_{\DD}}
\newcommand{\jsig}{j_{\Sigma}}
\newcommand{\pd}{\pi_{\DD}}
\newcommand{\psig}{\pi_{\Sigma}}
\newcommand{\del}{\partial}
\newcommand{\Ainf}{\mathcal{A}_{\infty}}
\newcommand{\CFhat}{\widehat{\textit{CF}}}
\newcommand{\HFhat}{\widehat{\textit{HF}}}
\newcommand{\CFD}{\widehat{\textit{CFD}}}
\newcommand{\CFA}{\widehat{\textit{CFA}}}
\newcommand{\BSD}{\widehat{\textit{BSD}}}
\newcommand{\CSFD}{\BSD}
\newcommand{\BSDA}{\widehat{\textit{BSDA}}}
\newcommand{\BSDAM}{\BSDA_{M}}
\newcommand{\BSDM}{\BSD_{M}}
\newcommand{\CSFDD}{\BSDM}
\newcommand{\BSA}{\widehat{\textit{BSA}}}
\newcommand{\CSFA}{\BSA}
\newcommand{\SFH}{\textit{SFH}}
\newcommand{\SFC}{\textit{SFC}}
\newcommand{\tDD}{\textit{DD}}
\newcommand{\DA}{\textit{DA}}
\newcommand{\dtens}{\mathbin{\widetilde{\otimes}}}
\newcommand{\sqtens}{\boxtimes}

\newcommand{\xgen}{\mathbf{x}}
\newcommand{\ygen}{\mathbf{y}}
\newcommand{\zgen}{\mathbf{z}}
\newcommand{\brho}{\boldsymbol\rho}
\newcommand{\brarrow}{\overrightarrow{\brho}}
\newcommand{\bcirc}{\mathbin{\boldsymbol\circ}}
\newcommand{\balpha}{\boldsymbol\alpha}
\newcommand{\bbeta}{\boldsymbol\beta}
\newcommand{\sigbar}{\overline{\Sigma}}
\newcommand{\sigint}{\Int(\Sigma)}
\newcommand{\ebar}{\overline{e}}
\newcommand{\obar}{\overline{o}}
\newcommand{\sbar}{\overline{s}}
\newcommand{\sige}{\Sigma_{\ebar}}
\newcommand{\trr}{\triangleright}

\newcommand{\parrow}{\overrightarrow{P}}
\newcommand{\mt}{\widetilde{\M}}

\newcommand{\dbar}{\overline{\del}}

\newtheorem{THM}{Theorem}
\newtheorem{COR}[THM]{Corollary}
\newtheorem{thm}{Theorem}[section]
\newtheorem{prop}[thm]{Proposition}
\newtheorem{cor}[thm]{Corollary}

\theoremstyle{definition}
\newtheorem{defn}[thm]{Definition}

\theoremstyle{remark}
\newtheorem{exm}{Example}[section]
\newtheorem*{rmk}{Remark}
\begin{document}

\begin{abstract}
We define a sutured cobordism category of surfaces with boundary and
3--manifolds with corners. In this category a sutured 3--manifold is regarded as a morphism from the empty
surface to itself.
In the process
we define a new class of geometric objects, called bordered sutured manifolds,
that generalize both sutured 3--manifolds and bordered 3--manifolds.
We extend the definition of bordered Floer homology to these objects, giving
a functor from a decorated version of the sutured category to $\Ainf$--algebras,
and $\Ainf$--bimodules.

As an application we give a way to recover the
sutured homology $\SFH(Y,\Gamma)$ of a sutured manifold from either of the
bordered invariants $\CFA(Y)$ and $\CFD(Y)$ of its underlying manifold $Y$.
A further application is
a new proof of the surface decomposition formula of Juh\'asz.
\end{abstract}

\maketitle
\tableofcontents

\section{Introduction}
\label{sec:introduction}

There are currently two existing Heegaard Floer theories for 3--manifolds with boundary, both of which require
extra structure on the 3--manifold.

The first such theory was developed by Juh\'asz in \cite{Juh:SFH}. It is defined for balanced sutured manifolds,
which are 3--manifolds with certain decorations on the boundary. While very fruitful, the theory has the disadvantage
that the homology $\SFH(Y,\Gamma)$ of a sutured manifold $Y$, with sutures $\Gamma$ tells us nothing about the
homology $\SFH(Y,\Gamma')$ of the same manifold with a different set of sutures $\Gamma'$.
Moreover, if a closed 3--manifold $Y$ is obtained by gluing two manifolds with boundary $Y_1$ and $Y_2$,
then the sutured Floer invariants $\SFH(Y_1,\Gamma_1)$ and $\SFH(Y_2,\Gamma_2)$ do not determine
the closed Heegaard Floer invariant $\HFhat(Y)$.

The second and more recent theory was developed by Lipshitz, Ozsv\'ath, and Thurston in \cite{LOT:pairing}.
It takes a more algebraically complicated form than sutured Floer homology. To a parametrized connected surface $F$
it assigns a differential graded algebra $\A(F)$, and to a 3--manifold $Y$ whose boundary is identified with $F$ it
assigns two different modules over $\A(F)$---an $\Ainf$--module $\CFA(Y)$, and a differential graded module
$\CFD(Y)$. These modules are well-defined up to homotopy equivalence, and have the property that
$\HFhat(Y_1\cup Y_2)$ is the homology of $\CFA(Y_1)\dtens\CFA(Y_2)$. While the bordered
invariants of a manifold $Y$ still depend on extra structure---in this case the parametrization of $\del Y$---the invariant
of one parametrization can be obtained from that of any other.

In the present work we relate these two invariants by a common generalization. We define a new set of
topological objects---bordered sutured manifolds, sutured surfaces, and a sutured cobordism category---and define several
invariants for them. We show that these invariants interpolate between sutured Floer homology and bordered Floer
homology, containing each of them as special cases. An advantage of the new invariants is that they are
well-defined and nontrivial even for unbalanced manifolds, unlike $\SFH$.

The aim of the present work is to define this common generalization. As a byproduct, we show how to obtain
the sutured invariant $\SFH(Y,\Gamma)$ of a manifold with connected boundary from either of its bordered
invariants $\CFD(Y)$ and $\CFA(Y)$. Finally, we give a new proof of the sutured decomposition formula of Juh\'asz.

\subsection{The sutured and decorated sutured categories}
A simplified description of a sutured manifold is given below. A more precise version is given in 
section~\ref{sec:sutured}.

\begin{defn}
A \emph{sutured 3--manifold} $(Y,\Gamma)$ is a 3--manifold $Y$, with a multi-curve $\Gamma$ on its boundary,
dividing the boundary into a positive and negative region, denoted $R_+$ and $R_-$, respectively. We usually
impose the conditions that $Y$ has no closed components, and that $\Gamma$ intersects every
component of $\del Y$.
\end{defn}

We can introduce analogous notions one dimension lower.

\begin{defn}
A \emph{sutured surface} $(F,\Lambda)$ is a surface $F$, with a 0--manifold $\Lambda\subset\del F$, diving
the boundary $\del F$ into a positive and negative region, denoted $S_+$ and $S_-$, respectively. Again, we
impose the condition that $F$ has no closed components, and that $\Lambda$ intersects every
component of $\del F$.
\end{defn}

\begin{defn}
A \emph{sutured cobordism} $(Y,\Gamma)$ between two sutured surfaces $(F_1,\Lambda_1)$ and $(F_2,\Lambda_2)$ is
a cobordism $Y$ between $F_1$ and $F_2$, together with a collection of properly embedded arcs and circles
\begin{equation*}
\Gamma\subset\del Y\setminus (F_1\cup F_2),
\end{equation*}
dividing $\del Y\setminus (F_1\cup F_2)$ into a positive and negative region, denoted $R_+$ and $R_-$, respectively,
such that
$R_\pm\cap F_i=S_\pm(F_i)$, for $i=1,2$. Again, we require that $Y$ has no closed components, and that
$\Gamma$ intersects every component of $\del Y\setminus(F_1\cup F_2)$.
\end{defn}

There is a \emph{sutured category} $\scat$ whose objects are sutured surfaces, and whose
morphisms are sutured cobordisms. The identity morphisms are cobordisms of the form $(F\times[0,1],\Lambda\times[0,1])$,
where $(F,\Lambda)$ is a sutured surface. As a special case, sutured manifolds are the morphisms from
the empty surface $(\varnothing,\varnothing)$ to itself. 

Unfortunately, we cannot directly define invariants for the sutured category, and we need impose a little extra
structure.

\begin{defn}
An \emph{arc diagram} is a relative handle diagram for a 2--manifold with corners, 
where the bottom and top boundaries are both 1--manifolds with no closed components.
\end{defn}

\begin{defn}
A \emph{parametrized} or \emph{decorated sutured surface} is a sutured surface $(F,\Lambda)$ with a handle
decomposition given by an arc diagram $\Z$, expressing $F$ as a cobordism from $S_+$ to $S_-$.

A \emph{parametrized} or \emph{decorated sutured cobordism} is a sutured cobordism $(Y,\Gamma)$ from $(F_1,\Lambda_1)$ to $(F_2,\Lambda_2)$,
such that $(F_i,\Lambda_i)$ is parametrized by an arc diagram $\Z_i$, for $i=1,2$.
\end{defn}

Examples of a sutured surface and its decorated version are given in Fig.~\ref{fig:sutured_surface}.
A sutured cobordism and its decorated version are given in Fig.~\ref{fig:sutured_cobordism}.
We visualize the handle decomposition coming from an arc diagram by drawing the cores of the 1--handles.

\begin{figure}
\begin{subfigure}[b]{.495\linewidth}
	\centering
	\labellist
	\small \hair 2pt
	\pinlabel $S_+$ [tr] at 3 73
	\pinlabel $S_+$ [br] at 0 0
	\pinlabel $S_+$ [l] at 188 40
	\endlabellist
	\includegraphics[scale=.7]{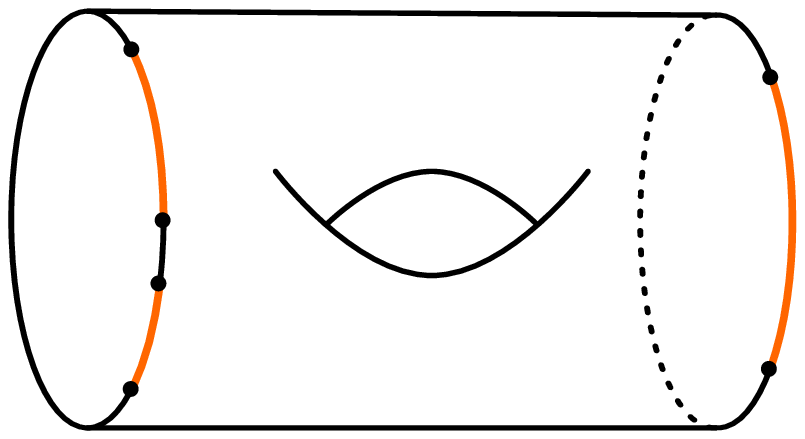}
	\caption{Unparametrized.}
	\label{subfig:sutured_surface_unparam}
\end{subfigure}
\begin{subfigure}[b]{.495\linewidth}
	\centering
	\includegraphics[scale=.7]{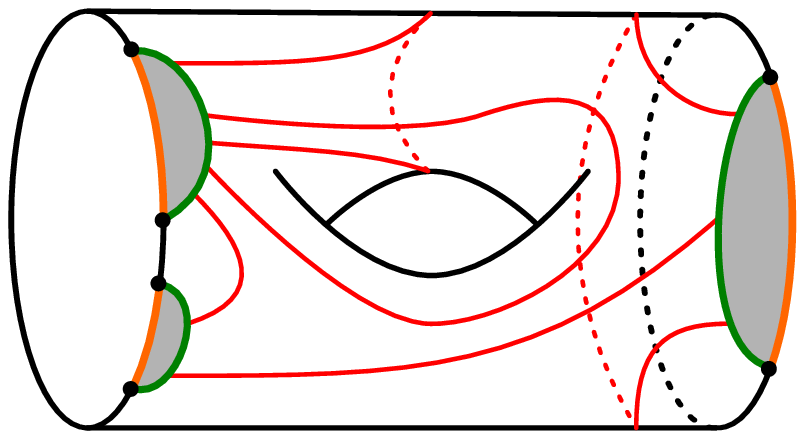}
	\caption{Parametrized by an arc diagram.}
	\label{subfig:sutured_surface_param}
\end{subfigure}
\caption{A sutured surface $(F,\Lambda)$.
The sutures $\Lambda$ are denoted by dots, while the positive region $S_+\subset\del F$ is colored in orange.}
\label{fig:sutured_surface}
\end{figure}

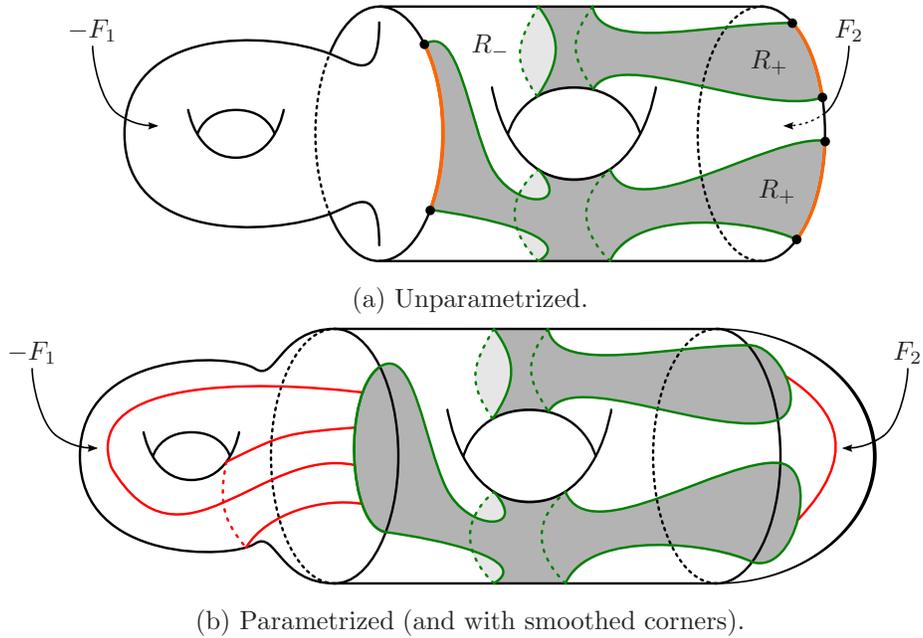
\begin{figure}
\begin{subfigure}[b]{\linewidth}
	\centering
	\labellist
	\small\hair 2pt
	\pinlabel $-F_1$ [b] at -20 255
	\pinlabel $F_2$ [b] at 455 255
	\pinlabel $R_+$ at 405 245
	\pinlabel $R_+$ at 410 164
	\pinlabel $R_-$ at 230 255
	\endlabellist
	\includegraphics[scale=.6]{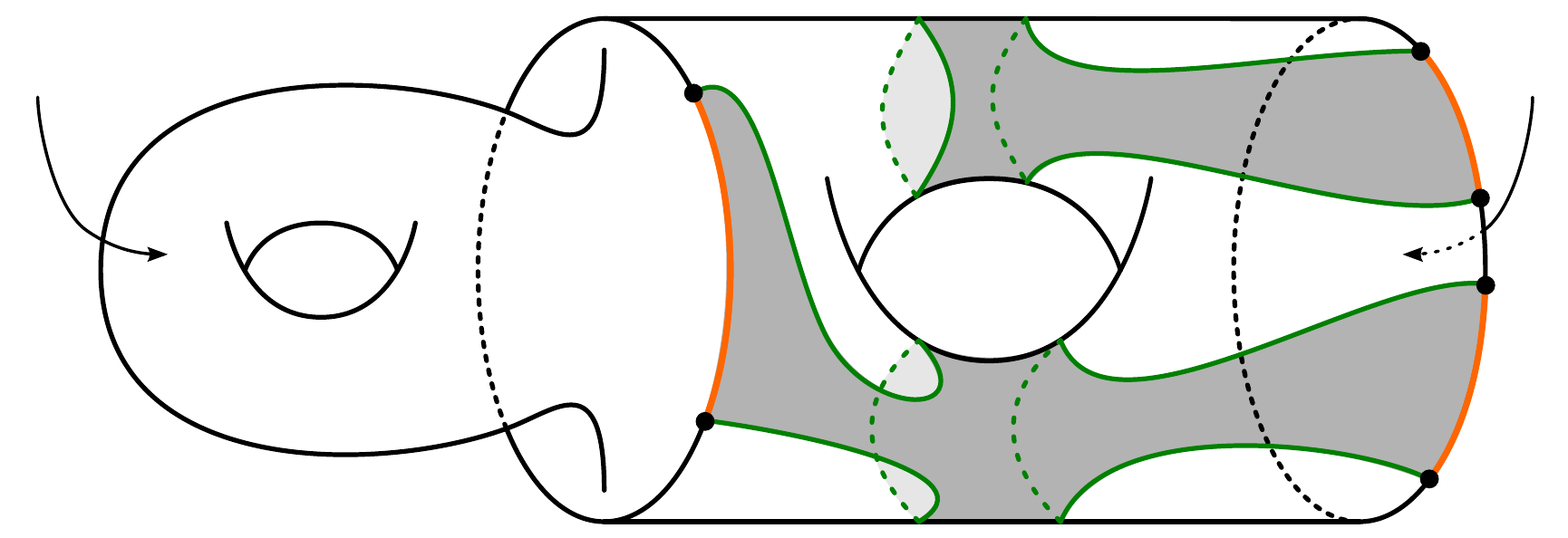}
	\caption{Unparametrized.}
	\label{subfig:sutured_cobordism_unparam}
\end{subfigure}
\begin{subfigure}[b]{\linewidth}
	\centering
	\labellist
	\small\hair 2pt
	\pinlabel $-F_1$ [b] at -30 255
	\pinlabel $F_2$ [b] at 520 255
	\endlabellist
	\includegraphics[scale=.6]{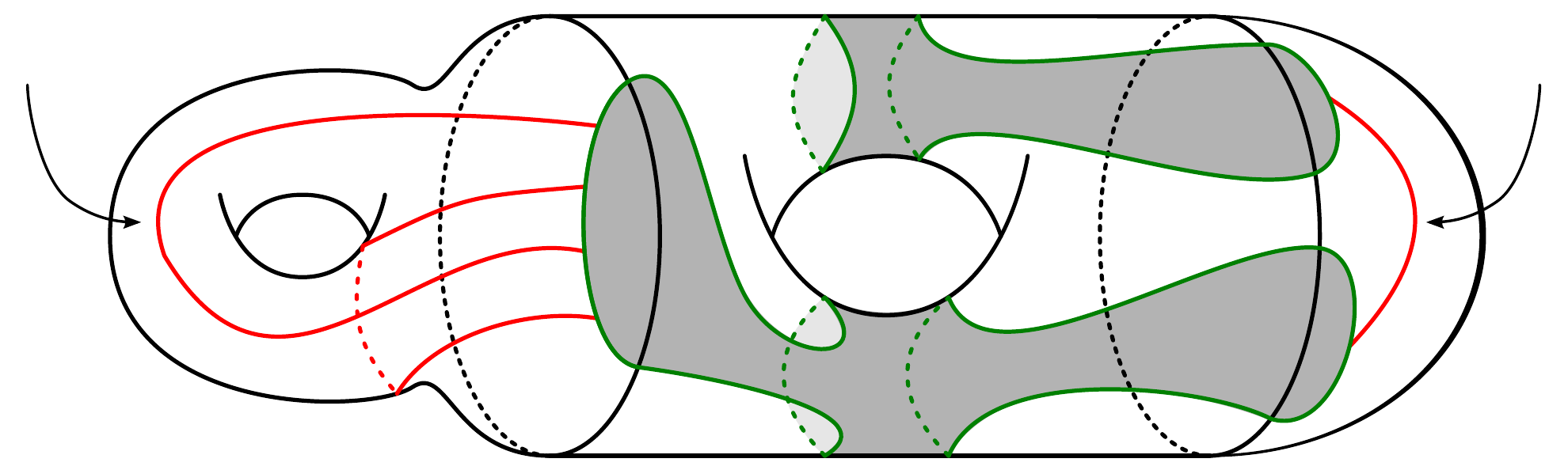}
	\caption{Parametrized (and with smoothed corners).}
	\label{subfig:sutured_cobordism_param}
\end{subfigure}
\caption{A sutured cobordism $(Y,\Gamma)$ from a once punctured torus to a disc.
The sutures $\Gamma$ are colored in green, while the positive region $R_+\subset\del Y$ is shaded.}
\label{fig:sutured_cobordism}
\end{figure}

The \emph{decorated sutured category} $\sdcat$
is a category whose objects are decorated sutured surfaces---or alternatively their arc diagrams---and
whose morphisms are decorated sutured cobordisms. Note that all decorations on the sutured identity
$(F\times[0,1],\Lambda\times[0,1])$ are isomorphisms, while the ones where the two parametrizations
on $F\times\{0\}$ and $F\times\{1\}$ agree are the identity morphisms in $\sdcat$. In particular,
any two parametrizations of the same sutured surface are isomorphic, and the forgetful functor
$\Z\mapsto F(\Z)$ is an equivalence of categories.

Sutured cobordisms have another, slightly different topological interpretation.
For a sutured
cobordism $(Y,\Gamma)$ from $(F_1,\Lambda_1)$ to $(F_2,\Lambda_2)$, we can smooth its corners, and set
$\Gamma'=\Gamma\cup S_+(F_1)\cup S_+(F_2)$. This turns $(Y,\Gamma')$ into a regular sutured manifold. Therefore, we can
think of a sutured cobordism as a sutured manifold, with two distinguished subsets $F_1$ and $F_2$ of its boundary.

Applying the same procedure to the decorated versions of sutured cobordisms, we come up with
the notion of \emph{bordered sutured manifolds},
defined more precisely in section~\ref{sec:sutured}.

\begin{defn}
A \emph{bordered sutured manifold} $(Y,\Gamma,\Z)$ is a sutured manifold $(Y,\Gamma)$, with a distinguished subset
$F\subset\del Y$, such that $(F,\del F\cap\Gamma)$ is a sutured surface, parametrized by the arc diagram $\Z$.
\end{defn}

Any bordered sutured manifold $(Y,\Gamma,\Z_1\cup\Z_2)$, where $\Z_i$ parametrizes $(F_i,\del F_i\cap\Gamma)$
gives a decorated sutured cobordism $(Y,\Gamma\setminus(F_1\cup F_2))$ from $-F_1$ to $F_2$, and vice versa.

\subsection{Bordered sutured invariants and TQFT}
To any arc diagram $\Z$---or alternatively decorated sutured surface parametrized by $\Z$---we
associate a differential graded algebra $\A(\Z)$, which is a subalgebra of some strand algebra, as defined in
\cite{LOT:pairing}.

These algebras behave nicely under disjoint union.
If $\Z_1$ and $\Z_2$ are arc diagrams, then $\A(\Z_1\cup\Z_2)\cong\A(\Z_1)\otimes\A(\Z_2)$.

To a bordered sutured manifold $(Y,\Gamma,\Z)$ we associate 
a right $\Ainf$--module $\CSFA(Y,\Gamma)$ over $\A(\Z)$, and
a left differential graded module $\CSFDD(Y,\Gamma)$ over $\A(-\Z)$.

Generalizing this construction, let $(F_1,\Lambda_1)$ and $(F_2,\Lambda_2)$ be two sutured surfaces,
parametrized by the arc diagrams $\Z_1$ and $\Z_2$, respectively.
To any sutured cobordism $(Y,\Gamma)$ between them we associate (a homotopy equivalence class of)
an $\Ainf$ $\A(\Z_1),\A(\Z_2)$--bimodule, denoted $\BSDAM(Y,\Gamma)$. This specializes to
$\CSFA(Y,\Gamma)$, respectively $\CSFDD(Y,\Gamma)$, when $F_1$, respectively $F_2$ is empty, or to
the sutured chain complex $\SFC(Y,\Gamma)$, when both are empty.

\begin{defn}
Let $\D$ be the category whose objects are differential graded algebras, and whose morphisms are
the graded homotopy equivalence classes of $\Ainf$--bimodules of any two such algebras. Composition is given
by the derived tensor product~$\dtens$. The identity is the homotopy equivalence class of the algebra
considered as a bimodule over itself.
\end{defn}

\begin{THM}
\label{thm:intro_pairing}
The invariant $\BSDAM$ respects compositions of decorated sutured cobordisms.
Explicitly, let $(Y_1,\Gamma_1,-\Z_1\cup\Z_2)$ and $(Y_2,\Gamma_2,-\Z_2\cup\Z_3)$ be
two bordered sutured manifolds, representing decorated sutured cobordisms
from $\Z_1$ to $\Z_2$, and from $\Z_2$ to $\Z_3$, respectively. Then there are graded homotopy equivalences
\begin{equation}
\label{eq:pairing_bsda}
\BSDAM(Y_1,\Gamma_1)\dtens_{\A(\Z_2)}\BSDAM(Y_2,\Gamma_2)\simeq\BSDAM(Y_1\cup Y_2,\Gamma_1\cup\Gamma_2).
\end{equation}

Specializing to $\Z_1=\Z_3=\varnothing$, we get
\begin{equation}
\label{eq:pairing_bsd_bsa}
\CSFA(Y_1,\Gamma_1)\dtens_{\A(\Z_2)}\CSFDD(Y_2,\Gamma_2)\simeq\SFC(Y_1\cup Y_2,\Gamma_1\cup\Gamma_2).
\end{equation}
\end{THM}

\begin{THM}
\label{thm:equivalence}
The invariant $\BSDAM$ respects the identity. In other words, if $(Y,\Gamma,-\Z\cup\Z)$ is the identity
cobordism from $\Z$ to itself, then $\BSDAM(Y,\Gamma)$ is graded homotopy equivalent to $\A(\Z)$ as an
$\Ainf$--bimodule over itself.
\end{THM}

Together, theorems~\ref{thm:intro_pairing} and~\ref{thm:equivalence} imply that $\A$ and $\BSDAM$ form a functor.

\begin{COR}
The invariants $\A$ and $\BSDAM$ give a functor from $\sdcat$ to $\D$, inducing a (non-unique) functor from
the equivalent category $\scat$ to $\D$. In particular, if $\Z_1$ and $\Z_2$ parametrize the same sutured surface,
then $\A(\Z_1)$ and $\A(\Z_2)$ are isomorphic in $\D$. In other words, there is an $\A(\Z_1),\A(\Z_2)$
$\Ainf$--bimodule providing an equivalence of the derived categories of $\Ainf$--modules over $\A(\Z_1)$ and $\A(\Z_2)$.
\end{COR}

\subsection{Applications}
The first application, which motivated the development of the new invariants, is to show that $\SFH(Y,\Gamma)$ can be
computed from $\CFA(Y)$ or $\CFD(Y)$.

\begin{THM}
\label{thm:intro_bordered_to_sutured}
Suppose $Y$ is a connected 3--manifold with connected boundary.
With any set of sutures $\Gamma$ on $\del Y$ we can associate modules $\CFA(\Gamma)$ and $\CFD(\Gamma)$ over
$\A(\pm\del Y)$, of the appropriate form, such that the following formula holds.
\begin{equation}
\SFH(Y,\Gamma)\cong H_*(\CFA(Y)\dtens\CFD(\Gamma))\cong H_*(\CFA(\Gamma)\dtens\CFD(Y)).
\end{equation}
\end{THM}

We prove a somewhat stronger version of theorem~\ref{thm:intro_bordered_to_sutured} in section~\ref{sec:bordered_to_sutured}. 

Another application of the invariants is a new proof of the surface decomposition formula for $\SFH$.
\begin{THM}
\label{thm:intro_decomposition}
Suppose $(Y,\Gamma)$ is a balanced sutured manifold, and $S$ is a good decomposing surface, i.e. $S$
has no closed components, and any component of $\del S$ intersects $\Gamma$. If $S$ decomposes
$(Y,\Gamma)$ to $(Y',\Gamma')$, then for a certain
subset $O$ of relative $\spinc$ structures on $(Y,\del Y)$, called outer for $S$, the following
equality holds.
\begin{equation}
\SFH(Y',\Gamma')\cong\bigoplus_{\s\in O}\SFH(Y,\Gamma,\s).
\end{equation}
\end{THM}

The proofs of both statements rely on the fact that by splitting a sutured manifold into bordered sutured pieces we
can localize the calculations.

For theorem~\ref{thm:intro_bordered_to_sutured}, we localize the suture information,
by considering a (punctured) collar neighborhood of $\del Y$, which knows about the sutures, but not about
the 3--manifold, and the complement, which knows about the 3--manifold, but not about the sutures.

For theorem~\ref{thm:intro_decomposition}, we can localize near the decomposing surface $S$.
An easy calculation shows that an equivalent formula holds for a neighborhood of $S$. The local
formula then implies the global one.

\subsection{Further questions}

There is currently work in progress \cite{Zar:gluing} to provide a general gluing formula
for sutured manifolds, generalizing theorem~\ref{thm:intro_decomposition}.
There is strong evidence that the theory is closely related to contact topology, and
to the gluing maps defined by Honda, Kazez and Mati\'c
in \cite{HKM:EH}, and the contact category and TQFT defined in
\cite{HKM:TQFT}. There are also interesting parallels to the work on chord diagrams by Mathews in
\cite{Mat:chord_diags}.

Another area of interest is a generalization, to allow a wider class of arc diagrams and Heegaard diagrams.
Currently we consider diagrams coming from splitting a sutured manifold along a surface containing
index--1 critical points, but no index--2 critical points. This corresponds to cutting a sutured Heegaard
diagram along arcs which intersect some $\alpha$ circles, but no $\beta$ circles. It would be of interest
to consider what happens if allow index--2 critical points, or equivalently cutting a diagram along an
arc that intersects some $\beta$ circles.

Extending the bordered sutured theory to this setting would require arc diagrams that contain two types of arcs,
$+$ and $-$, such that $S_+$ is the result of surgery on the $+$ arcs in $\Zmid$, while $-$ is the result
of surgery on $-$ arcs. There seem to be two distinct levels of
generalization. If we allow some components of $\Z$ to have only $+$ arcs, and the rest of the components to have
only $-$
arcs, the theory should be mostly unchanged. If however we allow the same components to have a mix of $+$ and $-$ arcs,
a qualitatively different approach is required.

\subsection*{Organization}

The first few sections are devoted to the topological constructions.
First, in section~\ref{sec:algebra} we define arc diagrams, and how they parametrize sutured surfaces,
as well as the $\Ainf$--algebra associated to an arc diagram.
In section~\ref{sec:sutured} we define bordered sutured manifolds, and in section~\ref{sec:diagrams}
we define the Heegaard diagrams associated to them.

The next few sections define the invariants and give their properties.
In section~\ref{sec:moduli} we talk about the moduli spaces of curves necessary for the definitions
of the invariants. In section~\ref{sec:invariants} we give the definitions of the bordered
sutured invariants $\CSFDD$ and $\CSFA$, and prove
Eq.~(\ref{eq:pairing_bsd_bsa}) from theorem~\ref{thm:intro_pairing}. In section~\ref{sec:bimodules} we
extend the definitions and properties to the bimodules $\BSDAM$, and sketch the proof of the rest of
theorem~\ref{thm:intro_pairing}, as well as theorem~\ref{thm:equivalence}.
The gradings are defined together for all three invariants on the diagram level in section~\ref{sec:grading}.

A lot of the material in these sections is a reiteration of
analogous constructions and definitions from \cite{LOT:pairing}, with the differences emphasized.
The reader who is encountering bordered Floer homology for the first time can skip most of that discussion
on the first reading, and use theorems~\ref{thm:nice_d},~\ref{thm:nice_a} and~\ref{thm:nice_da} as definitions. 

Section~\ref{sec:examples} gives some examples of bordered sutured manifolds and computations
of their invariants. The reader is encouraged to read this section first, or immediately after
section~\ref{sec:diagrams}. The examples can be more enlightening than the definitions,
which are rather involved.

Finally, section~\ref{sec:applications} gives several applications of the new invariants, in
particular proving theorems~\ref{thm:intro_bordered_to_sutured} and~\ref{thm:intro_decomposition}.

\subsection*{Acknowledgements} 
I am very grateful to my advisor Peter Ozsv\'ath for his suggestion to look into this problem, and
for his helpful suggestions and ideas.
I would also like to thank Robert Lipshitz and Dylan Thurston for many useful discussions,
and their comments on earlier drafts of this work.

\section{The algebra associated to a parametrized surface}
\label{sec:algebra}
The invariants defined by Lipshitz, Ozsv\'ath and Thurston in
\cite{LOT:pairing} work only for connected manifolds with one closed boundary component.
There is work in progress \cite{LOT:bimodules} to generalize this to
manifolds with two or more closed boundary components.

In our construction we parametrize surfaces with boundary, and possibly many
connected components. This class of surfaces and of their allowed parametrizations is
much wider, so we need to expand the algebraic constructions describing them. We discuss below the
generalized definitions and discuss the differences from the purely bordered setting.

\subsection{Arc diagrams and sutured surfaces}
We start by generalizing the definition of a pointed matched circle in
\cite{LOT:pairing}.

\begin{defn}
An \emph{arc diagram} $\Z=(\Zmid,\amid,M)$ is a triple consisting
of a collection $\Zmid=\{Z_1,\ldots,Z_l\}$ of oriented line segments, a
collection $\amid=\{a_1,\ldots,a_{2k}\}$ of distinct points in
$\Zmid$, and a matching of $\amid$, i.e. a 2--to--1 function
$M\co \amid\to\{1,\ldots,k\}$. Write $|Z_i|$ for $\#(Z_i\cap\amid)$.
We will assume $\amid$ is ordered by the order on $\Zmid$ and the
orientations of the individual segments.
We allow $l$ or $k$ to be 0.

We impose the following condition, called \emph{non-degeneracy}.
After 
performing oriented surgery on the 1--manifold $\Zmid$ at each 0--sphere
$M^{-1}(i)$, the resulting 1--manifold should have no closed components. 
\end{defn}

\begin{defn}
We can sometimes consider \emph{degenerate arc diagrams} which do not satisfy the
non-degeneracy condition. However, we will tacitly assume all arc diagrams are non-degenerate,
unless we specifically say otherwise.
\end{defn}

\begin{rmk}
The pointed matched circles of Lipshitz, Ozsv\'ath and Thurston
correspond to arc diagrams where $\Zmid$ has
only one component. The arc diagram is obtained by cutting the matched circle at
the basepoint.
\end{rmk}

We can interpret $\Z$ as an upside-down handlebody diagram for a sutured surface
$F(\Z)$, or just $F$.
It will often be convenient to think of $F$ as a surface with corners, and will
use these descriptions interchangeably.

To construct $F$ we start with a collection of rectangles $Z_i\times[0,1]$
for $i=1,\ldots,l$. Then attach 1--handles at $M^{-1}(i)\times\{0\}$ for
$i=1,\ldots,k$. Thus $\chi(F)=l-k$, and $F$ has no closed components.
Set $\Lambda=\del\Zmid\times\{1/2\}$, and $S_+=\Zmid\times\{1\}\cup\del\Zmid\times[1/2,1]$.
Such a description uniquely specifies $F$ up to isotopy fixing the boundary.

\begin{rmk}
The non-degeneracy condition on $\Z$, is equivalent to the condition that any
component of $\del F$ intersects $\Lambda$. Indeed, the effect
on the boundary of adding the 1--handles is surgery on 
$\Zmid\times\{0\}$. If $\Z$ is non-degenerate, this surgery produces no new closed
components, and $F$ is indeed a sutured surface.
\end{rmk}

Alternatively, instead of a handle decomposition we can consider a Morse
function on $F$. Whenever we talk about Morse functions, a 
(fixed) choice
of Riemannian metric is implicit.
\begin{defn}
A \emph{$\Z$--compatible Morse function} on $F$ is a self-indexing
Morse function $f\co F\to[-1,4]$, such that the following conditions hold.
There are no index--0 or index--2 critical points. There are exactly $k$
index--1 critical points and they are all interior. The gradient of $f$
is tangent to $\del F\setminus f^{-1}(\{-1,4\})$.
The preimage $f^{-1}([-1,-1/2])$ is isotopic to a collection
of rectangles $[0,1]\times[-1,-1/2]$ such that $f$ is projection on the
second factor.
Similarly, $f^{-1}([3/2,4])$ is isotopic to a collection of rectangles
$[0,1]\times[3/2,4]$ such that $f$ is projection on the second factor.

Furthermore, we can identify $f^{-1}(\{3/2\})$ with $\Zmid$ such that
the unstable manifolds of the $i$--th index--1 critical point intersect
$\Zmid$ at $M^{-1}(i)$. We require that the orientation of $\Zmid$
and $\nabla f$ form a positive basis everywhere.
\end{defn}

Clearly, a compatible Morse function and a handle decomposition as above are
equivalent. Examples of an arc diagram, and the different ways we can interpret
its parametrization of a sutured surface, are given in Fig.~\ref{fig:annulus}.
A slightly more complicated example of an arc diagram, corresponding to the
parametrization in Fig.~\ref{subfig:sutured_surface_param}, is given in
Fig.~\ref{fig:arc_diagram_big}.

\begin{figure}
\begin{subfigure}[t]{.48\linewidth}
	\centering
	\includegraphics[scale=.83]{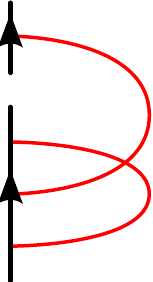}
	\caption{Arc diagram $\Z$ for an annulus.}
\end{subfigure}
\begin{subfigure}[t]{.48\linewidth}
	\centering
	\includegraphics[scale=.65]{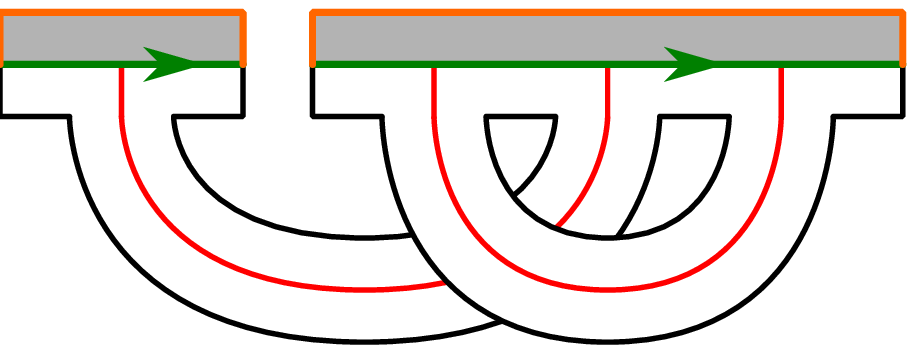}
	\caption{Handle decomposition associated with $\Z$.}
\end{subfigure}
\begin{subfigure}[t]{.48\linewidth}
	\centering
	\includegraphics[scale=.625]{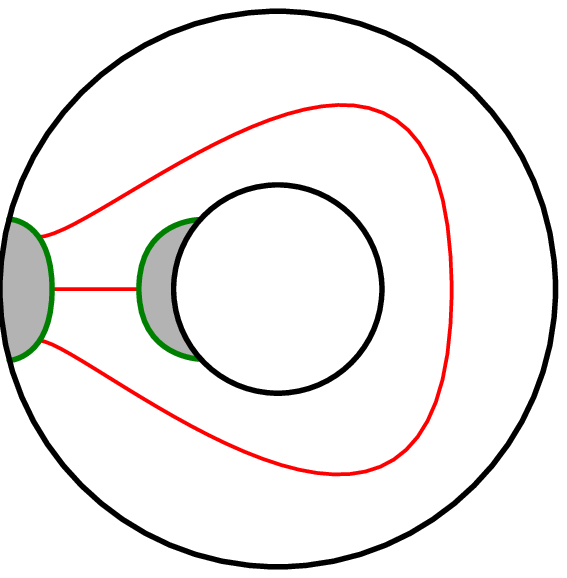}
	\caption{Annulus parametrized by $\Z$. }
\end{subfigure}
\begin{subfigure}[t]{.48\linewidth}
	\centering
	\labellist
	\small\hair 2pt
	\pinlabel $4$ [r] at 40 200
	\pinlabel $3/2$ [r] at 40 140
	\pinlabel $1$ [r] at 40 110
	\pinlabel $-1$ [r] at 40 40
	\endlabellist
	\includegraphics[scale=.5]{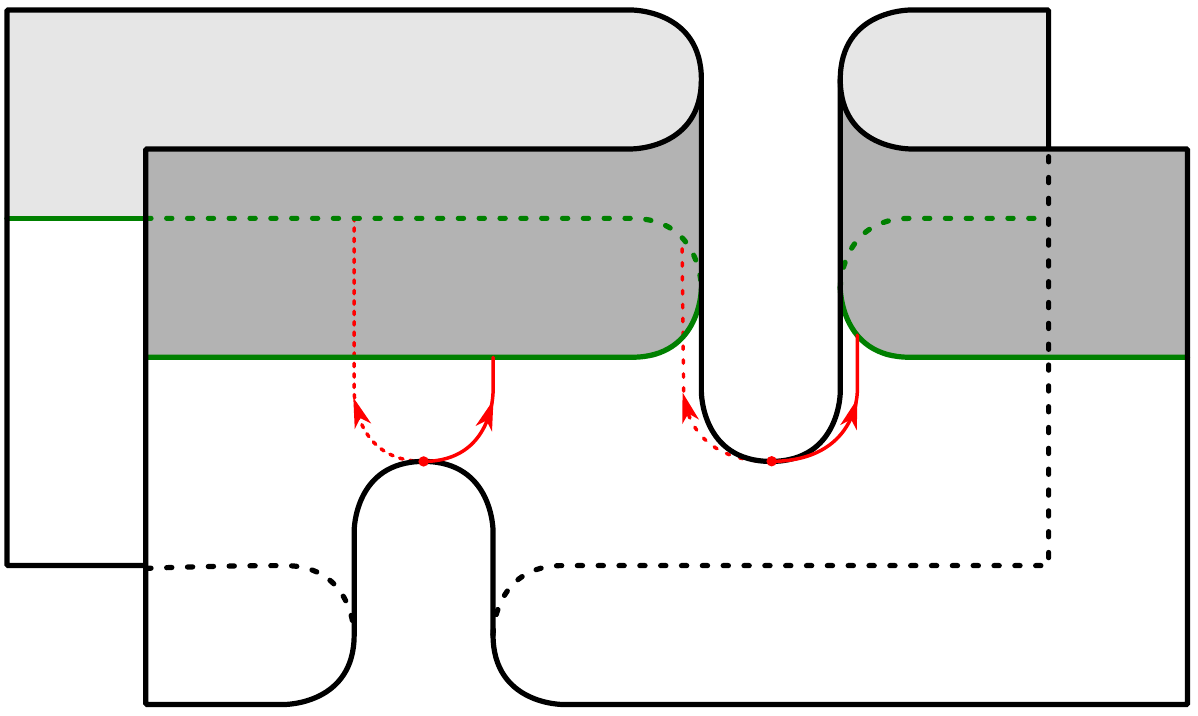}
	\caption{Morse function compatible with $\Z$.}
\end{subfigure}
\caption{Arc diagram for an annulus, and three different views of parametrization.}
\label{fig:annulus}
\end{figure}

There is one more way to describe the above parametrization. Recall that a
\emph{ribbon graph} is a graph with a cyclic ordering of the edges incident to
any vertex. An embedding of a ribbon graph into a surface will be considered
\emph{orientation preserving} if the ordering of the edges agrees with the
positive direction on the unit tangent circle of the vertex in the surface.

\begin{defn}
Let $F$ be a sutured surface obtained from an arc diagram $\Z$ as above.
The \emph{ribbon graph associated to $\Z$} is the ribbon graph $G(\Z)$
with vertices $\del\Zmid\cup\amid$, and edges the components of
$\Zmid\setminus\amid$ and the cores of the 1--handles,
which we denote $e_i$ for
$i=1,\ldots,k$. The cyclic ordering is induced from the orientation of $F$.
\end{defn}

In these terms, $F$ is parametrized by $\Z$ if we specify an orientation
preserving proper embedding $G(\Z)\hookrightarrow F$, such that $F$ deformation
retracts onto the image.

\begin{rmk}When we draw an arc diagram $\Z$ we are in fact drawing its graph $G(\Z)$. An example,
with all elements of the graph denoted, is given in Fig.~\ref{fig:arc_diagram_big}.
\end{rmk}

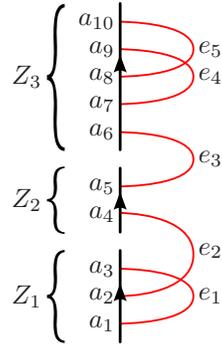
\begin{figure}
	\labellist
	\small\hair 2pt
	\pinlabel $a_1$ [r] at 0 10
	\pinlabel $a_2$ [r] at 0 25
	\pinlabel $a_3$ [r] at 0 40
	\pinlabel $a_4$ [r] at 0 70
	\pinlabel $a_5$ [r] at 0 85
	\pinlabel $a_6$ [r] at 0 115
	\pinlabel $a_7$ [r] at 0 130
	\pinlabel $a_8$ [r] at 0 145
	\pinlabel $a_9$ [r] at 0 160
	\pinlabel $a_{10}$ [r] at 0 175
	\pinlabel $Z_1$ [r] at -40 25
	\pinlabel $Z_2$ [r] at -40 77.5
	\pinlabel $Z_3$ [r] at -40 145
	\pinlabel $e_1$ [l] at 40 25
	\pinlabel $e_2$ [l] at 40 50
	\pinlabel $e_3$ [l] at 40 100
	\pinlabel $e_4$ [l] at 40 145
	\pinlabel $e_5$ [l] at 40 160
	\endlabellist
	\includegraphics[scale=.7]{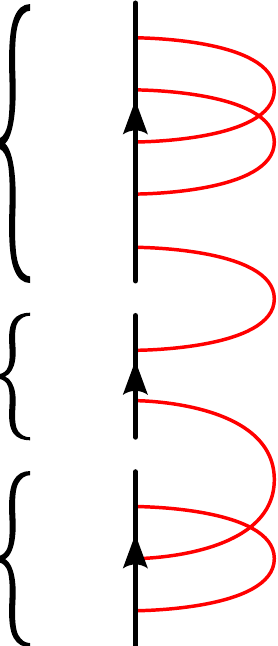}
\caption{An arc diagram $\Z$ for a twice punctured torus, and its graph $G(\Z)$.}
\label{fig:arc_diagram_big}
\end{figure}

\subsection{The algebra associated to an arc diagram}

Recall the definition of the strands algebra from \cite{LOT:pairing}.

\begin{defn}
The \emph{strands algebra} $\A(n,k)$ is a free $\ZZ/2$--module with generators
of the form $\mu=(T,S,\phi)$,
where $S$ and $T$ are $k$--element subsets of $\{1,\ldots,n\}$, and 
$\phi\co S\to T$ is a non--decreasing bijection.
(We think of $\phi$ as a collection of \emph{strands} from $S$ to $T$.)
Denote by $\inv(\mu)=\inv(\phi)$ the number of inversions of $\phi$, i.e.
the elements of $\Inv(\mu)=\{(i,j):i,j\in S, i<j, \phi(i)>\phi(j) \}$.

Multiplication is given by
\begin{equation*}
(S,T,\phi)\cdot(U,V,\psi)=\begin{cases}
(S,V,\psi\circ\phi) & \textrm{$T=U$, $\inv(\phi)+\inv(\psi)=\inv(\psi\circ\phi)$,}\\
0 & \textrm{otherwise.}
\end{cases}
\end{equation*}
The differential on $(S,T,\phi)$ is given by the sum of all possible ways to
``resolve'' an inversion,
i.e. switch $\phi(i)$ and $\phi(j)$ for some inversion $(i,j)\in\Inv(\mu)$.
\end{defn}

Next, we consider the larger \emph{extended strands algebra}
\begin{equation*}
A(n_1,\ldots,n_l;k)= \bigoplus_{k_1+\cdots+k_l=k}
\A(n_1,k_1)\otimes\cdots\otimes\A(n_l,k_l).
\end{equation*}

We will slightly abuse notation and think of elements of $\A(n_i,k_i)$ as
functions acting on subsets of
$\{(n_1+\cdots+n_{i-1})+1,\ldots,(n_1+\cdots+n_{i-1})+n_i\}$ instead of
$\{1,\ldots,n_i\}$.
This allows us to identify $\A(n_1,\ldots,n_l;k)$ with a subalgebra of
$\A(n_1+\cdots+n_l,k)$.

We will sometimes talk about the sums
$\A(n)=\A(n,0)\oplus\cdots\oplus\A(n,n)$, and
$\A(n_1,\ldots,n_l)=\A(n_1,\ldots,n_l;0)\oplus\cdots\oplus\A(n_1,\ldots,n_l;n_1+\cdots+n_l).$

The definition of $\A(\Z,i)$ as a subalgebra of $\A(|Z_1|,\ldots,|Z_l|;i)$
below is a straightforward generalization of the definition of the
algebra associated to a pointed matched circle in \cite{LOT:pairing}. There is, however,
a difference in notation. In \cite{LOT:pairing} $\A(\Z,0)$ denotes the middle summand and
negative summand indeices are allowed. Here, $\A(\Z,0)$ is the bottom summand, and
we only allow non-negative indices.

For any $i$--element subset $S\subset\{1,\ldots,2k\}$, there is an
idempotent $I(S)=(S,S,\id_S)\in\A(|Z_1|,\ldots,|Z_l|,i)$.
For an $i$--element subset $s\subset\{1,\ldots,k\}$, a \emph{section}
$S$ of $s$ is an $i$--element set 
$S\subset M^{-1}(s)$, such that $M|_S$ is
injective. To each $s$ there is an associated idempotent
\begin{equation*}
I(s)=\sum_{S~\textrm{is a section of}~s}I(S).
\end{equation*}

Consider triples of the form $(S,T,\psi)$, where $S,T\subset\{1,\ldots,2k\}$,
$\psi\co S\to T$ is a strictly increasing bijection. Consider all
possible sets $U\subset\{1,\ldots,2k\}$ disjoint from $S$ and $T$, and such that
$S\cup U$ has $i$ elements. Let
\begin{equation*}
a_i(S,T,\psi)=\sum_{U~\textrm{as above}}
(S\cup U,T\cup U,\psi_U)\in \A(|Z_1|,\ldots,|Z_l|;i),
\end{equation*}
where $\psi_U|_T=\psi$, and $\psi_U|_U=\id_U$. In the language of strands,
this means ``to a set of moving strands add all possible consistent
collections of
stationary (or horizontal) strands''.

Let $\I(\Z,i)$ be the subalgebra generated by $I(s)$ for all $i$--element sets
$s$, and let $I=\sum_{s}I(s)$ be their sum.
Let $\A(\Z,i)$ be the subalgebra generated by $\I(\Z,i)$ and all elements of the form
$I\cdot a_i(S,T,\psi) \cdot I$.

All elements $(S,T,\phi)$ considered have the property that $M|_S$ and $M|_T$ are injective.

\begin{defn}
The \emph{algebra associated} with the arc diagram $\Z$ is
\begin{equation*}
\A(\Z)=\bigoplus_{i=0}^k\A(\Z,i),
\end{equation*}
which is a module over
\begin{equation*}
\I(\Z)=\bigoplus_{i=0}^k\I(\Z,i).
\end{equation*}
\end{defn}

To any element of $\mu=(S,T,\phi)\in\A(|Z_1|,\ldots,|Z_l|)$ we can associate its \emph{homology class}
$[\mu] \in H_1(\Zmid,\amid)$, by setting
\begin{equation*}
[\mu]=\sum_{i\in S}[l_i],
\end{equation*}
where $l_i$ is the positively oriented segment $[a_i,a_{\phi(i)}]\subset\Zmid$. It is additive
under multiplication and preserved by the differential on $\A(|Z_1|,\ldots,|Z_l|)$.
Since it only depends on the moving strands of $(S,T,\phi)$, any element of $\A(\Z)$ is homogeneous
with respect to this homology class, and therefore we can talk about the homology class of an
element in $\A(\Z)$.

\begin{rmk}
With a collection $\Z_1, \ldots, \Z_p$ of arc diagrams we can associate their
\emph{union} $\Z=\Z_1\cup\cdots\cup\Z_p$, where
$\Zmid=\Zmid_1\sqcup\cdots\sqcup\Zmid_p$, preserving the
matching on each piece.

There are natural identifications, of algebras
\begin{equation*}
\A(\Z)=\bigotimes_{i=1}^{p}\A(\Z_i),
\end{equation*}
and of surfaces
\begin{equation*}
F(\Z)=\bigsqcup_{i=1}^{p}F(\Z_i).
\end{equation*}
\end{rmk}

\subsection{Reeb chord description}
We give an alternative interpretation of the strands algebra $\A(\Z)$.

Given an arc diagram $\Z$ with $k$ arcs, there is a unique positively oriented contact
structure
on the 1--manifold $\Zmid$, while the 0--manifold $\amid\subset\Zmid$ is
Legendrian. There is a family of \emph{Reeb chords} in $\Zmid$, starting and ending
at $\amid$ and positively oriented. For a Reeb chord $\rho$ we will denote its
starting and ending point by $\rho^-$ and $\rho^+$, respectively.
Moreover, for a collection $\brho=\{\rho_1,\ldots,\rho_n\}$ of Reeb chords as above,
we will write $\brho^-=\{\rho_1^-,\ldots,\rho_n^-\}$, and
$\brho^+=\{\rho_1^+,\ldots,\rho_n^+\}$.

\begin{defn}A collection $\brho=\{\rho_1,\ldots,\rho_n\}$ of Reeb chords is 
\emph{$p$--com\-plet\-a\-ble} if the following conditions hold:
\begin{enumerate}
\item $\rho_i^-\neq\rho_i^+$ for all $i=1,\ldots,n$.
\item $M(\rho_1^-),\ldots,M(\rho_n^-)$ are all distinct.
\item $M(\rho_1^+),\ldots,M^(\rho_n^+)$ are all distinct.
\item \label{cond:completable} $\#(M(\brho^-)\cup M(\brho^+)) \leq k-(p-n)$.
\end{enumerate}
\end{defn}

Condition~(\ref{cond:completable})
guarantees that there is at least one choice of a $(p-n)$--element set
$s\subset\{1,\ldots,k\}$, disjoint from $M(\brho^-)$ and $M(\brho^+)$. Such a set is called
\emph{a $p$--completion} or just \emph{completion} of $\brho$.
Every completion of $\brho$ defines an element of $\A(\Z,p)$:

\begin{defn}
\label{def:a_rho_s}
For a $p$--completable collection $\brho$ and a completion $s$, their
\emph{associated element} in $\A(\Z,p)$ is
\begin{equation*}
a(\brho,s)=\sum_{S~\textrm{is a section of}~s}(\brho^-\cup S,\brho^+\cup S,\phi_S),
\end{equation*}
where $\phi_S(\rho_i^-)=\rho_i^+$, for $i=1,\ldots,n$, and $\phi|_S=\id_S$.
\end{defn}

\begin{defn} The \emph{associated element} of $\brho$ in $\A(\Z,p)$ is the sum over all
$p$--completions:
\begin{equation*}
a_p(\brho)=\sum_{s~\textrm{is a}~p\textrm{-completion of}~\brho}a(\brho,s).
\end{equation*}

If $\brho$ is not $p$--completable, we will just set $a_p(\brho)=0$. We will also sometimes
use the complete sum
\begin{equation*}
a(\brho)=\sum_{p=0}^k a_p(\brho).
\end{equation*}
\end{defn}

The algebra $\A(\Z,p)$ is generated over $\I(\Z,p)$ by the elements $a_p(\brho)$ for all possible
$p$--completable $\brho$.
Algebra multiplication of such associated elements corresponds to certain concatenations of
Reeb chords.

We can define the \emph{homology} class $[\rho]\in H_1(\Zmid,\amid)$ in the obvious
way, and extend to a set of Reeb chords $\brho=\{\rho_1,\ldots,\rho_n\}$, by taking the
sum $[\brho]=[\rho_1]+\cdots+[\rho_n]$. It is easy to see that $[a(\brho,s)]=[\brho]$, and in
particular it doesn't depend on the completion $s$.

\subsection{Grading}
There are two ways to grade the algebra $\A(\Z)$. The simpler is to grade it
by a nonabelian group $\Gr(\Z)$, which is a $\frac{1}{2}\ZZ$--extension of $H_1(\Zmid,\amid)$.
This group turns out to be too big, and does not allow for a graded version of the pairing theorems.
For this a subgroup $\Grdn(\Z)$ of $\Gr(\Z)$ is necessary, that can be identified with a $\frac{1}{2}\ZZ$--extension
of $H_1(F(\Z))$. Unfortunately, there is no canonical way to get a $\Grdn(\Z)$--grading on $\A(\Z)$.

\begin{rmk} Our notation differs from that in~\cite{LOT:pairing}.
In particular, our grading group $\Gr(\Z)$ is analogous to the group $G'(\Z)$ used by Lipshitz, Ozsv\'ath and Thurston,
while $\Grdn(\Z)$ corresponds to their $G(\Z)$.
Moreover, our grading function $\gr$ corresponds to their $\gr'$, while $\grdn$ corresponds to $\gr$.
\end{rmk}

We start with the $\Gr(\Z)$--grading.
Suppose $\Zmid=\{Z_1,\ldots,Z_l\}$. We will define a grading on the bigger algebra
$\A(|Z_1|,\ldots,|Z_l|)$ that descends to a grading on
$\A(\Z)$.

First, we define some auxiliary maps.

\begin{defn}
Let $m\co H_0(\amid)\times H_1(\Zmid,\amid)\to\frac{1}{2}\ZZ$ be the map defined by counting
local multiplicities. More precisely, given the positively oriented
line segment $l=[a_i,a_{i+1}]\subset Z_p$, set
\begin{equation*} 
m([a_j],[l]) = \begin{cases}
\frac{1}{2} & \textrm{if $j=i,i+1$,}\\ 
0 & \textrm{otherwise,}
\end{cases} 
\end{equation*}
and extend linearly to all of $H_0(\amid)\times H_1(\Zmid,\amid)$.
\end{defn}

\begin{defn}
Let $L\co H_1(\Zmid,\amid)\times H_1(\Zmid,\amid)\to\frac{1}{2}\ZZ$, be
\begin{equation*}
L(\alpha_1,\alpha_2)=m(\del(\alpha_1),\alpha_2),
\end{equation*}
where $\del$ is the connecting homomorphism in homology.
\end{defn}

The group $\Gr(\Z)$ is defined as a central extension of $H_1(\Zmid,\amid)$ by $\frac{1}{2}\ZZ$ in
the following way.

\begin{defn}
Let $\Gr(\Z)$ be the set $\frac{1}{2}\ZZ\times H_1(\Zmid,\amid)$, with multiplication
\begin{equation*}
(a_1,\alpha_1)\cdot(a_2,\alpha_2)=(a_1+a_2+L(\alpha_1,\alpha_2),\alpha_1+\alpha_2).
\end{equation*}

For an element $g=(a,\alpha)\in\Gr(\Z)$ we call $a$ the \emph{Maslov component},
and $\alpha$ the \emph{homological component} of $g$.
\end{defn}

Note that if $\Zmid$ has just one component $Z_1$ and $|Z_1|=n$, then this grading group is the same
as the group $G'(n)$ defined in \cite[Section 3]{LOT:pairing}.
In general, if $\Zmid=\{Z_1,\ldots,Z_l\}$, as a set
\begin{equation*}
G'(|Z_1|)\times\cdots\times G'(|Z_l|)\cong\left(\frac{1}{2}\ZZ\right)^l\times H_1(\Zmid,\amid),
\end{equation*}
since $H_1(Z_1,\amid\cap Z_1)\oplus\cdots\oplus H_1(Z_l,\amid\cap Z_l)\cong H_1(\Zmid,\amid)$.
Adding the Maslov components together induces a surjective homomorphism
\begin{equation*}
\sigma\co G'(|Z_1|)\times\cdots\times G'(|Z_l|)\to \Gr(\Z).
\end{equation*}

We can now define the grading $\gr\co \A(|Z_1|,\ldots,|Z_l|)\to \Gr(\Z)$.

\begin{defn}
For an element $a=(S,T,\phi)$ of $\A(|Z_1|,\ldots,|Z_l|)$, set
\begin{align*}
\iota(a) & =\inv(\phi)-m(S,[a]),\\
\gr(a) & =(\iota(a),[a]).
\end{align*}
\end{defn}

Breaking up $a$ into its components $a=(a_1,\ldots,a_l)\in\A(|Z_1|)\oplus\cdots\oplus\A(|Z_l|)$,
we see that $\gr(a)=\sigma(\gr'(a_1),\ldots,\gr'(a_l))$.

Therefore, we can apply the results about $G'$ and $\gr'$ from \cite{LOT:pairing} to deduce the following
proposition.

\begin{prop}
The function $\gr$ is indeed a grading on $\A(|Z_1|,\ldots,|Z_l|)$, with the same properties as $G'$ on $\A(n)$.
Namely, the following statements hold.
\begin{enumerate}
\item 
\label{cond:algebra_grading}
Under $\gr$, $\A(|Z_1|,\ldots,|Z_l|)$ is a differential graded algebra, where the differential drops
the grading by the central element $\lambda=(1,0)$.
\item For any completable collection of Reeb chords $\brho$, the element $a(\brho)$ is homogeneous.
\item The grading $\gr$ descends to $\A(\Z)$.
\item For any completable collection $\brho$, the grading of $a(\brho,s)$ does not depend on the completion $s$.
\end{enumerate}
\end{prop}

\begin{proof}
The proof of~(\ref{cond:algebra_grading}) follows from the corresponding statement for $\gr'$, after noticing that
the differential on $\A(|Z_1|,\ldots,|Z_l|)$ is defined via the Leibniz rule, and the differentials on the
individual components drop one Maslov component by $1$, while keeping all the rest fixed.

The rest of the statements then follow analogously to those for $\gr'$ in \cite{LOT:pairing}.
\end{proof}

\subsection{Reduced grading}
We can now define the refined grading group $\Grdn(\Z)$. Recall that the surface $\F(\Z)$ retracts to the graph
$G(\Z)$, consisting of the segments $\Zmid$, and the arcs $E=\{e_1,\ldots,e_k\}$, such that $\Zmid\cap E=\amid$.
From the long exact sequence for the pair $(G,E)$ we know that the following piece is exact.
\begin{equation*}
0\to H_1(G) \to H_1(G,E)\to H_0(E)
\end{equation*}

The differential $\del\co H_1(G,E)\to H_0(E)$ can be identified with the composition
$M_*\circ\del\co H_1(\Zmid,\amid)\to H_0(E)$,
and $H_1(F)=H_1(G)$ can be identified with $\ker \del \subset H_1(\Zmid,\amid)$. The identification can also be seen
by adding the arcs $e_i$ to cycles in $(\Zmid,\amid)$ to obtain cycles in $G=\Zmid\cup E$.
This is induces a map $\del'\co \Gr(\Z)\to H_0(E)$, and the kernel $\Grdn(\Z)=\ker \del'$ is just the subgroup
of $\Gr(\Z)$, consisting of elements with homological component in $\ker \del\cong H_1(F)$.

\begin{prop}Under the identification $\ker \del=H_1(F)$, the group $\Grdn(\Z)$ can be explicitly described as
a central extension of $H_1(F)$ by $\frac{1}{2}\ZZ$, with multiplication law
\begin{equation*}
(a_1,[\alpha_1])\cdot(a_2,[\alpha_2])=(a_1+a_2+ \#(\alpha_1\cap\alpha_2),[\alpha_1]+[\alpha_2]),
\end{equation*}
where $a_1,a_2\in\frac{1}{2}\ZZ$, and $\alpha_1$ and $\alpha_2$ are curves in $F$, and
$\#(\alpha_1\cap\alpha_2)$ is the signed intersection number, according to the orientation of $F$.
\end{prop}
\begin{proof}
First, notice that the intersection pairing is well-defined, as it is, via Poinc\'are duality, just the
pairing $\left<\cdot\cup\cdot,[F,\del F]\right>$ on $H^1(F,\del F)$.
The remaining step is to show that under the identification $\ker \del=H_1(F)$, this agrees with the pairing
$L$ on $H_1(\Zmid,\amid)$.
This can be seen by starting with line segments on $\Zmid$ and arcs in $E$, pushing the arcs on $E$ away from each other in
the $2^{\# E}$ possible ways.
One can then count that $\pm 1$ contributions to $L$ always give rise to an intersection
point, while $\pm 1/2$ contributions create an intersection point exactly half of the time.
\end{proof}

In fact, for any generator $a\in\A(\Z)$ with starting and ending idempotents $I_s$ and $I_e$, respectively, 
$\del'(\gr(a))=I_e-I_s$, if we think of the idempotents as linear combinations of the $e_i$. Therefore,
for any $a$ with $I_e=I_s$, $\gr(a)$ is already in $\Grdn$, and in general it is ``almost'' in there.
At this point we would like to find a retraction $\Gr\to\Grdn$ and use this to define the refined grading.
However this fails even in simple cases.
For instance, when $\Z$ is an arc diagram for a disc with several sutures, $\Grdn(\Z)=\frac{1}{2}\ZZ$ is abelian,
as $H_1(F)$ vanishes, while the commutator of $\Gr(\Z)$ is $\ZZ\subset\Grdn(\Z)$, and there can be no retraction,
even if we pass to $\QQ$--coefficients.

The solution is to assign a grading to $\A(\Z)$ with values in $\Grdn(\Z)$, depending on the starting and ending
idempotents. First, note that the generating idempotents come in sets of \emph{connected components}, where $I$
is connected to $J$ if and only if $I-J$ is in the image of $\del'$, or equivalently in the kernel of $H_0(E)\to H_0(F)$.
These connected components correspond to the possible choices of how many arcs are occupied in each connected component
of $F(\Z)$.

\begin{defn}
A \emph{grading reduction} $r$ for $\Z$ is a choice of a \emph{base idempotent} $I_0$ in each connected component,
and a choice $r(I)\in\del'^{-1}(I-I_0)$ for any $I\in[I_0]$.
\end{defn}

\begin{defn}
Given a grading reduction $r$, define the \emph{reduced grading}
\begin{equation*}
\grdn_r(a)=r(I_s)\cdot\gr(a)\cdot r(I_e)^{-1}\in\Grdn(\Z),
\end{equation*}
for any generator $a\in\A(\Z)$ with starting and ending idempotents $I_s$ and $I_e$, respectively. When
unambiguous, we write simply $\grdn(a)$.
\end{defn}

For any elements $a$ and $b$, such that $a\cdot b$, or even $a\otimes b$ is nonzero, the $r$--terms in $\grdn$
cancel, and $\grdn(a \otimes b)=\grdn(a \cdot b)=\grdn(a)\cdot\grdn(b)$. Since $\left<\frac{1}{2}\ZZ,0\right>$ is in the
center, there is still a well-defined $\ZZ$--action by $\lambda=\left<1,0\right>$, and
$\grdn(\del a)=\lambda^{-1}\grdn(a)$. Therefore, $\grdn$ is indeed a grading.

Notice that for any $a$ with $I_s=I_e$, $\grdn(a)$ is the conjugate of $\gr(a)\in\Grdn$ by $r(I_s)$. In particular,
the homological part of the grading is unchanged, and whenever it vanishes, the Maslov component is also unchanged.

\begin{rmk} Given a set of Reeb chords $\brho$, the element $a(\brho)\in\A(\Z)$ is no longer homogeneous under $\grdn$.
Indeed, $\grdn(a(\brho,s))$ depends on the completion $s$. 
\end{rmk}

\subsection{Orientation reversals}
\label{sec:reversal}
It is sometimes useful to compare the arc diagrams $\Z$ and $-\Z$ and the corresponding gradings. Recall that
$-\Z$ and $\Z$ differ only by the orientation of $\Zmid$. Consequently, the homology components $H_1(\pm\Zmid,\amid)$ in
$\Gr(\pm\Z)$ can be identified, while their canonical bases are opposite in order and sign. In particular, the pairings
$L_{\pm\Z}$ are opposite from each other. Therefore $\Gr(\Z)$ and $\Gr(-\Z)$ are anti-isomorphic, via the map fixing
both the Maslov and homological components.

Similarly, $F(\pm\Z)$ differ only in orientation, the homological
components $H_1(F)$ can be naturally identified while the intersection pairings
are opposite from each other. Thus $\Grdn(\Z)$ and $\Grdn(-\Z)$ are also anti-homomorphic,
via the map that fixes both components, which agrees
with the restriction of the corresponding map on $\Gr(\Z)$.

Thus, left actions by $\Gr(\Z)$ or $\Grdn(\Z)$ naturally correspond to right actions by $\Gr(-\Z)$ or $\Grdn(-\Z)$,
respectively, and vice versa.

\section{Bordered sutured \texorpdfstring{3--manifolds}{3-manifolds}}
\label{sec:sutured}

In this section---and for most of the rest of the paper---we will be working from the point
of view of bordered sutured manifolds, as sutured manifolds with extra structure. We will
largely avoid the alternative description of decorated sutured cobordisms.

\subsection{Sutured manifolds}

\begin{defn}
A \emph{divided surface} $(S,\Gamma)$ is a closed surface $F$, together with
a collection $\Gamma=\{\gamma_1,\ldots,\gamma_n\}$ of pairwise disjoint
oriented simple closed curves on $F$, called \emph{sutures}, satisfying the
following conditions.

Every component $B$ of $F \setminus\Gamma$ has nonempty boundary
(which is the union
of sutures). Moreover, the boundary orientation and the suture orientation of
$\del B$ either agree on all components, in which case we call $B$ a
\emph{positive region},
or they disagree on all components, in which case we call $B$ a
\emph{negative region}.
We denote by $R_+(\Gamma)$ or $R_+$ (respectively $R_-(\Gamma)$ or $R_-$)
the closure of the union of all positive (negative) regions.
\end{defn}

Notice that the definition doesn't require $F$ to be connected, but it requires
that each component contain a suture.

\begin{defn}
A divided surface $(F,\Gamma)$ is called \emph{balanced} if
$\chi(R_+)=\chi(R_-)$.

It is called \emph{$k$--unbalanced} if $\chi(R_+)=\chi(R_-)+2k$, where $k$
could be positive, negative or $0$. In particular $0$--unbalanced is the same
as balanced.
\end{defn}

Notice that since $F$ is closed, and $\chi(S)=\chi(R_+)+\chi(R_-)$, it follows
that $\chi(R_+)-\chi(R_-)$ is always even.

Now we can express the balanced sutured manifolds of \cite{Juh:SFH} in terms
of divided surfaces.

\begin{defn}
A \emph{balanced sutured manifold} $(Y,\Gamma)$ is a 3--manifold $Y$ with no
closed components, such that $(\del Y, \Gamma)$ is a balanced divided surface.
\end{defn}

We can extend this definition to the following.
\begin{defn}
A \emph{$k$--unbalanced sutured manifold} $(Y,\Gamma)$ is a 3--manifold $Y$ with
no closed components, such that $(\del Y,\Gamma)$ is a $k$--unbalanced
divided surface.
\end{defn}

Although our unbalanced sutured manifolds are more general than the balanced
ones of Juh\'asz, they are still strictly a subclass of Gabai's general
definition in \cite{Gab:foliations}. For example, he allows toric sutures,
while we do not.

\subsection{Bordered sutured manifolds}
In this section we describe how to obtain a bordered sutured manifold
from a sutured manifold, by parametrizing part of its boundary.

\begin{defn}
A \emph{bordered sutured manifold} $(Y,\Gamma,\Z,\phi)$ consists of the
following.
\begin{enumerate}
\item A sutured manifold $(Y,\Gamma)$.
\item An arc diagram $\Z$.
\item An
orientation preserving embedding $\phi\co G(\Z)\hookrightarrow\del Y$, such that
$\phi|_{\Zmid}$ is an orientation preserving embedding into $\Gamma$, 
and $\phi(G(\Z)\setminus\Zmid)\cap\Gamma=\varnothing$. It follows that each arc $e_i$ embeds in $R_-$.
\end{enumerate}
\end{defn}

Note that a closed neighborhood $\overline{\nu(G(\Z))}\subset\del Y$
can be identified with the parametrized surface $F(\Z)$. We will make this identification
from now on.

An equivalent way to give a bordered sutured manifold
would be to specify an embedding 
$F(\Z)\hookrightarrow\del Y$,
such that the following conditions hold.
Each 0--handle of $F$ intersects $\Gamma$ in a single arc,
while each 1--handle is
embedded in $\Int(R_-(\Gamma))$.

\begin{prop}
Any bordered sutured manifold $(Y,\Gamma,\Z,\phi)$ satisfies the following condition,
called \emph{homological linear independence}.
\begin{equation}
\label{eq:hli}
\textrm{$\pi_0(\Gamma\setminus\phi(\Zmid))\to\pi_0(\del Y\setminus F(\Z))$ is surjective.}
\end{equation}
\end{prop}
\begin{proof}
Indeed, Eq.~(\ref{eq:hli}) is equivalent to $\Gamma$ intersecting any component of 
$\del Y\setminus F$. But $\Gamma$ already intersects any component of
$\del Y$. Any component of $\del Y\setminus F$ is either a component of
$\del Y$, or has common boundary with $F$. The non-degeneracy condition on $\Z$
guarantees that any component of $\del F$ hits $\Gamma$.
\end{proof}

\begin{rmk}
If we want to work with degenerate arc diagrams (which give rise to degenerate sutured surfaces)
we can still get well-defined invariants, as long as we impose homological linear independence on
the manifolds. However, in that case there is no category, since the identity cobordism from
a degenerate sutured surface to itself \emph{does not} satisfy homological linear independence.
\end{rmk}

\subsection{Gluing}
\label{sec:glue_mflds}
We can glue two bordered sutured manifolds to obtain a sutured manifold in the
following way.

Let $(Y_1,\Gamma_1,\Z,\phi_1)$, and $(Y_2,\Gamma_2,-\Z,\phi_2)$ be two
bordered sutured manifolds. Since $\phi_1$ and $\phi_2$ are embeddings, and
$G(-\Z)$ is naturally isomorphic to $G(\Z)$ with its orientation reversed,
there is a diffeomorphism $\phi_1(G(\Z))\to\phi_2(G(-\Z))$ that can be
extended to an orientation reversing diffeomorphism
$\psi\co F(\Z)\to F(-\Z)$ of their neighborhoods. Moreover,
$\psi|_{\Gamma_1}\co \Gamma_1\cap F(\Z)\to\Gamma_2\cap F(-\Z)$ is orientation
reversing.

Set $Y=Y_1\cup_{\psi}Y_2$, and
$\Gamma=(\Gamma_1\setminus F(\Z))\cup(\Gamma_2\setminus F(-\Z))$.
By homological linear independence on $Y_1$ and $Y_2$, the sutures $\Gamma$
on $Y$ intersect all components of $\del Y$, and $(Y,\Gamma)$ is a sutured manifold.

~

More generally, we can do partial gluing.
Suppose $(Y_1,\Gamma_1,\Z_0\cup\Z_1,\phi_1)$ and
$(Y_2,\Gamma_2,-\Z_0\cup\Z_2,\phi_2)$ are bordered sutured. Then
\begin{equation*}
(Y_1\cup_{F(\Z_0)}Y_2,
(\Gamma_1\setminus F(\Z_0))\cup(\Gamma_2\setminus F(-\Z_0)),
\Z_1\cup\Z_2,
\phi_1|_{G(\Z_1)}\cup\phi_2|_{G(\Z_2)})
\end{equation*}
is also bordered sutured.

\section{Heegaard diagrams}
\label{sec:diagrams}

\subsection{Diagrams and compatibility with manifolds}
\begin{defn}
A \emph{bordered sutured Heegaard diagram} $\HH=(\Sigma,\balpha,\bbeta,\Z,\psi)$
consists of the following data:
\begin{enumerate}
\item A surface with with boundary $\Sigma$.
\item An arc diagram $\Z$.
\item An orientation reversing embedding
$\psi\co G(\Z)\hookrightarrow\Sigma$, such that $\psi|_{\Zmid}$ is an 
orientation preserving embedding into $\del\Sigma$, while
$\psi|_{G(\Z)\setminus\Zmid}$ is an embedding into $\sigint$.
\item The collection $\balpha^a=\{\alpha_1^a,\ldots,\alpha_k^a\}$ of arcs
$\alpha_i^a=\psi(e_i)$.
\item A collection of simple closed curves
$\balpha^c=\{\alpha_1^c,\ldots,\alpha_n^c\}$ in $\sigint$, which are
disjoint from each other and from $\balpha^a$.
\item A collection of simple closed
curves $\bbeta=\{\beta_1,\ldots,\beta_m\}$ in $\sigint$, which are
pairwise disjoint and transverse to $\balpha=\balpha^a\cup\balpha^c$.
\end{enumerate}

We also require that
$\pi_0(\del\Sigma\setminus\Zmid)\to\pi_0(\Sigma\setminus\balpha)$
and $\pi_0(\del\Sigma\setminus\Zmid)\to\pi_0(\Sigma\setminus\bbeta)$ be surjective.
We call this condition \emph{homological linear independence} since it is equivalent
to each of $\balpha$ and $\bbeta$ being linearly independent in $H_1(\Sigma,\Zmid)$.
\end{defn}

Homological linear independence on diagrams is the key condition required for
admissibility and avoiding boundary degenerations.

\begin{defn}
A \emph{boundary compatible Morse function} on a bordered sutured manifold
$(Y,\Gamma,\Z,\phi)$ is a self-indexing Morse function $f\co Y\to[-1,4]$ (with an
implicit choice of Riemannian metric $g$) with the following properties.
\begin{enumerate}
\item The parametrized surface $F(\Z)=\overline{\nu(G(\Z))}$ is totally geodesic,
$\nabla f$ is parallel to $F$, and $f|_F$ is a $\Z$--compatible Morse function.
\item A closed neighborhood $N=\overline{\nu(\Gamma\setminus\Z)}$ is isotopic to
$(\Gamma\setminus\Z)\times[-1,4]$, such that $f$ is projection on the second
factor (and $f(\Gamma)=3/2$).
\item 
$f^{-1}(-1)=\overline{R_-(\Gamma)\setminus(N\cup F)}$, and
$f^{-1}(4)=\overline{R_+(\Gamma)\setminus(N\cup F)}$.
\item $f$ has no index--0 or index--3 critical points.
\item The are no critical
points in $\del Y\setminus F$, and the index--1 critical points for $F$ are
also index--1 critical points for $Y$.
\end{enumerate}
\end{defn}

See Fig.~\ref{subfig:morse_true} for a schematic illustration.

From a boundary compatible Morse function $f$ we can get a bordered
sutured Heegaard diagram by
setting $\Sigma=f^{-1}(3/2)$, and letting $\balpha$ be the intersection
of the stable manifolds of the index--1 critical points with $\Sigma$, and
$\bbeta$ be the intersection of the unstable manifolds of the index--2
critical points with $\Sigma$. Note that the internal critical points
give $\balpha^c$ and $\bbeta$, while the ones in $F\subset \del Y$ give
$\balpha^a$. We notice that $\Zmid\subset F\cap\Sigma$ and $\balpha^a$
form an embedding $\psi\co G(\Z)\to\Sigma$. Homological linear independence
for the diagram follows from that of manifold.

\begin{defn}
A diagram as above is called a \emph{compatible bordered sutured Heegaard
diagram} to $f$.
\end{defn}

\begin{prop}
Compatible diagrams and boundary compatible Morse functions are in
a one-to-one correspondence.
\end{prop}

\begin{proof}
We need to give an inverse construction.
Start with a bordered sutured
diagram $\HH$, and construct a bordered sutured manifold in the following way.
To $\Sigma\times[1,2]$
attach 2--handles at $\balpha_i^c\times\{1\}$, and at $\bbeta_i\times\{2\}$.
Finally, at $\balpha_i^a\times\{1\}$ attach ``halves of 2--handles''.
These are thickened discs $D^2\times[0,1]$ attached along an arc
$a\times\{1/2\}\subset\del D^2\times\{1/2\}$. (See Fig.~\ref{fig:half_handle}.)
Then $\Gamma$ is $\Sigma\times\{3/2\}$,
and $F(\Z)$ is $\Zmid\times[1,2]$, together with the ``middles'' of the partial
handles, i.e.\ $(\del D^2\setminus a)\times[0,1]$.
To such a handle decomposition on the new manifold $Y$ corresponds a
canonical boundary compatible Morse function $f$. Note that attaching the half-handles
has no effect topologically, but adds boundary critical points.
\end{proof}

\begin{figure}
	\includegraphics[scale=.6]{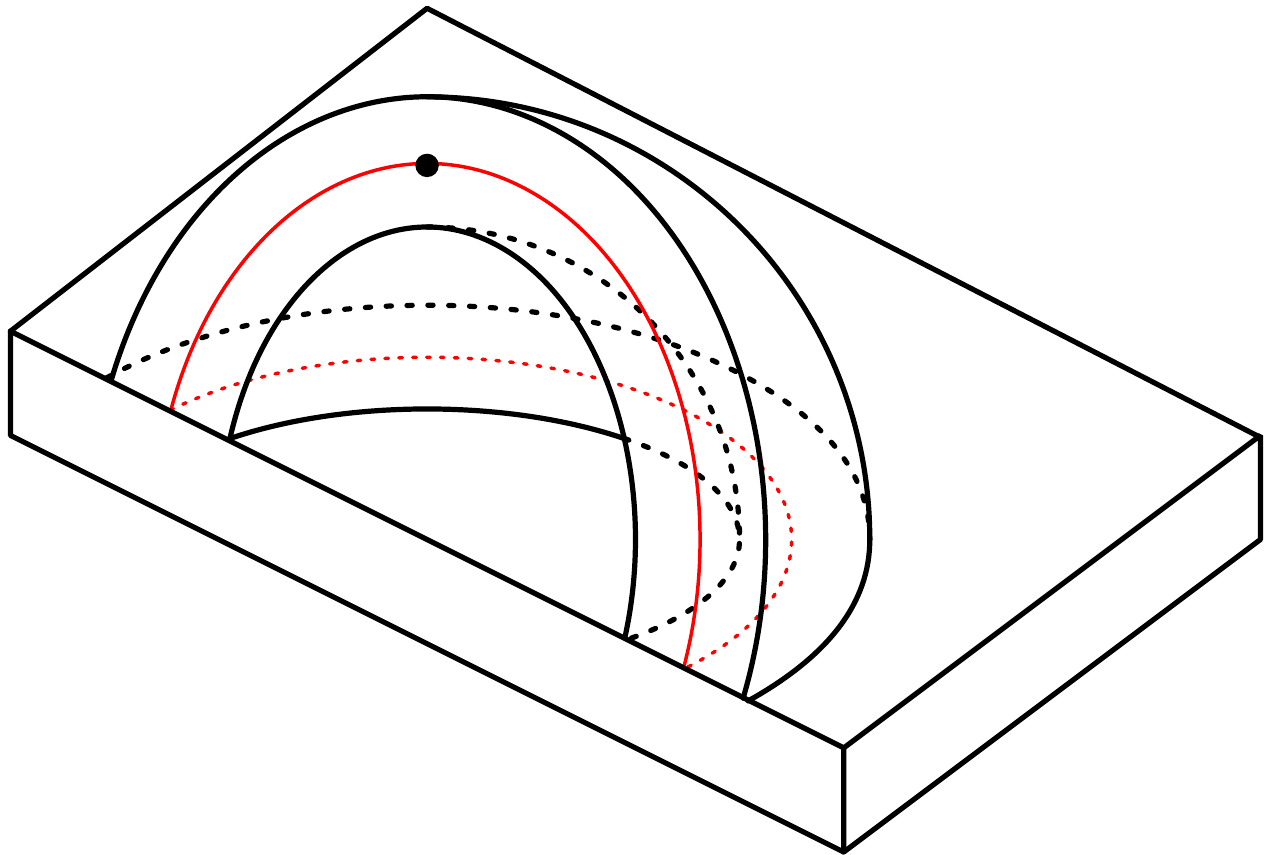}
\caption{Half of a 2--handle attached along an arc. Its critical
point and two incoming gradient flow lines are in the boundary.}
\label{fig:half_handle}
\end{figure}

\begin{prop}
\label{prop:heeagaard}
Any sutured bordered manifold has a compatible diagram in the above sense.
Moreover, any two compatible diagrams can be connected by a sequence of
moves of the following types:
\begin{enumerate}
\item Isotopy of the circles in $\balpha^c$ and $\bbeta$, and 
isotopy, relative to the endpoints, of the arcs in $\balpha^a$.
\item Handleslide of a circle in $\bbeta$ over another circle in $\bbeta$.
\item Handleslide of any curve in $\balpha$ over a circle in $\balpha^c$.
\item Stabilization.
\end{enumerate}
\end{prop}

\begin{proof}
For the proof of this proposition we will modify our definition of a
compatible Morse function, to temporarily ``forget'' about $F$.

A \emph{pseudo boundary compatible Morse function} $f$ for a bordered sutured
manifold $(Y,\Gamma,\Z,\phi)$ is a boundary compatible Morse function for the
manifold
$(Y,\Gamma,\varnothing,\varnothing\hookrightarrow\del Y)$ (which is just a standard
Morse function for the sutured manifold $(Y,\Gamma)$, in the sense of
\cite{Juh:SFH}), with some additional conditions. Namely, we require that
$f^{-1}([-1,3/2])\cap\phi(e_i)$ consist of two arcs (at the endpoints of
$\phi(e_i)$), tangent to $\nabla f$. We also require that $\phi(G(\Z))$ be
disjoint from the unstable manifolds of index--1 critical points.

Such Morse functions are in 1--to--1 correspondence with compatible diagrams
by the following construction. As usual, $\Sigma=f^{-1}(3/2)$, while
$\balpha^c$ and $\bbeta$ are the intersections of $\Sigma$ with stable,
respectively unstable, manifolds for index--1 and index--2 critical points.
On the other hand, $\balpha^a_i$ is the intersection of $\Sigma$ with the
gradient flow from $e_i$. Since the flow avoids index--1 critical points,
$\balpha^a$ is disjoint from $\balpha^c$. See Fig.~\ref{fig:morse} for a
comparison between the two types of Morse functions.

\begin{figure}
\begin{subfigure}[t]{.9\linewidth}
	\centering
	\labellist
	\small\hair 4pt
	\pinlabel $4$ [r] at -130 80
	\pinlabel $2$ [r] at -130 35
	\pinlabel $3/2$ [r] at -130 -20
	\pinlabel $1$ [r] at -130 -75
	\pinlabel $-1$ [r] at -130 -120
	\hair 2pt
	\pinlabel $R_+$ [bl] at 355 75
	\pinlabel $\Sigma$ [l] at 360 -20
	\pinlabel $\Gamma$ [l] at 360 -75
	\pinlabel $R_-$ [bl] at 355 -115
	\endlabellist
	\includegraphics[scale=.65]{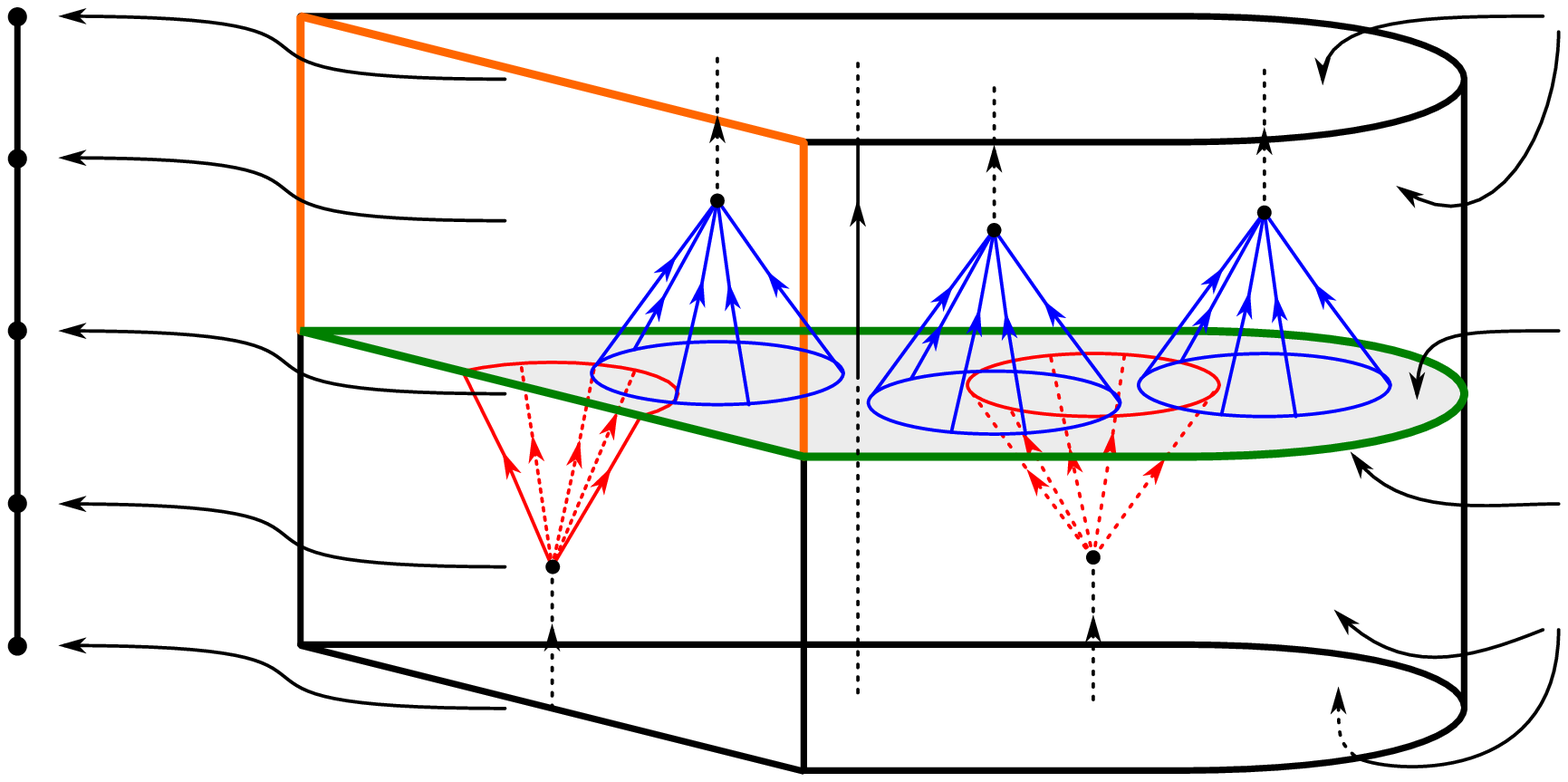}
	\caption{A true boundary compatible Morse function. There is one boundary critical point giving rise to $\balpha^a$.}
	\label{subfig:morse_true}
\end{subfigure}
\begin{subfigure}[t]{.9\linewidth}
	\centering
	\labellist
	\small\hair 4pt
	\pinlabel $4$ [r] at -130 80
	\pinlabel $2$ [r] at -130 35
	\pinlabel $3/2$ [r] at -130 -20
	\pinlabel $1$ [r] at -130 -75
	\pinlabel $-1$ [r] at -130 -120
	\hair 2pt
	\pinlabel $R_+$ [bl] at 355 75
	\pinlabel $\Sigma$ [l] at 360 -20
	\pinlabel $\Gamma$ [l] at 360 -75
	\pinlabel $R_-$ [bl] at 355 -115
	\endlabellist
	\includegraphics[scale=.65]{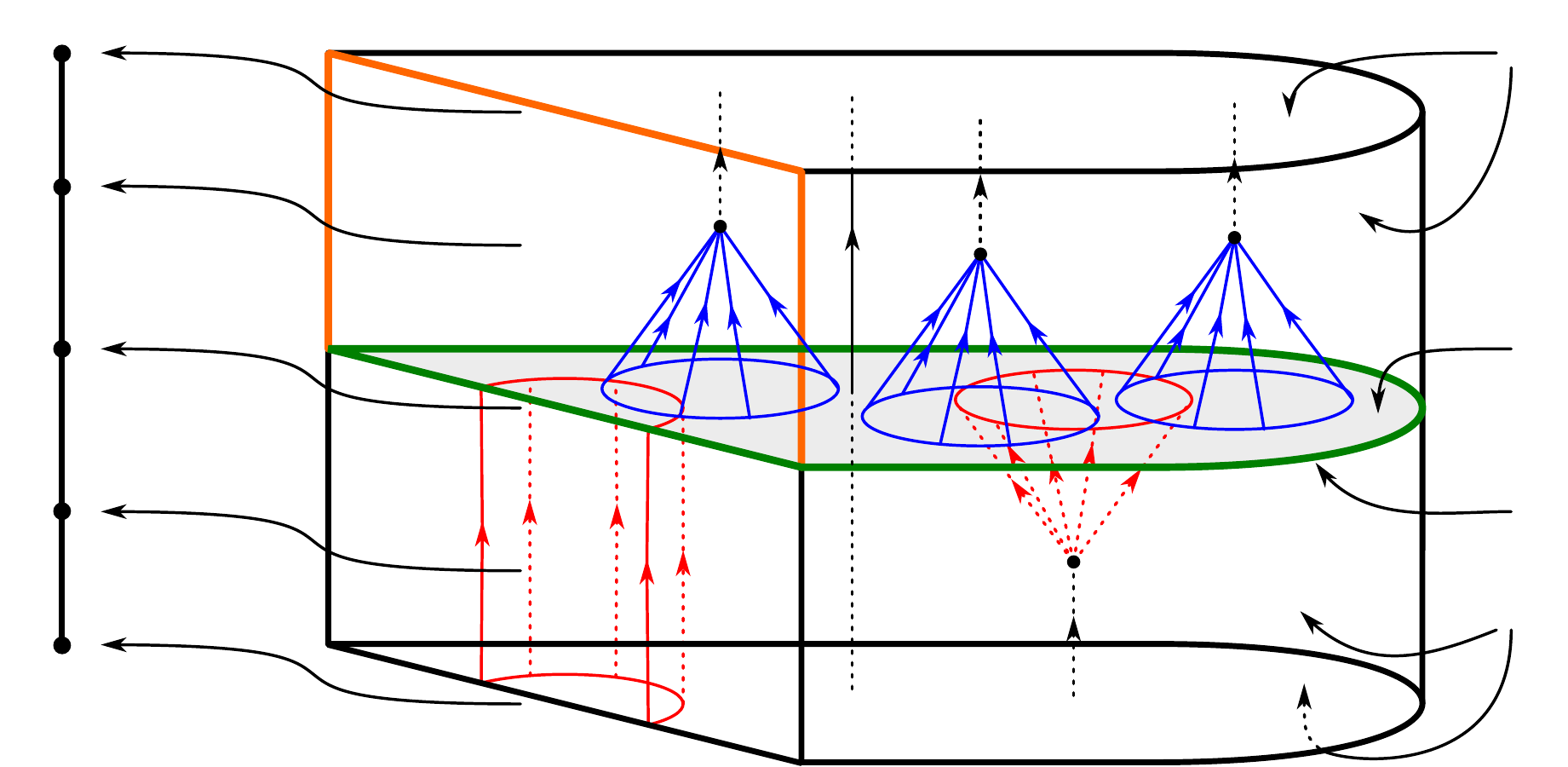}
	\caption{A pseudo boundary compatible Morse function. There is one arc in $f^{-1}(-1)$ giving rise to $\balpha^a$.}
	\label{subfig:morse_psudo}
\end{subfigure}
\caption{Comparison of a boundary compatible and pseudo boundary compatible Morse functions. Several internal
critical points are given in each, with gradient flowlines, giving rise to $\balpha^c$ and $\bbeta$.}
\label{fig:morse}
\end{figure}

The backwards construction is the same as for true boundary-compatible
Morse functions, except we do not attach the half 2--handles' at
$\balpha^a\times\{1\}$, and instead just set $e_i=\alpha_i^a\times\{1\}\cup\del\alpha_i^a\times[1,3/2]$.

This alternative construction allows us to use standard results about 
sutured manifolds. In particular, \cite[Propositions 2.13---2.15]{Juh:SFH}
imply that $(Y,\Gamma)$ has a compatible Morse function, and 
hence Heegaard diagram, and any two compatible diagrams are connected by
Heegaard moves. Namely, there is a family $f_t$ of Morse functions, which
for generic $t$ corresponds to an isotopy, and for a finite number of
critical points corresponds to a index--1, index--2 critical point creation,
(i.e. stabilization of the diagram), or a flowline between critical points of
the same index (handleslides between circles in $\balpha^c$ or between 
circles in $\bbeta$).

Since the stable manifold of any index--1 critical point
intersects $R_-$ at a pair of points, we can always perturb $f$ to get a
pseudo-compatible diagram for $(Y,\Gamma,\Z,\phi)$. Any two such diagrams are
connected by a sequence of sutured Heegaard moves (ignoring $\balpha^a$).
For generic $t$, a sutured compatible $f_t$ is also
pseudo bordered sutured compatible. At non-generic $t$, there is a flow from some point on $e_i$
to an index--1 critical point. This corresponds to sliding
$\alpha_i$ over the corresponding circle in $\balpha^c$, so we must add those to
the list of allowed Heegaard moves.
\end{proof}

\subsection{Generators}
\begin{defn}
A \emph{generator} for a bordered sutured diagram $\HH=(\Sigma,\balpha,\bbeta)$
is a collection $\xgen=(x_1,\ldots,x_g)$ of intersection points in
$\balpha\cap\bbeta$, such that there is exactly one point on each
$\alpha^c$ circle, exactly one point on each $\beta$ circle, and at most one
point on each $\alpha^a$ arc.

The set of all generators for $\HH$ is denoted $\G(\HH)$ or $\G$.

As a degenerate case, when $\#\bbeta=\#\balpha^c=0$, we will let $\G$ contain
a single element, which is the empty collection $\xgen=()$.
\end{defn}

Notice that if $\G$ is nonempty, then necessarily $g=\#\bbeta\geq\#\balpha^c$.
 We call $g$
the \emph{genus} of $\HH$. Moreover, exactly $p=g-\#\balpha^c$ many of the
$\alpha^a$ arcs are occupied by each generator.
Let $o(\xgen)\subset \{1,\ldots,k\}$ denote the set of occupied $\alpha^a$ arcs,
and $\obar(\xgen)=\{1,\ldots,k\}\setminus o(\xgen)$ denote the set of unoccupied arcs.

\begin{rmk}
If $\HH=(\Sigma,\balpha,\bbeta)$ is a bordered sutured diagram compatible with
a $p$--unbalanced bordered sutured manifold, then exactly $p$ many $\alpha^a$
arcs are occupied by each generator for $\HH$. 

Indeed, let $g=\#\bbeta$, and $h=\#\balpha$. By the construction of a
compatible manifold, $R_-(\Gamma)$ is diffeomorphic to $\Sigma$ after
surgery at each $\alpha^c$ circle, while $R_+(\Gamma)$ is diffeomorphic to
$\Sigma$ after surgery at each $\beta$ circle. But surgery on a surface at a
closed curve increases its Euler characteristic by 2. Therefore,
the manifold is $(g-h)$--unbalanced.
\end{rmk}

\subsection{Homology classes}
Later we will look at pseudoholomorphic curves that go ``between'' two
generators. We can classify such curves into homology classes as follows.

\begin{defn}
\label{def:homology_classes}
For given generators $\xgen$ and $\ygen$, the
\emph{homology classes} from $\xgen$ to $\ygen$, denoted by $\pi_2(\xgen,\ygen)$,
be the elements of
\begin{multline*}
H_2(\Sigma\times[0,1]\times[0,1],(\balpha\times\{1\}\times[0,1])\cup
(\bbeta\times\{0\}\times[0,1])\\
\cup(\Zmid\times[0,1]\times[0,1])
\cup(\xgen\times[0,1]\times\{0\})\cup(\ygen\times[0,1]\times\{1\})),
\end{multline*}
which map to the relative fundamental class of $\xgen\times[0,1]\cup\ygen\times[0,1]$
under the boundary homomorphism, and collapsing the rest of the boundary.
\end{defn}

There is
a product map $*\co \pi_2(\xgen,\ygen)\times\pi_2(\ygen,\zgen)\to\pi_2(\xgen,\zgen)$
given by concatenation at $\ygen\times[0,1]$. This product turns
$\pi_2(\xgen,\xgen)$ into a group, called the group of \emph{periodic classes at $\xgen$}.

\begin{defn}
The \emph{domain} of a homology class $B\in\pi_2(\xgen,\ygen)$ is the image
\begin{equation*}
[B]=\psig{}_*(B)\in H_2(\Sigma,\Zmid\cup\balpha\cup\bbeta).
\end{equation*}
We interpret it as a linear combination of regions in $\Sigma\setminus(\balpha\cup\bbeta)$.
We call the coefficient of such a region in a domain $D$ its \emph{multiplicity}.

The domain of a periodic class is a \emph{periodic domain}.
\end{defn}

We can split the boundary $\del [B]$ into pieces $\del^\del B\subset\Zmid,
\del^\alpha B\subset\balpha$, and $\del^\beta B\subset\bbeta$.
We can interpret $\del^\del B$ as an element of $H_1(\Zmid,\amid)$.

\begin{defn}
The set of \emph{provincial homology classes} from $\xgen$ to $\ygen$ is the
kernel $\pi_2^\del(\xgen,\ygen)$ of $\del^\del\co \pi_2(\xgen,\ygen)\to H_1(\Zmid,\amid)$.

The periodic classes in $\pi_2^\del(\xgen,\xgen)$ are \emph{provincial periodic class}
and their domains are \emph{provincial periodic domains}.
\end{defn}

The groups of periodic classes reduce to the much simpler forms
\begin{align*}
\pi_2(\xgen,\xgen)&\cong H_2(\Sigma\times[0,1],\Zmid\times[0,1]\cup\balpha\times\{0\}\cup\bbeta\times\{1\}),\\
\pi_2^{\del}(\xgen,\xgen)&\cong H_2(\Sigma\times[0,1],\balpha^{c}\times\{0\}\cup\bbeta\times\{1\}).
\end{align*}

Since 2--handles and half-handles are contractible, these groups are isomorphic to $H_2(Y,F)$ and $H_2(Y)$, respectively,
by attaching the cores of the handles.

\subsection{Admissibility}
As usual in Heegaard Floer homology, in order to get well defined invariants, we
need to impose certain admissibility conditions on the Heegaard diagrams. Like in
\cite{LOT:pairing}, there are two different notions of admissibility.

\begin{defn}
A bordered sutured Heegaard diagram is called \emph{admissible} if every nonzero periodic
domain has both positive and negative multiplicities.

A diagram is called \emph{provincially admissible} if every nonzero provincial periodic
domain has both positive and negative multiplicities.
\end{defn}

\begin{prop}
\label{prop:admissibility}
Any bordered sutured Heegaard diagram can be made admissible by performing
isotopy on $\bbeta$.
\end{prop}

\begin{cor}Any bordered sutured 3--manifold has an admissible diagram, and
any two admissible diagrams are connected, using Heegaard moves, through
admissible diagrams.

The analogous statement holds for provincially admissible diagrams.
\end{cor}

Since provincially admissible diagrams are a subset of admissible diagrams,
the second part of the argument trivially follows from the first. The first part,
on the other hand, follows from proposition~\ref{prop:admissibility}, by taking
any sequence of diagrams connected by Heegaard moves, and making all of them
admissible, through a consistent set of isotopies.

\begin{proof}[Proof of proposition~\ref{prop:admissibility}]
The proof is analogous to those for bordered manifolds and sutured manifolds.
We use the isomorphism from the previous section between periodic domains and $H_2(Y,F)$.

Notice that $H_1(\Sigma,\del\Sigma\setminus \Zmid)$ maps onto $H_1(Y,\del Y\setminus F)$, and therefore pairs
with $H_2(Y,F)$ and periodic domains. Find a basis for $H_1(\Sigma,\del\Sigma\setminus \Zmid)$,
represented by pairwise disjoint properly
embedded arcs $a_1,\ldots,a_m$. We can always do that since every component of $\Sigma$ hits
$\del\Sigma\setminus\Zmid$. Cutting $\Sigma$ along such arcs will give a collection of discs,
each of which contains exactly one component of $\Zmid$ in its boundary.

We can do finger moves of $\bbeta$ along each $a_i$, and along a push off $b_i$ of $a_i$, in the opposite direction.
This ensures that there are regions, for which the multiplicities of any periodic domain $D$ are equal
to its intersection numbers with $a_i$ and $b_i$, which have opposite signs.
Suppose $D$ has a nonzero region, and pick a point $p$ in such a region. By homological linear independence $p$ can be
connected to $\del\Sigma\setminus\Zmid$ in the complement of $\balpha\cup\Zmid$, as well as in
the complement of $\bbeta$. Connecting these paths gives a cycle in $H_1(\Sigma,\del\Sigma\setminus\Zmid)$ ,
which pairs non trivially with $D$. Since the $a_i$ span this group, at least one of them pairs non trivially
with $D$, which means $D$ has negative multiplicity in some region.
\end{proof}

\subsection{\texorpdfstring{$\spinc$--structures}{Spinc-structures}}

Recall that a \emph{$\spinc$--structure} on an $n$--manifold is a lift of its
principal $SO(n)$--bundle to a $\spinc(n)$--bundle. For 3--manifolds there is a
useful reformulation due to Turaev (see \cite{Tur:spinc}).
In this setting, a $\spinc$--structure $\s$ on the 3--manifold $Y$
is a choice of a non vanishing vector
field $v$ on $Y$, up to \emph{homology}. We say that
two vector fields are \emph{homologous} if they are homotopic outside of a finite collection of disjoint
open balls.

Given a trivialization of $TY$, a vector field $v$ on $Y$ can be thought of
as a map $v\co Y\to S^2$. This gives an identification of the set $\spinc(Y)$
of all $\spinc$--structures with $H^2(Y)$ via $\s(v)\mapsto v^*([S^2])$.
The identification depends on the trivialization of $TY$ by an overall shift
by a 2--torsion element. This means that $\spinc(Y)$ is naturally an
affine space over $H^2(Y)$.

Given a fixed vector field $v_0$ on a subspace $X\subset Y$, we can define
the space of \emph{relative $\spinc$--structures} $\spinc(Y,X,v_0)$, or just
$\spinc(Y,X)$ in the following way. A relative $\spinc$--structure is a vector
field $v$ on $Y$, such that $v|_X=v_0$, considered up to homology in
$Y\setminus X$. If $\spinc(Y,X,v_0)$ is nonempty, it is an affine space over
$H^2(Y,X)$.

To a $\spinc$--structure $\s$ in $\spinc(X)$ or $\spinc(Y,X,v_0)$, represented by a vector field $v$,
we can associate its \emph{Chern class} $c_1(\s)$, which is just the first Chern class
$c_1(v^{\bot})$ of the orthogonal complement subbundle $v^{\bot}\subset TY$.

With a generator in a Heegaard diagram
we will associate two types of $\spinc$--structures.
Let $\xgen\in\G(\HH)$ be a generator. Fix a boundary-compatible Morse function
$f$ (and appropriate metric). The vector field $\nabla f$ vanishes only at
the critical points of $f$. Each intersection point in $\xgen$ lies on
a gradient trajectory connecting an index--1 to an index--2 critical point.
If we cut out a neighborhood of that trajectory, we can modify the vector
field inside to one that is non vanishing (the two critical points have opposite
parity). For any unoccupied $\alpha^a$ arc, the corresponding critical point is
in $F\subset\del Y$. We can therefore modify the vector field in its 
neighborhood to be non vanishing. Call the resulting vector field $v(\xgen)$.

Notice that $v_0=v(\xgen)|_{\del Y\setminus F}=
\nabla f|_{\del Y\setminus F}$ does not depend on $\xgen$, while
$v(\xgen)|_{\del Y}=v_{o(\xgen)}$ only depends on $o(\xgen)$.
Moreover, under a change of the Morse function
or metric (even for different diagrams), $v_0$ and $v_{o(\xgen)}$
can only vary inside a contractible set. Therefore the corresponding
sets $\spinc(Y,\del Y\setminus F,v_0)$ and
$\spinc(Y,\del Y,v_{o(\xgen)})$, respectively, are canonically identified.
Thus we can talk about $\spinc(Y,\del Y\setminus F)$ and
$\spinc(Y,\del Y,o)$, where $o\subset\{1,\ldots,k\}$, as invariants of the
underlying bordered sutured manifold. This justifies the following definition.

\begin{defn}
Let $\s(\xgen)$ and $\s^{\rel}(\xgen)$ be the relative $\spinc$--structures
induced by $v(\xgen)$ in $\spinc(Y,\del Y\setminus F)$ and
$\spinc(Y,\del Y,o(\xgen))$, respectively.
\end{defn}

We can separate the generators into $\spinc$ classes.
Let
\begin{align*}
\G(\HH,\s) & = \{\xgen\in\G(\HH):\s(\xgen))=\s\},\\
\G(\HH,o,\s^{\rel}) & = \{\xgen\in\G(\HH):o(\xgen)=o,\s^{\rel}(\xgen)=\s^{\rel} \}.
\end{align*}

The fact that the invariants split by $\spinc$ structures is due to the following
proposition.

\begin{prop}
\label{prop:spinc_diff}
The set $\pi_2(\xgen,\ygen)$ is nonempty if and only if $\s(\xgen)=\s(\ygen)$.
The set $\pi_2^\del(\xgen,\ygen)$ is nonempty if and only if $o(\xgen)=o(\ygen)$ and
$\s^{\rel}(\xgen)=\s^{\rel}(\ygen)$.
\end{prop}
\begin{proof}
This proof is, again, analogous to those for bordered and for sutured manifolds.

To each pair of generators $\xgen,\ygen\in\G(\HH)$, we associate a homology class
$\epsilon(\xgen,\ygen)\in H_1(Y,F)$. We do that by picking 1--chains
$a\subset\balpha$, and $b\subset\bbeta$, such that $\del a=\ygen-\xgen+\mathbf{z}$, where
$\mathbf{z}$ is a 0--chain in $\Zmid$, 
and $\del b=\ygen-\xgen$, and
setting $\epsilon(\xgen,\ygen)=[a-b]$. We can interpret $a-b$ as a set of properly
embedded arcs and circles in $(Y,F)$ containing all critical points.

The vector fields $v(\xgen)$ and $v(\ygen)$ differ only in a neighborhood of $a-b$. One
can see that in fact $\s(\ygen)-\s(\xgen)=\PD([a-b])=\PD(\epsilon(\xgen,\ygen))$. On the
other hand, we can interpret $\epsilon(\xgen,\ygen)$ as an element of
\begin{equation*}
H_1(\Sigma\times[0,1],\balpha\times\{0\}\cup\bbeta\times\{1\}\cup\Zmid\times[0,1])\cong\ H_1(Y,F).
\end{equation*}
In particular, $\pi_2(\xgen,\ygen)$ is nonempty, if and only if there is a 2--chain in 
$\Sigma\times[0,1]$ with boundary which is a representative for $\epsilon(\xgen,\ygen)$ in the relative group above.
This is equivalent to $\epsilon(\xgen,\ygen)=0\in H_1(Y,F)$. This proves the first part
of the proposition.

The second one follows analogously, noticing that we can pick a path $a-b$, such that $a\subset\balpha$,
if and only if $o(\xgen)=o(\ygen)$, and in that case
$\pi_2^{\del}(\xgen,\ygen)$ is nonempty if and only if $\epsilon^{\rel}(\xgen,\ygen)=[a-b]=0\in H_1(Y)$, while
$\s^{\rel}(\ygen)-\s^{\rel}(\xgen)=PD([a-b])\in H^2(Y,\del Y).$
\end{proof}

\subsection{Gluing}
We can glue bordered sutured diagrams, similar to the way we glue bordered
sutured manifolds.

Let $\HH_1=(\Sigma_1,\balpha_1,\bbeta_1)$ and $\HH_2=(\Sigma_2,\balpha_2,\bbeta_2)$
be bordered sutured diagrams for $(Y_1,\Gamma_1,\Z,\phi_1)$ and
$(Y_2,\Gamma_2,-\Z,\phi_2)$, respectively. We can identify $\Zmid$ with its
embeddings in
$\del\Sigma_1$ and $\del\Sigma_2$ (one is orientation preserving, the other is
orientation reversing).

Let $\Sigma=\Sigma_1\cup_{\Zmid}\Sigma_2$. Each $\alpha^a$ arc in $\HH_1$
matches up with the corresponding one in $\HH_2$ to form a closed curve in
$\Sigma$. Let $\balpha$ denote the union of all $\alpha^c$ circles in $\HH_1$
and $\HH_2$, together with the newly formed circles from all $\alpha^a$ arcs.
Finally, let $\bbeta=\bbeta_1\cup\bbeta_2$.

\begin{prop}
The diagram $\HH=(\Sigma,\balpha,\bbeta)$ is compatible with the sutured manifold
$Y_1\cup_{F(\Z)}Y_2$, as defined in 
section~\ref{sec:glue_mflds}.
\end{prop}
\begin{proof}
The manifolds $Y_1$ and $Y_2$ are obtained from $\Sigma_1\times[1,2]$ and
$\Sigma_2\times[1,2]$, respectively, by attaching 2--handles (corresponding
to $\alpha^c$ and $\beta$ circles), and halves of 2--handles (corresponding
to $\alpha^a$ arcs). The surface of gluing $F$ can be identified with the union
of $\Zmid\times[1,2]$ with the middles of the half-handles. Thus, we get
a base of $(\Sigma_1\cup_{\Zmid}\Sigma_2)\times[1,2]$, with the combined
2--handles from each side. In addition the half-handles glue in pairs to form
actual 2--handles, each of which is glued along matching $\alpha^a$ arcs.
\end{proof}

Similarly, we can do partial gluing. If we have manifolds
$(Y_1,\Gamma_1,\Z_0\cup\Z_1,\phi_1)$ and $(Y_2,\Gamma_2,-\Z_0\cup\Z_2,\phi_2)$
with diagrams $\HH_1$ and $\HH_2$, respectively,
$\HH_1\cup_{\Zmid_0}\HH_2$ is a diagram compatible with the bordered sutured
manifold $Y_1\cup_{F(\Z_0)}Y_2$.

\subsection{Nice diagrams}
As with the other types of Heegaard Floer invariants, the invariants become a lot
easier to compute (at least conceptually) if we work in the category of
\emph{nice diagrams}, developed originally by Sarkar and Wang in 
\cite{SW:nice_diagrams}.

\begin{defn}
A bordered sutured diagram $\HH=(\Sigma,\balpha,\bbeta,\Z,\psi)$ is \emph{nice}
if every region of $\Sigma\setminus(\balpha\cup\bbeta)$ is either adjacent to
$\del\Sigma\setminus\Zmid$---in which case we call it a \emph{boundary region}---or
is a topological disc with at most 4 corners.
\end{defn}

\begin{prop}
Any bordered sutured diagram can be made nice by isotopies of $\bbeta$, handleslides
among the circles in $\bbeta$, and stabilizations.
\end{prop}
\begin{proof}
The proof is a combination of those for bordered and sutured manifolds, in 
\cite{LOT:pairing} and \cite{Juh:decompositions}, respectively.

First, we make some stabilizations until every component of $\Sigma$ contains both $\alpha$ and $\beta$
curves. Next we do finger moves of $\beta$ curves until any curve in $\balpha$ intersects $\bbeta$, and
vice versa. Then, we ensure all non boundary regions are simply connected. We do that inductively, decreasing the rank
of $H_1$ relative boundary for each region.

Then, following \cite{LOT:pairing},
we do finger moves of some $\beta$ curves along curves parallel to each component of $\Zmid$
to ensure that all regions adjacent to some Reeb chord in $\Zmid$ are rectangles (where one side is in
$\Zmid$,
two are in $\balpha^a$, and one is in $\bbeta$). 

Finally, we label all regions by their \emph{distance}, i.e. number of $\beta$ arcs in $\Sigma\setminus\balpha$ one
needs to cross, to get to a boundary region, and by their badness (how many extra corners they have).
We do finger moves of a $\beta$ arc in a bad region through $\alpha$ arcs, until we hit a boundary, a bigon, or
another (or the same) bad region. There are several cases depending on what kind of region we hit, but the overall
badness of the diagram decreases, so the algorithm eventually terminates. The setup is such that we can never hit
a region adjacent to a Reeb chord, so the algorithm for sutured manifolds goes through for bordered sutured manifolds.
\end{proof}

\section{Moduli spaces of holomorphic curves}
\label{sec:moduli}

In this section we describe the moduli spaces of holomorphic curves involved in the
definitions of the bordered invariants and prove the necessary properties.
The definitions and arguments are mostly a straightforward generalization
of those in \cite[Chapter 5]{LOT:pairing}.

\subsection{Differences with bordered Floer homology}
For the reader familiar with border Floer homology we highlight the similarities
and the differences with our definitions.

In the bordered setting of Lipshitz, Ozsv\'ath, and Thurston, there is
one boundary component and one basepoint on the boundary. One counts
pseudoholomorphic discs in $\Sigma\times[0,1]\times\RR$, but in practice
one thinks of their domains in $\Sigma$. Loosely speaking, the
curves that do not hit $\del\Sigma$ correspond to differentials, the ones
that do hit the boundary correspond to algebra actions, while the ones that
hit the basepoint are not counted at all.

In the bordered sutured setting, the boundary $\del\Sigma$ has
several components, while some subset $\Zmid$ of $\del\Sigma$ is distinguished.
We again count pseudoholomorphic curves in $\Sigma\times[0,1]\times\RR$, and
again, those curves that do not hit the boundary correspond to differentials.
The novel idea is the interpretation of the boundary. Here the algebra action
comes from curves that hit any component of $\Zmid\subset\del\Sigma$, while the
curves that hit any component of $\del\Sigma\setminus\Zmid$ are not
counted. In a sense, the set $\del\Sigma\setminus\Zmid$ plays the 
role of the basepoint.

With this in mind, most of the constructions in \cite{LOT:pairing} carry over.
Below we describe the necessary analytic constructions.

\subsection{Holomorphic curves and conditions}
We will consider several variations of the Heegaard surface $\Sigma$, namely 
the compact surface with boundary $\sigbar=\Sigma$, the open surface $\sigint$,
which can be thought of as a surface with several punctures
$\textbf{p}=\{p_1,\ldots,p_n\}$, and the closed surface $\sige$, obtained by
filling in those punctures. Alternatively, it is obtained from $\sigbar$ by collapsing
all boundary components to points.

We will also be interested in the surface $\DD=[0,1]\times\RR$, with
coordinates $s\in[0,1]$ and $t\in\RR$.

Let $\wsig$ be a symplectic form on $\sigint$, such that $\del\Sigma$ is a cylindrical
end, and let $\jsig$ be a compatible almost complex structure. We can assume that $\balpha^a$
is cylindrical near the punctures in the following sense. There is a neighborhood
$U_{\textbf{p}}$ of the punctures, symplectomorphic to
$\del\Sigma\times(0,\infty)\subset T^*(\del\Sigma)$, such that
$\jsig$ and $\balpha^a\cap U_\textbf{p}$ are invariant with respect to the
$\RR$--action on $\del\Sigma\times(0,\infty)$. 
Let $\wdd$ and $\jd$ be the standard symplectic form and almost complex structure on
$\DD\subset\CC$.

Consider the projections
\begin{align*}
\psig\co \sigint\times\DD&\to\sigint,\\
\pd\co \sigint\times\DD&\to\DD,\\
s\co \sigint\times\DD&\to[0,1],\\
t\co \sigint\times\DD&\to\RR.
\end{align*}

\begin{defn}
An almost complex structure $J$ on $\sigint\times\DD$ is called \emph{admissible} if the
following conditions hold:
\begin{itemize}
\item $\pd$ is $J$--holomorphic.
\item $J(\del_s)=\del_t$ for the vector fields $\del_s$ and $\del_t$ in the fibers of $\psig$.
\item The $\RR$--translation action in the $t$--coordinate is $J$--holomorphic.
\item $J=\jsig\times\jd$ near $\textbf{p}\times\DD$.
\end{itemize}
\end{defn}

\begin{defn}
A \emph{decorated source} $S^{\triangleright}$ consists of the following data:
\begin{itemize}
\item A topological type of a smooth surface $S$ with boundary, and a finite number
of boundary punctures.
\item A labeling of each puncture by one of ``$+$'', ``$-$'', or ``$e$''.
\item A labeling of each $e$ puncture by a Reeb chord $\rho$ in $\Zmid$.
\end{itemize}
\end{defn}

Given $S^{\trr}$ as above, denote by $S_{\ebar}$ the surface obtained from $S$
by filling in all the $e$ punctures.

We consider maps
\begin{equation*}
u\co (S,\del S)\to(\sigint\times\DD,(\balpha\times\{1\}\times\RR)
\cup(\bbeta\times\{0\}\times\RR))
\end{equation*}
satisfying the following conditions:
\begin{enumerate}
\item \label{cond:first} $u$ is $(j,J)$--holomorphic for some almost complex structure $j$ on $S$.
\item $u\co S\to\sigint\times\DD$ is proper.
\item $u$ extends to a proper map $u_{\ebar}\co S_{\ebar}\to\sige\times\DD$.
\item $u_{\ebar}$ has finite energy in the sense of Bourgeois, Eliashberg, Hofer,
Wysocki and Zehnder \cite{BEHWZ:compactness}.
\item $\pd\circ u\co S\to\DD$ is a $g$--fold branched cover. (Recall that
$g$ is the cardinality of $\bbeta$, not the genus of $\Sigma$).
\item At each $+$ puncture $q$ of $S$, $\lim_{z\to q}t\circ u(z)=+\infty$.
\item At each $-$ puncture $q$ of $S$, $\lim_{z\to q}t\circ u(z)=-\infty$.
\item At each $e$ puncture $q$ of $S$, $\lim_{z\to q}\psig\circ u(z)$ is the Reeb chord
$\rho$ labeling $q$.
\item \label{cond:last}
$\psig\circ u\co S\to\sigint$ does not cover any of the regions of
$\Sigma\setminus(\balpha\cup\bbeta)$ adjacent to $\del\Sigma\setminus\Zmid$.
\item \label{cond:strong_mon}
\emph{Strong boundary monotonicity.} For each $t\in\RR$, and each $\beta_i\in\bbeta$,
$u^{-1}(\beta_i\times\{0\}\times\{t\})$ consists of exactly one point. For each
$\alpha_i^c\in\balpha^c$, $u^{-1}(\alpha_i^c\times\{1\}\times\{t\})$ consist of exactly
one point. For each
$\alpha_i^a\in\balpha^a$, $u^{-1}(\alpha_i^a\times\{1\}\times\{t\})$ consists of \emph{at most}
one point.
\item \label{cond:embedded}
$u$ is embedded.
\end{enumerate}

Under conditions~(\ref{cond:first})--(\ref{cond:last}), at each $+$ or $-$ puncture, $u$
is asymptotic to an arc $z\times[0,1]\times\{\pm\infty\}$, where $z$ is some
intersection point in $\balpha\cap\bbeta$.
If in addition we require condition~(\ref{cond:strong_mon}), then the intersection points
$x_1,\ldots,x_g$ corresponding to $-$ punctures form a generator $\xgen$,
while the ones $y_1,\ldots,y_g$ corresponding to $+$ punctures form a 
generator $\ygen$. We call $\xgen$ the \emph{incoming} generator, and $\ygen$ the
\emph{outgoing} generator for $u$.

If we compactify the $\RR$ component of $\DD$ to include $\{\pm\infty\}$, we get a
compact rectangle $\widetilde{\DD}=[0,1]\times[-\infty,+\infty]$. Let $u$ be a map satisfying
conditions~(\ref{cond:first})--(\ref{cond:strong_mon}), and with incoming and outgoing generators
$\xgen$ and $\ygen$.
Let $\widetilde{S}$ be $S$ with all punctures filled in by arcs. Then $u$ extends to a map
\begin{multline*}
\widetilde{u}\co (\widetilde{S},\del\widetilde{S})\to
(\Sigma\times\widetilde{\DD},(\balpha\times\{1\}\times[-\infty,+\infty])\cup
(\bbeta\times\{0\}\times[-\infty,+\infty])\\
\cup(\Zmid\times\widetilde{\DD})\cup
(\xgen\times[0,1]\times\{-\infty\})\cup(\ygen\times[0,1]\times\{+\infty\})).
\end{multline*}

Notice that the pair of spaces on the right is the same as the one in definition~\ref{def:homology_classes}.
It is clear that a map $u$ satisfying conditions~(\ref{cond:first})--(\ref{cond:strong_mon}) has an
associated homology class $B=[u]=[\widetilde u]\in\pi_2(\xgen,\ygen)$. 

We will also impose an extra condition on the height of the $e$ punctures of $S$.

\begin{defn}
For a map $u$ from a decorated source $S^\trr$, and an $e$ puncture $q$ on $\del S$, the
\emph{height} of $q$ is the evaluation $\ev(q)=t\circ u_{\ebar}(q)\in\RR$.
\end{defn}

\begin{defn}
Let $E(S^\trr)$
be the set of all $e$ punctures in $S$. Let $\parrow=(P_1,\ldots,P_m)$ be an ordered partition
of $E(S^\trr)$ into nonempty subsets. We say $u$ is \emph{$\parrow$--compatible} if for $i=1,\ldots,m$ all the
punctures in $P_i$ have the same height $\ev(P_i)$, and moreover
$\ev(P_i)<\ev(P_j)$ for $i<j$.
\end{defn}

To a partition $\parrow=(P_1,\ldots,P_m)$ we can associate a sequence $\brarrow(\parrow)=(\brho_1,\ldots,\brho_m)$ of
sets of Reeb chords, by setting
\begin{equation*}
\brho_i=\{\rho:\rho~{labels}~q,q\in P_i\}.
\end{equation*}

Moreover, to any such sequence $\brarrow$ we can associate a homology class
\begin{equation*}
[\brarrow]=[\brho_1]+\cdots+[\brho_m]\in H_1(\Zmid,\amid),
\end{equation*}
and an algebra element
\begin{equation*}
a(\brarrow)=a(\brho_1)\cdots a(\brho_m).
\end{equation*}

It is easy to see that $[a(\brarrow)]=[\brarrow]$. It is also easy to see that for a curve $u$ satisfying
conditions~(\ref{cond:first})--(\ref{cond:strong_mon}) with homology class $[u]=B$,
and for any
partition $\parrow$ we have
$[\brarrow(\parrow)]=\del^{\del}B$. 
\subsection{Moduli spaces}

We are now ready to define the moduli spaces that we will consider.

\begin{defn}
Let $\xgen,\ygen\in\G(\HH)$ be generators, let $B\in\pi_2(\xgen,\ygen)$ be a homology
class, and let $S^{\trr}$ be a decorated source. We will write
\begin{equation*}
\mt^B(\xgen,\ygen,S^\trr)
\end{equation*}
for the space of curves $u$ with source $S^\trr$ satisfying conditions~(\ref{cond:first})--(\ref{cond:strong_mon}),
asymptotic to $\xgen$ at $-\infty$ and to $\ygen$ at $+\infty$, and with homology class $[u]=B$.

This moduli space is stratified by the possible partitions of $E(S^\trr)$. More precisely, given a partition
$\parrow$ of $E(S^\trr)$, we write
\begin{equation*}
\mt^B(\xgen,\ygen,S^\trr,\parrow)
\end{equation*}
for the space of $\parrow$--compatible maps in $\mt^B(\xgen,\ygen,S^\trr)$, and
\begin{equation*}
\mt_{\emb}^B(\xgen,\ygen,S^\trr,\parrow)
\end{equation*}
for the space of maps in $\mt^B(\xgen,\ygen,S^\trr,\parrow)$ that also satisfy~(\ref{cond:embedded}).
\end{defn}

\begin{rmk}
The definitions in the current section are analogous to those in \cite{LOT:pairing}, and
a lot of the results in that paper carry over without change. We will cite several of them here without
proof.
\end{rmk}

\begin{prop} (Compare to \cite[Proposition 5.5]{LOT:pairing}.) There is a dense set of admissible $J$ with the
property that
for all generators $\xgen$, $\ygen$, all homology classes $B\in\pi_2(\xgen,\ygen)$ and all
partitions $\parrow$, the spaces $\mt^B(\xgen,\ygen,S^\trr,\parrow)$ are transversely cut out by the
$\dbar$--equations.
\end{prop}

\begin{prop} (Compare to \cite[Proposition 5.6]{LOT:pairing}.) The expected dimension $\ind(B,S^\trr,\parrow)$ of
$\mt^B(\xgen,\ygen,S^\trr,\parrow)$ is
\begin{equation*}
\ind(B,S^\trr,P)=g-\chi(S)+2e(B)+\#\parrow,
\end{equation*}
where $e(B)$ is the Euler measure of the domain of $B$.
\end{prop}

It turns out that whether the curve $u\in\mt^B(\xgen,\ygen,S^\trr,\parrow)$ is embedded depends
entirely on the topological data consisting of $B$, $S^\trr$, and $\parrow$.
That is, there are entire components of
embedded and of non embedded curves. Moreover, for such curves there is another index formula
that does not depend on $S^\trr$. To give it we need some more definitions.

For a homology class $B\in\pi_2(\xgen,\ygen)$, and a point $z\in\balpha\cap\bbeta$, let
$n_z(B)$ be the average multiplicity of $[B]$ at the four regions adjacent to $z$. Let
$n_\xgen=\sum_{x\in\xgen}n_x(B)$, and $n_\ygen=\sum_{y\in\ygen}n_y(B)$.

For a sequence $\brarrow=(\brho_1,\ldots,\brho_m)$, let $\iota(\brarrow)$ be the Maslov component
of the grading $\gr(\brho_1)\cdots\gr(\brho_m)$.

\begin{defn} For a homology class $B\in\pi_2(\xgen,\ygen)$ and a sequence $\brarrow=(\brho_1,\ldots,\brho_m)$
of Reeb chords,
define the \emph{embedded Euler characteristic} and
\emph{embedded index}
\begin{align*}
\chi_{\emb}(B,\brarrow) & = g+e(B)-n_\xgen(B)-n_\ygen(B)-\iota(\brarrow), \\
\ind(B,\brarrow) & = e(B)+n_\xgen(B)+n_\ygen(B)+\#\brarrow+\iota(\brarrow).
\end{align*}
\end{defn}

\begin{prop}
Suppose $u\in\mt^B(\xgen,\ygen,S^\trr,\parrow)$.
Exactly one of the following two statements holds.

\begin{enumerate}
\item
$u$ is embedded and the following equalities hold.
\begin{align*}
\chi(S^\trr)&=\chi_{\emb}(B,\brarrow(\parrow)),\\
\ind(B,S^\trr,\parrow)&=\ind(B,\brarrow(\parrow)),\\
\mt_{\emb}^B(\xgen,\ygen,S^\trr,\parrow)&=\mt^B(\xgen,\ygen,S^\trr,\parrow).
\end{align*}
\item
$u$ is not embedded and the following inequalities hold.
\begin{align*}
\chi(S^\trr)&>\chi_{\emb}(B,\brarrow(\parrow)),\\
\ind(B,S^\trr,\parrow)&<\ind(B,\brarrow(\parrow)),\\
\mt_{\emb}^B(\xgen,\ygen,S^\trr,\parrow)&=\varnothing.
\end{align*}
\end{enumerate}
\end{prop}

\begin{proof}

This is essentially a restatement of \cite[Proposition 5.47]{LOT:pairing}
\end{proof}

Each of these moduli spaces has an $\RR$--action that is translation in the $t$ factor.
It is free on each $\mt^B(\xgen,\ygen,S^\trr,\parrow)$, except when the moduli space consists
of a single curve $u$, where $\pd\circ u$ is a trivial
$g$--fold cover of $\DD$, and $\psig\circ u$ is constant (so $B=0$). We say that $u$ is
\emph{stable} if it is not this trivial case.

For moduli spaces of stable curves we mod out by this $\RR$--action:
\begin{defn}
For given $\xgen$, $\ygen$, $S^\trr$, and $\parrow$, set
\begin{align*}
\M^B(\xgen,\ygen,S^\trr,\parrow)&=\mt^B(\xgen,\ygen,S^\trr,\parrow)/\RR,\\
\M_{\emb}^B(\xgen,\ygen,S^\trr,\parrow)&=\mt_{\emb}^B(\xgen,\ygen,S^\trr,\parrow)/\RR.
\end{align*}
\end{defn}

\subsection{Degenerations}
\label{sec:degenerations}
The properties of the moduli spaces which are necessary to prove that the invariants are well defined and
have the expected properties, are essentially the same as in \cite{LOT:pairing}. Their proofs
also carry over with minimal change. We sketch below the most important results.

To study degenerations we first pass to the space of \emph{holomorphic combs} which are trees of
holomorphic curves in $\Sigma\times\DD$, and ones that live at \emph{East infinity}, i.e. in
$\Zmid\times\RR\times\DD$. This is the proper ambient space to work in, to ensure compactness.

The possible degenerations that can occur at the boundary of 1--di\-men\-sion\-al moduli spaces of embedded curves
are of two types. One is a two story holomorphic building, as usual in Floer theory. The second type
consists of a single curve $u$ in $\Sigma\times\DD$, with another curve degenerating at East infinity, at
the $e$ punctures of $u$. Those curves that can degenerate at East infinity are of several types,
\emph{join} curves, \emph{split} curves, and \emph{shuffle} curves, that correspond to
certain operations on the algebra $\A(\Z)$. In fact, the types of curves that can appear dictate how
the algebra should behave.

There are also corresponding gluing results, that tell us that in the cases we care about, a rigid holomorphic
comb is indeed the boundary of a 1--dimensional space of curves. Unfortunately, in some cases the compactified
moduli spaces are not compact 1--manifolds. However, we can still recover the necessary result that certain
counts of 0--dimensional moduli spaces are even, and thus become $0$, when reduced to $\ZZ/2$.

The only place where significant changes need to be made to the arguments, are in ruling out
bubbling and boundary degenerations. The reason for the changes are the different homological
assumptions we have made for $\Sigma$, $\Zmid$, $\balpha$, and $\bbeta$ in the definition of bordered sutured
Heegaard diagrams. We give below the precise statement, and the modified proof.
The rest of the arguments are essentially local in nature, and do not depend on these
homological assumptions. 

\begin{prop}
\label{prop:degenerations}
Suppose $\M=\M^B(\xgen,\ygen,S^\trr,\parrow)$ is 1--dimensional. Then the following types
of degenerations cannot occur as the limit $u$ of a sequence $u_j$ of curves in $\M$.
\begin{enumerate}
\item \label{cond:bubble}
$u$ bubbles off a closed curve.
\item \label{cond:bound_degen}
$u$ has a \emph{boundary degeneration}, i.e. $u$ is a nodal curve that
collapses one or more properly embedded arcs in $(S,\del S)$.
\end{enumerate}
\end{prop}
\begin{proof}

For~(\ref{cond:bubble}) notice that if a closed curve bubbles off,
it has to map to $\sigint\times\DD\simeq\sigint$ which has no closed components.
In particular, $H_2(\sigint\times\DD)=0$, and the bubble will have zero energy.

For~(\ref{cond:bound_degen}), assume there is such a degeneration $u$ with source $S^{\trr}{}'$.
Repeating the argument in \cite[Lemma 5.37]{LOT:pairing},
if an arc $a\in S^{\trr}$ collapses in $u$, then by strong boundary monotonicity its endpoints $\del a$
lie on the same arc in $\del\Sigma$. If $b$ is the arc in $\del S^{\trr}{}'$ connecting them, then
$t\circ u$ is constant on $b$. Therefore, $\pd\circ u$ is constant on the entire component $T$ of $S^{\trr}{}'$
containing $b$.

There is a compactification $\overline{T}$ of $T$, filling in the punctures by arcs, and
an induced map $\overline{u}\co \overline{T}\to\Sigma\times\DD$. Under $\overline{u}$, the boundary
$\del\overline{T}$ can only map to either
$\balpha\cup\Zmid\times\{1\}\times\RR$ or to $\bbeta\times\{0\}\times\RR$,
which have different $t$ components. Therefore, the entire boundary maps entirely to one of them, or entirely to the
other. In particular, we have either a map
\begin{equation*}
\psig\circ\overline{u}\co (\overline{T},\del\overline{T})\to(\Sigma,\balpha\cup\Zmid),
\end{equation*}
or a map
\begin{equation*}
\psig\circ\overline{u}\co (\overline{T},\del\overline{T})\to(\Sigma,\bbeta).
\end{equation*}

By homological linear independence both of the groups $H_2(\Sigma,\balpha\cup\Zmid)$ and $H_2(\Sigma,\bbeta)$ are
$0$, and $u|_T$ must have zero energy.
\end{proof}

The equivalent statement in the bordered setting is necessary for \cite[Proposition 5.32]{LOT:pairing}.

\section{Diagram gradings}
\label{sec:grading}
In this section we define gradings on the set of generators $\G(\HH)$ for a given bordered sutured diagram $\HH$.
More precisely, if $\HH$ represents the bordered sutured manifold $(Y,\Gamma,\Z)$, for each $\spinc$--structure
$\s\in\spinc(Y,\del Y\setminus F(\Z))$ we define grading sets $\Gr(\HH,\s)$ and $\Grdn(\HH,\s)$ which have left
actions by $\Gr(-\Z)$ and $\Grdn(-\Z)$, respectively, and right actions by $\Gr(\Z)$ and $\Grdn(\Z)$, respectively.
Then we define maps $\G(\HH,\s)\to\Gr(\HH,\s)$ and $\G(\HH,\s)\to\Grdn(\HH,\s)$, which are well-defined up to a shift to
be made precise below. In the next couple of sections we use these maps to define relative gradings on the bordered sutured
invariants.

\subsection{Domain gradings}
We start by defining a grading on all homology classes in $\pi_2(\xgen,\ygen)$ for $\xgen$ and $\ygen$ generators in
$\G(\HH)$. We will abuse notation and will not distinguish between a given homology class and its associated domain in
$H_2(\Sigma,\Zmid\cup\balpha\cup\bbeta)$.

\begin{defn} Given a domain $B\in\pi_2(\xgen,\ygen)$ define
\begin{equation*}
\gr(B)=(-e(B)-n_{\xgen}(B)-n_{\ygen}(B),\del^{\del}B)\in\Gr(\Z).
\end{equation*}

Given a grading reduction $r$ from $\Gr(\Z)$ to $\Grdn(\Z)$, define
\begin{equation*}
\grdn(B)=r(I(o(\xgen)))\cdot\gr(B)\cdot r(I(o(\ygen)))^{-1}\in\Grdn(\Z).
\end{equation*}
\end{defn}

The basic properties of these gradings, and in fact the reason they are called gradings is that they are compatible with
composition of domains. They are also compatible with the indices of moduli spaces.

\begin{prop}
Given a domain $B\in\pi_2(\xgen,\ygen)$, for any compatible sequence $\brarrow$ of sets of Reeb chords, we have
$\gr(B)=\lambda^{-\ind(B,\brarrow)+\#\brarrow}\cdot\gr(\brarrow)$.

For any two domains $B_1\in\pi_2(\xgen,\ygen)$ and $B_2\in\pi_2(\ygen,\zgen)$, their concatenation has grading
$\gr(B_1*B_2)=\gr(B_1)\cdot\gr(B_2)$.

Similar statements hold for $\grdn(B)$.
\end{prop}
\begin{proof}
For the first statement, recall that $\ind(B,\brarrow)=e(B)+n_{\xgen}(B)+n_{\ygen}(B)+\iota(\brarrow)+\#\brarrow$, and
the homological components of $\gr(\brarrow)$ and $\gr(B)$ are both $\del^{\del}B$ for a compatible pair. The second
statement follows from the first, using the fact that the index is additive for domains, and $\lambda$ is central.

For the equivalent statement for $\grdn$, we just have to use $\grdn(I_{o(\xgen)}\cdot a(\brarrow)\cdot I_{o(\ygen)})$,
instead
of $\grdn(\brarrow)$ which is not defined, and notice that the reduction terms match up.
\end{proof}

\subsection{Generator gradings}
We will give a relative grading for the generators in each $\spinc$--structure. Here a \emph{relative grading in a $G$--set}
means a map $g\co\G(\HH,\s)\to A$, where $G$ acts on $A$, say on the right. Two such gradings $g$ and $g'$ with values
in $A$ and $A'$ are equivalent, if there is a bijection $\phi\co A\to A'$, such that $\phi$ commutes with both the
$G$--actions and with $g$ and $g'$. The traditional case of a relative $\ZZ$ or $\ZZ/n$--valued grading corresponds to $\ZZ$
acting on its quotient, with the grading map defined up to an overall shift by a constant.

\begin{defn}
For a Heegaard diagram $\HH$ and generator $\xgen\in\G(\HH)$ define the \emph{stabilizer subgroup}
$\PP(\xgen)=\gr(\pi_2(\xgen,\xgen))\subset\Gr(\Z)$. For any $\spinc$--structure $\s$ pick a generator
$\xgen\in\G(\HH,\s)$ and let $\Gr(\HH,\s)$ be the set of right cosets
$\PP(\xgen_0)\backslash \Gr(\Z)$ with the usual right $\Gr(\Z)$--action.
Define the grading $\gr\co\G(\HH,\s)\to\Gr(\HH,\s)$ by $\gr(\xgen)=\PP\cdot\gr(B)$ for any $B\in\pi_2(\xgen_0,\xgen)$.
\end{defn}

\begin{prop}
Assuming $\G(\HH,\s)$ is nonempty, this is a well-defined relative grading, independent of the choice of $\xgen_0$, and has
the property $\gr(\xgen)\cdot\gr(B)=\gr(\ygen)$ for any $B\in\pi_2(\xgen,\ygen)$.
\end{prop}
\begin{proof}
These follow quickly from the fact that concatenation of domains respects the grading. For example,
for any two domains $B_1$, $B_2$ from $\xgen_0$ to $\xgen$, the cosets $\PP(\xgen_0)\cdot\gr(B_i)$ are the same.
Independence from the choice of $\xgen_0$ follows from the fact that $\PP(\xgen)$ is a conjugate of $\PP(\xgen_0)$.
\end{proof}

Fixing a grading reduction $r$, and setting
$\underline{\PP}(\xgen)=\grdn(\pi_2(\xgen,\xgen))=r(I_{o(\xgen)})\cdot\PP(\xgen)\cdot r(I_{o(\xgen)})^{-1}$,
we get a reduced
grading set $\Grdn(\HH,\s)$ with a right $\Grdn(\Z)$--action, and reduced grading $\grdn$ on $\G(\HH,\s)$ with the same
properties as $\gr$.

In light of the discussion in section~\ref{sec:reversal}, the sets $\Gr(\HH,\s)$ and $\Grdn(\HH,\s)$ have left actions
by $\Gr(-\Z)$ and $\Grdn(-\Z)$, respectively. Keep in mind that for the reduced grading,
the reduction term used for acting on $\grdn(\xgen)$ is $r(I_{\obar(\xgen)})$, corresponding to the complementary
idempotent of $\xgen$.

To define the grading on the bimodules $\BSDA$, we will need to take a mixed approach. Given a bordered sutured manifold
$(Y,\Gamma,\Z_1\cup\Z_2)$, thought of as a cobordism from $-\Z_1$ to $\Z_2$, we will use the left action of 
$\Gr(-\Z_1)\subset\Gr(-(\Z_1\cup\Z_2))$ and the right action of $\Gr(\Z_2)\subset\Gr(\Z_1\cup\Z_2)$. The two actions
commute since the correction term $L$ vanishes on mixed pairs. Moreover, the Maslov components act the same on both sides.

\subsection{A simpler description}
In the special case when $\Z=\varnothing$ and the manifold is just sutured, the grading takes a simpler form that is the
same as the usual relative grading on $\SFH$. Recall that the \emph{divisibility} of a $\spinc$--structure $\s$ is the
integer $\ddiv(\s)=\gcd_{\alpha\in H_2(Y)}\left<c_1(\s),\alpha\right>$, and that sutured Floer homology groups
$\SFH(Y,\s)$ are relatively-graded by the cyclic group $\ZZ/\ddiv(\s)$. (See~\cite{Juh:SFH}.)

\begin{thm}
\label{thm:grading_sfh}
Let $\HH$ be a Heegaard diagram for a sutured manifold $(Y,\Gamma)$, which can also be interpreted as a diagram for the
bordered sutured manifold $(Y,\Gamma,\varnothing)$. For any $\spinc$--structure $\s$, the grading sets are
$\Grdn(\HH,\s)=\Gr(\HH,\s)=\frac{1}{2}\ZZ/\ddiv(\s)$, with the usual action by
$\Grdn(\varnothing)=\Gr(\varnothing)=\frac{1}{2}\ZZ$. Moreover, the relative gradings $\gr=\grdn$ on $\G(\HH,\s)$ coincide
with relative grading on $\SFH$. In particular, only the integer gradings are occupied.
\end{thm}
\begin{proof}
The grading on $\SFH$ is defined in essentially the same way, on a diagram level. There a domain $B\in\pi_2(\xgen,\ygen)$
is graded as $-\ind(B)=-e(B)-n_{\xgen}-n_{\ygen}=\gr(B)\in\Gr(\varnothing)$. The rest of definition is exactly the same,
with the result that the gradings coincide, except that in the bordered sutured case we start with the bigger group
$\frac{1}{2}\ZZ$, while $\gr(B)=-\ind(B)$ still takes only integer values.
\end{proof}

In general, the grading sets $\Gr(\HH,\s)$ and $\Grdn(\HH,\s)$ can look very complicated, but if we forget some of the
structure we can give a reasonably nice description similar to the purely sutured case.

\begin{prop}
There is a projection map $\pi\co\Grdn(\HH,\s)\to\im(i_*\co H_1(F)\to H_1(Y))$ with the following properties. Each fiber
looks like $\frac{1}{2}\ZZ/\ddiv(\s)$, with the usual translation action by the central subgroup $(\frac{1}{2}\ZZ,0)$.
Any element of the form $(*,\alpha)$
permutes the fibers of $\pi$, sending $\pi^{-1}(\beta)$ to $\pi^{-1}(\beta+i_*(\alpha))$, while preserving
the $\frac{1}{2}\ZZ$--action.
\end{prop}
\begin{proof}
Recall that $\pi_2(\xgen,\xgen)$ is isomorphic to $H_2(Y,F)$ by attaching the cores of 2--handles and half-handles. Inside,
the subgroup $\pi_2^{\del}(\xgen,\xgen)$ of provincial periodic domains is isomorphic to $H_2(Y)\subset H_2(Y,F)$. 
Similar to the purely sutured case, for provincial periodic domains $e(B)+2n_\xgen(B)=\left<c_1(\s),[B]\right>$.
The subgroup
\begin{equation*}
\PP^\del(\xgen)=\gr(\pi_2^{\del}(\xgen,\xgen))=(\left<c_1(\s),H_2(Y)\right>,0)=(\ddiv(\s)\ZZ,0)
\end{equation*}
is central, and therefore $\underline{\PP}^\del(\xgen)=\grdn(\pi_2^{\del}(\xgen,\xgen))\subset\underline{\PP}(\xgen)$ is
also $(\ddiv(\s)\ZZ,0)$ and central in $\Grdn(\Z)$.

In particular, taking the quotient $\Grdn(\Z)/\underline{\PP}^\del$ has the effect of reducing the Maslov component modulo
$\ddiv(\s)$. On the other hand, since any two classes $B_{1,2}$ with the same $\del^\del$ differ by a provincial domain,
$\underline{\PP}/\underline{\PP}^\del$ is isomorphic to
$\im(\del\co H_2(Y,F)\to H_1(F))=\ker(i_*\co H_1(F)\to H_1(Y))$.
Ignoring the Maslov component, passing to
$(\underline{\PP}/\underline{\PP}^{\del})\backslash (\Grdn(\Z)/\underline{\PP}^{\del})=
\underline{\PP}\backslash \Grdn(\Z)$ reduces the
homological component $H_1(F)$ modulo $\ker i_*$. Therefore, the new homological component is valued in
$H_1(F)/\ker i_*\cong\im i_*$.
\end{proof}

\subsection{Grading and gluing}
The most important property of the reduced grading is that it behaves nicely under gluing of diagrams. This will later allow
us to show that the pairing on $\BSDA$ respects the grading.
First, we define a grading for a pair of diagrams which can be glued together, and then show it coincides with the grading
on the gluing.

Suppose $\HH_1$ and $\HH_2$ are diagrams for $(Y_1,\Gamma_1,-\Z_1\cup\Z_2)$ and $(Y_2,\Gamma_2,-\Z_2\cup\Z_3)$,
respectively, and fix reductions for $\Gr(\Z_1)$, $\Gr(\Z_2)$, and $\Gr(\Z_3)$. Recall that
$\Grdn(\HH_1,\s_1)$ has left and right actions by $\Grdn(\Z_1)$ and $\Grdn(\Z_2)$, respectively, while
$\Grdn(\HH_2,\s_2)$ has left and right actions by $\Grdn(\Z_2)$ and $\Grdn(\Z_3)$, respectively.

It is easy to see that generators in $\HH_1\cup_{\Zmid_2}\HH_2$ correspond to pairs of generators with complementary
idempotents at $\Z_2$, and there are restriction maps on $\spinc$--structures, such that
$\s(\xgen_1,\xgen_2)|_{Y_i}=\s(\xgen_i)$. Let $F_i=F(\Z_i)$.
From the long exact homology sequence for the triple
$(Y_1\cup Y_2,F_1\cup F_2\cup F_3,F_1\cup F_3)$ and Poincar\'e duality, we can see that $\{\s:\s|_{Y_i}=\s_i\}$ is either
empty or an affine set over $\im(i_*\co H_1(F_2)\to H_1(Y_1\cup Y_2,F_1\cup F_3))$.

\begin{defn}
Let $\Grdn(\HH_1,\HH_2,\s_1,\s_2)$ be the product
\begin{equation*}
\Grdn(\HH_1,\s_1)\times_{\Grdn(\Z_2)}\Grdn(\HH_2,\s_2),
\end{equation*}
i.e. the usual product of the two sets, modulo the relation $(a\cdot g,b)\sim(a,g\cdot b)$ for any $g\in\Grdn(\Z_2)$.
It inherits a left action by $\Grdn(\Z_1)$ and a right action by $\Grdn(\Z_3)$, which commute and where the Maslov
components act in the same way.

Define a grading on $\cup_{\s|_{Y_i}=\s_i}\G(\HH_1\cup\HH_2,\s)$ by
\begin{equation*}
\grdn'(\xgen_1,\xgen_2)=[(\grdn(\xgen_1),\grdn(\xgen_2)]\in\Grdn(\HH_1,\HH_2,\s_1,\s_2).
\end{equation*}
\end{defn}

\begin{thm}
\label{thm:grading_gluing}
Assume $\s_1$ and $\s_2$ are compatible, i.e. there is at least one $\s$ restricting to each of them.
There is a projection on the mixed grading set
\begin{equation*}
\pi\co\Grdn(\HH_1,\HH_2,\s_1,\s_2)\to\im(i_*\co H_1(F_2)\to H_1(Y_1\cup Y_2,F_1\cup F_3)),
\end{equation*}
defined up to a shift in the image, with the following properties.
\begin{enumerate}
\item\label{cond:distinguish}
For any two generators $\xgen$ and $\ygen$ with
$\s(\xgen)|_{Y_i}=\s(\ygen)|_{Y_i}=\s_i$, we have
\begin{equation*}
\PD(\s(\ygen)-\s(\xgen))=\pi(\grdn'(\xgen))-\pi(\grdn'(\ygen)),
\end{equation*}
i.e. $\pi$ distinguishes $\spinc$--structures. Moreover, the $\Grdn(\Z_1)$ and $\Grdn(\Z_3)$--actions preserve the fibers
of $\pi$.
\item\label{cond:same_grading}
For each $\s$, such that $\s|_{Y_i}=\s_i$, there is a unique fiber $\Grdn_{\s}$ of $\pi$, such that the grading
$\grdn'|_{\G(\HH_1\cup\HH_2,\s)}$ is valued in $\Grdn_{\s}$, and is equivalent to $\grdn$ valued in $\Grdn(\HH_1\cup\HH_2,\s)$.
\end{enumerate}
\end{thm}
\begin{proof}
It is useful to pass to only right actions, as the grading sets were originally defined.
Recall that $\Grdn(\HH_i,\s_i)$ are defined as the cosets
$\underline{\PP}(\xgen_i)\backslash\Grdn(-\Z_i\cup\Z_{i+1})$ for some
$\xgen_i\in\G(\HH_i,\s_i)$, for $i=1,2$. We define the mixed grading as the quotient of the product by the action of the
subgroup $H'=\{((a,\alpha),(-a,-\alpha))\}\cong\Grdn(\Z_2)$, for all $a\in\frac{1}{2}\ZZ$, $\alpha\in H_1(F_2)$
on the right.
Since $H'$ commutes with $\Grdn(-\Z_1)\times\Grdn(\Z_3)\subset\Grdn(-\Z_1\cup\Z_2)\times\Grdn(-\Z_2\cup\Z_4)$, we can
think of $\Grdn(\HH_1,\HH_2,\s_1,\s_2)$ as the cosets
\begin{equation*}
\PP'\backslash(\Grdn(-\Z_1\cup\Z_2)\times\Grdn(-\Z_2\cup\Z_3))/H',
\end{equation*}
with a right action by
$\Grdn(-\Z_1)\times\Grdn(\Z_3)$.
Here $\PP'=\underline{\PP}(\xgen_1)\times\underline{\PP}(\xgen_2)$.
Since the Maslov components behave nicely and commute with everything,
we can take a quotient by the Maslov component $\{((a,0),(-a,0))\}$ in $H'$. We get
\begin{equation*}
\Grdn(\HH_1,\HH_2,\s_1,\s_2)=\PP\backslash\Grdn(-\Z_1\cup\Z_2\cup-\Z_2\cup\Z_3))/H,
\end{equation*}
with a right action of $\Grdn(-\Z_1\cup\Z_3)$. Here 
$H=\{(0,\alpha,-\alpha):\alpha\in H_1(F_2)\}\subset\Grdn(\Z_2\cup-\Z_2)$, while the abelian group $\PP$ is generated by
the two subgroups
$\underline{\PP}(\xgen_i)$. In other words, the mixed grading set has elements of the form 
$[a]=\PP\cdot a\cdot H$, with action $[a]\cdot g=[a\cdot g]$.

Let $\Pi$ be addition of the $H_1(F_2)$--homological component terms together, and ignoring the rest. $H$ is in the kernel
of $\Pi$, while $\Pi|_{\PP}$ is the restriction of
$\Pi\circ\grdn'$ to
\begin{multline*}
\pi_2(\xgen_1,\xgen_1)\times\pi_2(\xgen_2,\xgen_2)\cong H_2(Y_1,F_1\cup F_2)\times H_2(Y_2,F_2\cup F_3)\\
\cong
H_2(Y_1\cup Y_2,F_1\cup F_2\cup F_3),
\end{multline*}
and
coincides with the boundary map
\begin{equation*}
\del\co H_2(Y_1\cup Y_2,F_1\cup F_2\cup F_3)\to H_1(F_1\cup F_2\cup F_3,F_1 \cup F_3)\cong H_1(F_2)
\end{equation*}
from the long exact sequence of the triple. Therefore
$\Pi(\PP)=\im\del=\ker(i_*\co H_1(F_2)\to H_1(Y_1\cup Y_2,F_1\cup F_3))$, and $\Pi$ descends on the cosets to a projection
$\pi$ with values in $H_1(F)/\ker i_*\cong\im i_*$. A different choice of $\xgen_i$ only shifts the homological components,
and so the image of $\pi$.

To prove~(\ref{cond:distinguish}), we need to check that for any compatible $\ygen_i\in\G(\HH_i,\s_i)$, the difference
$\s(\ygen_1,\ygen_2)-\s(\xgen_1,\xgen_2)$ is the same as $-(\pi(\grdn'(\ygen_1,\ygen_2))-\pi(\grdn'(\xgen_1,\xgen_2)))$.
Suppose $B_i\in\pi_2(\xgen_i,\ygen_i)$. Then the latter difference is $-\pi([\grdn(B_1),\grdn(B_2)])+\pi([0,0])
=-i_*(h_1+h_2)$, where $h_i$ is the $H_1(F_2)$ part of the homological component of $\grdn(B_i)$. Since the
$\Z_2$--idempotents of $\xgen_1$ and $\xgen_2$ are complementary, as well as those of $\ygen_1$ and $\ygen_2$, the reduction terms
$r$ cancel, and we can look at $\gr(B_1)$ and $\gr(B_2)$, instead.
Therefore $h_1+h_2=\del^{\del_2}B_1+\del^{\del_2}B_2$, interpreted
as an element of $H_1(F_2)\subset H_1(\Zmid_2,\amid_2)$. Here $\del^{\del_2}$ denotes the $\Z_2$ part of $\del^\del$.
It is indeed in that subgroup, again because of the
complementary idempotents.

By the proof of proposition~\ref{prop:spinc_diff}, we have $\s(\ygen)-\s(\xgen)=\PD([a-b])$, where $a$ and $b$
are any two 1--chains in $\balpha$ and $\bbeta$, with $\del a=\ygen-\xgen+\mathbf{z}$ and $\del b=\ygen-\xgen$,
where $\mathbf{z}$ is a
0--chain in $\Zmid_1\cup\Zmid_3$. The boundaries of $B_1$ and $B_2$ almost give us such chains. Let $a_i=\del^{\alpha}B_i$,
$b_i=\del^{\beta}B_i$, and $c_i=\del^{\del_2}B_i$, as chains. Then $[c_1+c_2]=h_1+h_2$, in $H_1(\Zmid_2,\amid_2)$, and
$[a_1+a_2+b_1+b_2+c_1+c_2]=-[\del^{\del_1}B_1+\del^{\del_3}B_2]=0\in H_1(Y_1\cup Y_2,F_1\cup F_3)$. Notice that we can
represent $h_1+h_2\in H_1(F_2)$ by the 1--chain $c_1+c_2+d$, where $d$ is a sum of some $\balpha^a$--arcs in $\HH_1$. Let
$a=a_1+a_2-d$, and $b=b_1+b_2$. They have the desired properties, and
$[a-b]=[a_1+a_2+b_1+b_2+c_1+c_2]-[c_1+c_2+d]=0-i_*(h_1+h_2)$. Thus $\pi$ does indeed distinguish $\spinc$--structures.
It is clear that the $\pi|_{\Grdn(-\Z_1\cup\Z_3)}=0$ and the action preserves the fibers.

For~(\ref{cond:same_grading}),
we know the restriction $\grdn'_{\G(\HH_1\cup\HH_2,\s)}$ takes values in a unique fiber $\Grdn_\s$. To see that the
grading is equivalent to $\grdn$, we need three results. First, we need to show that the action of
$\Grdn(-\Z_1\cup\Z_3)$ is transitive
on $\Grdn_\s$. Second, we need to show that the stabilizers of
$\grdn'(\ygen_0)$ and $\grdn(\ygen_0)$ are the same for some $\ygen_0\in\G(\HH_1\cup\HH_2,\s)$.
These two steps show that $\Grdn_\s$ and $\Grdn(\HH_1\cup\HH_2,\s)$ are equivalent as grading
sets. Finally, we need to show that for any other $\ygen\in\G(\HH_1\cup\HH_2,\s)$, there is at least one
$g\in\Grdn(-\Z_1\cup\Z_3)$, such that $\grdn(\ygen)=\grdn(\ygen_0)\cdot g$ and $\grdn'(\ygen)=\grdn'(\ygen_0)\cdot g$.

For the first part, notice that $\Grdn(-\Z_1\cup\Z_3)\times H$ is exactly the kernel of $\Pi$, while the reduction to
$\pi$ was exactly by the image of $\PP$. Therefore, if $\pi([a_1])=\pi([a_2])$, then $\Pi(a_1)=\Pi(p\cdot a_2)$ for some
$p\in\PP$, so $a_1=p\cdot a_2\cdot g\cdot h$ for some $g\in\Grdn(-\Z_1\cup\Z_3)$ and $h\in H$. In other words,
$[a_1]=[a_2]\cdot g$.

For the second part, we can assume $(\xgen_1,\xgen_2)$ are in $\s$, and use that as $\ygen_0$. In this case the stabilizer
for $\grdn'$
is $(\PP\cdot H)\cap\Grdn(-\Z_1\cup\Z_3)$. We may also assume the base idempotents for $r$ are
$I_{\obar_1(\xgen_1)}$ for $\Z_1$, $I_{o_2(\xgen_1)}=I_{\obar_2(\xgen_2)}$ for $\Z_2$, and $I_{o_3(\xgen_3)}$ for $\Z_3$.
This ensures $\grdn=\gr$ for periodic domains at $(\xgen_1,\xgen_2)$. This corresponds to the gradings of pairs
of periodic classes $B_i\in\pi_2(\xgen_i,\xgen_i)$ with $\del^{\del_2}B_1$+$\del^{\del_2}B_2=0$, canceling those terms.
But such pairs are in 1--to--1 correspondence with periodic class $B\in\pi_2((\xgen_1,\xgen_2),(\xgen_1,\xgen_2))$.
The gradings for such pairs are additive, so the stabilizer of $\grdn'$ is the same as $\grdn$.

Finally, to show the relative gradings are the same, pick any $B\in\pi_2(\ygen_0,\ygen)$. It decomposes into two classes
$B_i$ connecting the $\HH_i$ components of $\ygen_0$ and $\ygen$ with canceling $\del^{\del_2}$. Similar to the above
discussion, the regular gradings satisfy $\gr(B)=\gr(B_1)\gr(B_2)$. The reduction terms match up, so the same holds for
$\grdn$. Thus $\grdn(B)$ is the grading difference between $\ygen$ and $\ygen_0$ for both gradings.
\end{proof}

\section{One-sided invariants}
\label{sec:invariants}

\subsection{Overview}

To a bordered sutured manifold $(Y,\Gamma,\Z,\phi)$, we will associate the following invariants.
Each of them is defined up to homotopy equivalence, in the appropriate sense.
\begin{enumerate}
\item A right $\Ainf$--module over $\A(\Z)$, denoted
$\CSFA(Y,\Gamma,\Z,\phi)$.
\item A left \emph{type $D$ structure} over $\A(-\Z)$, denoted
$\CSFD(Y,\Gamma,\Z,\phi)$.
\item A left differential graded module over $\A(-\Z)$, which we denote
$\CSFDD(Y,\Gamma,\Z,\phi)$.
\end{enumerate}

\subsection{Type $D$ structures}
\label{sec:prelims}
Although we can express all of the invariants and their properties in terms of differential
graded modules and $\Ainf$--modules, from a practical standpoint it is more convenient to use
the language of type $D$ structures introduced in \cite{LOT:pairing}. We recall here the
definitions and basic properties. To simplify the discussion we
will restrict to the case where the algebra is
differential graded, and not a general $\Ainf$--algebra. This is all that is necessary
for the present applications. 

\begin{rmk} Any algebra or module has an implicit action
by a base ring, and any usual tensor product $\otimes$ is taken over such a base ring.
In the case of the algebra $\A(\Z)$ associated with an arc algebra, and any modules over it, the base ring is the
idempotent ring $\I(\Z)$. We will omit the base ring from the notation to avoid clutter.
\end{rmk}

Let $A$ be a differential graded algebra with differential $\mu_1$ and multiplication $\mu_2$.

\begin{defn} A \emph{(left) type $D$ structure} over $A$ is a graded module $N$ over the base ring,
with a homogeneous operation
\begin{equation*}
\delta\co N\to (A \otimes N)[1],
\end{equation*}
satisfying the compatibility condition
\begin{equation} \label{eq:type_d}
(\mu_1\otimes\id_N)\circ\delta + (\mu_2\otimes\id_N)\circ(\id_A\otimes\,\delta)\circ\delta=0.
\end{equation}
\end{defn}

We can define induced maps
\begin{equation*}
\delta_k\co N\to (A^{\otimes k} \otimes N)[k],
\end{equation*}
by setting
\begin{equation*}
\delta_k=\begin{cases}
\id_N & \textrm{for $k=0$,} \\
(\id_{A}\otimes\,\delta_{k-1})\circ \delta &\textrm{for $k\geq 1$.}
\end{cases}
\end{equation*}

\begin{defn}
We say a type $D$ structure $N$ is \emph{bounded} if for any $n\in N$,
$\delta_k(n)=0$
for sufficiently large $k$.
\end{defn}

Given two left type $D$ structures $N$, $N'$ over $A$, the homomorphism space $\Hom(N,A\otimes N')$,
graded by grading shifts, becomes a graded chain complex with differential
\begin{equation*}
Df=(\mu_1\otimes\id_{N'})\circ f + (\mu_2\otimes\id_{N'})\circ(\id_{A}\otimes\,\delta')\circ f
+(\mu_2\otimes\id_{N'})\circ(\id_{A}\otimes f)\circ \delta.
\end{equation*}

\begin{defn}
A \emph{type $D$ structure map} from $N$ to $N'$ is a cycle in the above chain complex.

Two such maps $f$ and $g$ are \emph{homotopic} if $f-g=Dh$ for some $h\in \Hom(N,A\otimes N')$, 
called a \emph{homotopy} from $f$ to $g$.

If $f\co N\to A\otimes N'$, and $g\co N'\to A\otimes N''$ are type $D$ structure maps, their
\emph{composition} $f\bcirc g\co N\to A\otimes N''$ is defined to be
\begin{equation*}
f\bcirc g=(\mu_2\otimes\id_{N''})\circ(\id_A\otimes\,g)\circ f.
\end{equation*}
\end{defn}

With the above definitions, type $D$ structures over $A$ form a differential graded category.
This allows us, among other things, to talk about homotopy equivalences.
(In general, for an $\Ainf$--algebra $A$, this is an $\Ainf$--category.)

~

Let $M$ be a right $\Ainf$--module over $A$, with higher $\Ainf$ actions
\begin{equation*}
m_k\co M\otimes A^{\otimes (k-1)}\to M[2-k]\textrm{, for $k\geq 1$.}
\end{equation*}

Let $N$ be a left type $D$ structure. We can define a special
tensor product between them.

\begin{defn}
Assuming at least one of $M$ and $N$ is bounded, let
\begin{equation*}
M\sqtens_{A} N
\end{equation*}
be the graded vector space $M\otimes N$, with differential
\begin{equation*}
\del\co M\otimes N\to (M\otimes N)[1],
\end{equation*}
defined by
\begin{equation*}
\del=\sum_{k=1}^{\infty}(m_k\otimes\id_N)\circ(\id_M\otimes\,\delta_{k-1}).
\end{equation*}
\end{defn}

The condition that $M$ or $N$ is bounded guarantees that the sum is always finite. In that case
$\del^2=0$ (using $\ZZ/2$ coefficients), and $M\sqtens N$ is a graded chain complex.

The most important property of $\sqtens$, as shown in \cite{LOT:bimodules} is that it is functorial up to
homotopy and induces a bifunctor on the level of derived categories.

The chain complex $A\sqtens N$ is in fact a graded differential module over $A$, with
differential
\begin{equation*}
\del=\mu_1\otimes\id_N+(\mu_2\otimes\id_N)\circ\delta,
\end{equation*}
and algebra action
\begin{equation*}
a\cdot(b,n)=(\mu_2(a,b),n).
\end{equation*}

In a certain sense working with type $D$ structures is equivalent to working with their associated left
modules. In particular, $A\sqtens\cdot$ is a functor, and $M\dtens(A\sqtens N)$ and $M\sqtens N$ are
homotopy equivalent as graded chain complexes.

\subsection{\texorpdfstring{$\CSFD$ and $\CSFDD$}{BSD and BSDM}}
Let $\HH=(\Sigma,\balpha,\bbeta,\Z,\psi)$ be a provincially admissible bordered sutured
Heegaard diagram, and let $J$ be an admissible almost complex structure.

We will define $\CSFD$ as a type $D$ structure over $\A(-\Z)$.

\begin{defn}
Fix a relative $\spinc$--structure $\s\in\spinc(Y,\del Y\setminus F)$.
Let $\CSFD(\HH,J,\s)$ be the $\ZZ/2$ vector space generated by the set of all generators $\G(\HH,\s)$.
Give it the structure of an $\I(-\Z)$ module as follows. For any $\xgen\in\G(\HH,\s)$ set
\begin{equation*}
I(s)\cdot\xgen=\begin{cases}
\xgen & \textrm{if $s=\obar(\xgen)$,}\\
0 & \textrm{oterwise.}
\end{cases}
\end{equation*}
\end{defn}

We consider only discrete partitions $\parrow=(\{q_1\},\ldots,\{q_m\})$.

\begin{defn} For $\xgen,\ygen\in\G(\HH)$ define
\begin{equation*}
a_{\xgen,\ygen}=\sum_{\substack{\ind(B,\brarrow(\parrow))=1 \\ \textrm{$\parrow$ discrete}}}
\#\M_{\emb}^B(\xgen,\ygen,S^\trr,\parrow)\cdot a(-P_1)\cdots a(-P_m).
\end{equation*}
\end{defn}

We compute $a(-P_i)$, since the Reeb chord $\rho_i$ labeling the puncture $q_i$ is
oriented opposite from $-\Z$.

\begin{defn}
Define $\delta\co \CSFD(\HH,J,\s)\to\A(-\Z)\otimes\CSFD(\HH,J,\s)$ as follows.
\begin{equation*}
\delta(\xgen)=\sum_{\ygen\in\G(\HH)}a_{\xgen,\ygen}\otimes\ygen.
\end{equation*}
\end{defn}

Note that, $\pi_2(\xgen,\ygen)$ is nonempty if and only if $\s(\xgen)=\s(\ygen)$,
so the range of $\delta$ is indeed correct.

\begin{thm}
\label{thm:bsd}
The following statements are true.
\begin{enumerate}
\item \label{cond:d_squared}
$\CSFD(\HH,J,\s)$ equipped with $\delta$, and the grading $\Grdn(\HH,\s)$--valued grading $\grdn$ is a type
$D$ structure over $\A(-\Z)$. In particular,
\begin{equation*}
\lambda^{-1}\cdot\grdn(\xgen)=\grdn(a)\cdot\grdn(\ygen),
\end{equation*}
whenever $\delta(\xgen)$ contains the term $a\otimes\ygen$.
\item \label{cond:bounded}
If $\HH$ is totally admissible, $\CSFD(\HH,J,\s)$ is bounded.
\item \label{cond:invariance}
For any two provincially admissible diagrams $\HH_1$ and $\HH_2$, equipped with 
admissible almost complex structures $J_1$ and $J_2$, there is a graded homotopy equivalence
\begin{equation*}
\CSFD(\HH_1,J_1,\s)\simeq\CSFD(\HH_2,J_2,\s).
\end{equation*}
Therefore we can talk about $\CSFD(Y,\Gamma,\Z,\psi,\s)$ or just $\CSFD(Y,\Gamma,\s)$, relatively graded by $\Grdn(Y,\s)$.
\end{enumerate}
\end{thm}
\begin{proof}

In light of the discussion in section~\ref{sec:degenerations}, the proofs carry over from 
those for $\CFD$ in the
bordered case. We sketch the main steps below.

For~(\ref{cond:d_squared}), first we use provincial admissibility to guarantee the
sums in the definitions are finite. Indeed, only finitely many provincial domains
$B\in\pi_2^{\del}(\xgen,\ygen)$ are positive and can contribute. The number of non provincial
domains ends up irrelevant, since only finitely many sequences of elements of $\A(-\Z)$ have
nonzero product.

To show that Eq.~(\ref{eq:type_d}) is satisfied,
we count possible degenerations of 1--dimensional moduli spaces, which are always
an even number.
Two story buildings correspond to the 
$(\mu_2\otimes\id_N)\circ(\id_\A\otimes\,\delta)\circ\delta$ term.
The degenerations with a curve at East infinity correspond to
the $(\mu_1\otimes\id_N)\circ\delta$ term.

To show that the grading condition is satisfied, recall that $a\otimes\ygen$ can be a term in $\delta(\xgen)$ only
if there is a domain $B\in\pi_2(\xgen,\ygen)$, and a compatible sequence $\brarrow=(\{\rho_1\},\ldots,\{\rho_p\})$ of
Reeb chords, such that $\ind(B,\brarrow)=1$, and
$a=I_{\obar(\xgen)}\cdot a(-\rho_1)\cdots a(-\rho_p)\cdot I_{\obar(\ygen)}$. We will prove the statement for $\gr$, which
allows us to ignore the idempotents at the end. The $\grdn$--version then follows from using the same reduction terms.

Notice that $\gr_{-\Z}(-\rho_i)=(-1/2,-[\rho_i])$, and so
\begin{equation*}
\gr(a)=(-p/2+\sum_{i<j}L_{-\Z}(-[\rho_i],-[\rho_j]),-[\brarrow]),
\end{equation*}
which we can also interpret as a $\Gr(\Z)$--grading
acting on the right. On the other hand, $\gr_{\Z}(\rho_i)=(-1/2,[\rho_i])$, and
\begin{equation*}
\gr(\brarrow)=(-p/2+\sum_{i<j}L_{\Z}([\rho_i],[\rho_j]),[\brarrow]).
\end{equation*}

Recall that $L_{\Z}$ and $L_{-\Z}$ have
opposite signs. Therefore, we have the relation $\gr(\brarrow)^{-1}=\lambda^{-p}\gr(a)$. Thus, we have
\begin{align*}
\gr(a\otimes\ygen)&=\gr_{-\Z}(a)\cdot\gr(\ygen)=\gr(\ygen)\cdot\gr_{\Z}(a)\\
&=\gr(\xgen)\cdot\gr(B)\gr(a)=\gr(\xgen)\cdot\lambda^{-\ind(B,\brarrow)+\#\brarrow}\gr(\brarrow)\gr(a)\\
&=\gr(\xgen)\cdot\lambda^{-1+p}\lambda^{-p}\gr(a)^{-1}\gr(a)=\gr(\xgen)\cdot\lambda^{-1}.
\end{align*}

For~(\ref{cond:bounded}), we use the fact that with full admissibility, only finitely
many domains $B\in\pi_2(\xgen,\ygen)$ are positive, and could contribute to $\delta_k$, for
any $k$. Therefore, only finitely many of the terms of $\delta_k(\xgen)$ are nonzero.

For~(\ref{cond:invariance}) we use the fact that provincially admissible diagrams can be
connected by Heegaard moves. To isotopies and changes of almost complex structure,
we associate moduli spaces, depending on a path $(\HH_t,J_t)$ of isotopic diagrams
and almost complex structures. Counting 0--dimensional spaces gives a type $D$ map
$\CSFD(\HH_0,J_0)\to\CSFD(\HH_1,J_1)$.
Analogous results to those in section~\ref{sec:moduli} and
counting the ends of 1--dimensional moduli spaces show that the map is well defined and
is in fact a homotopy equivalence. To handleslides, we associate maps coming from
counting holomorphic triangles, which also behave as necessary in this special case.

For invariance of the grading, we show that both in time-dependent moduli spaces,
and when counting triangles we can grade domains
compatibly. In particular, the stabilizers are still conjugate, and the grading set is preserved.
In both cases we count domains with index 0, so the relative gradings of individual elements are also preserved.
\end{proof}

If we ignore $\spinc$ structures we can talk about the total invariant
\begin{equation*}
\CSFD(Y,\Gamma)=\bigoplus_{\s\in\spinc(Y,\Gamma)}\CSFD(Y,\Gamma,\s).
\end{equation*}

We define $\CSFDD$ in terms of $\CSFD$. The two are essentially different algebraic
representations of the same object.

\begin{defn}
Given a bordered sutured manifold $(Y,\Gamma,\Z,\phi)$, let
\begin{align*}
\CSFDD(Y,\Gamma,\s)&=\A(-\Z)\sqtens\CSFD(Y,\Gamma,\s),\\
\CSFDD(Y,\Gamma)&=\A(-\Z)\sqtens\CSFD(Y,\Gamma).
\end{align*}
\end{defn}

\begin{rmk}
Recall that if $(Y,\Gamma)$ is $p$--unbalanced, then any generator has $p$ many occupied
arcs. However, for $\CSFD$ the algebra action depends on unoccupied arcs. Therefore, if $\Z$
has $k$ many arcs, then $\CSFD(Y,\Gamma)$ is in fact a type $D$ structure over $\A(-\Z,k-p)$ only.
\end{rmk}

\subsection{\texorpdfstring{$\CSFA$}{BSA}}
The definition of $\CSFA$ is similar to that of $\CSFD$, but differs in some important
aspects. In particular, we count a wider class of curves and they are recorded differently.

Let $\HH=(\Sigma,\balpha,\bbeta,\Z,\psi)$ be a provincially admissible bordered sutured
Heegaard diagram, and let $J$ be an admissible almost complex structure.

We define $\CSFA$ as an $\Ainf$--module over $\A(\Z)$.

\begin{defn}
Fix a relative $\spinc$--structure $\s\in\spinc(Y,\del Y\setminus F)$.
Let $\CSFA(\HH,J,\s)$ be the $\ZZ/2$ vector space generated by the set of all generators $\G(\HH,\s)$.
Give it the structure of an $\I(\Z)$ module by setting
\begin{equation*}
\xgen\cdot I(s) = \begin{cases}
\xgen & \textrm{if $s=o(\xgen)$,}\\
0 & \textrm{otherwise.}
\end{cases}
\end{equation*}
\end{defn}

For generators $\xgen,\ygen\in\G(\HH)$, a homology class $B\in\pi_2(\xgen,\ygen)$,
and a source $S^\trr$ we consider all partitions $\parrow=(P_1,\ldots,P_m)$,
not necessarily discrete. We also
associate to a sequence of Reeb chords a sequence of algebra elements, instead of a product. Let 
\begin{equation*}
\overrightarrow{a}(\xgen,\ygen,\brarrow)=
I(o(\xgen))\cdot(a(\brho_1)\otimes\cdots\otimes a(\brho_m))\cdot I(o(\ygen)).
\end{equation*}

\begin{defn} For $\xgen,\ygen\in\G(\HH)$, $\brarrow=(\brho_1,\ldots,\brho_m)$ define
\begin{equation*}
c_{\xgen,\ygen,\brarrow}=\sum_{\substack{\brarrow(\parrow)=\brarrow \\ \ind(B,\brarrow)=1}}
\#\M_{\emb}^B(\xgen,\ygen,S^\trr,\parrow). 
\end{equation*}
\end{defn}

\begin{defn}
Define $m_k\co \CSFA(\HH,J,\s)\otimes\A(\Z)^{\otimes(k-1)}\to\CSFA(\HH,J,\s)$ as follows.
\begin{equation*}
m_k(\xgen,a_1,\ldots,a_{k-1})=
\sum_{\substack{\ygen\in\G(\HH)\\
\overrightarrow{a}(\xgen,\ygen,\brarrow)=a_1\otimes\cdots\otimes a_{k-1}}}
c_{\xgen,\ygen,\brarrow}\cdot \ygen.
\end{equation*}
\end{defn}

\begin{thm}
\label{thm:bsa}
The following statements are true.
\begin{enumerate}
\item $\CSFA(\HH,J,\s)$ equipped with the actions $m_k$ for $k\geq 1$, and the $\Grdn(\HH,\s)$--valued grading $\grdn$
is an $\Ainf$--module over $\A(\Z)$. In particular,
\begin{equation*}
\grdn(m_k(\xgen,a_1,\ldots,a_{k-1}))
=\grdn(\xgen)\cdot\grdn(a_1)\cdots\grdn(a_{k-1})\lambda^{k-2}.
\end{equation*}
\item If $\HH$ is totally admissible, $\CSFA(\HH,J,\s)$ is bounded.
\item For any two provincially admissible diagrams $\HH_1$ and $\HH_2$, equipped with 
admissible almost complex structures $J_1$ and $J_2$, there is a graded homotopy equivalence
\begin{equation*}
\CSFA(\HH_1,J_1,\s)\simeq\CSFA(\HH_2,J_2,\s).
\end{equation*}
Therefore we can talk about $\CSFA(Y,\Gamma,\Z,\psi,\s)$ or just $\CSFA(Y,\Gamma,\s)$, relatively graded by
$\Grdn(Y,\s)$.
\end{enumerate}
\end{thm}
\begin{proof}
The proofs are analogous to those for $\CSFD$, with some differences. The biggest
difference is that we count more domains, so we need to use more results about
degenerations.

The other major difference is the grading.
Again, we prove the statement for $\gr$, and the one for $\grdn$ follows immediately.
Suppose $\ygen$ is a term in $m_k(\xgen,a_1,\ldots,a_{k-1})$. Then there is a domain $B\in\pi_2(\xgen,\ygen)$ and
a compatible sequence $\brarrow=(\brho_1,\ldots,\brho_{k-1})$ of sets of Reeb chords, such that
$\ind(B,\brarrow)=1$, and $a_i=a(\brho_i,s_i)$, for some appropriate completion $s_i$. In particular,
$\gr(a_1)\cdots\gr(a_{k-1})=\gr(\brarrow)$. On the other hand,
\begin{multline*}
\gr(\ygen)=\gr(\xgen)\cdot\gr(B)=\gr(\xgen)\cdot\lambda^{-\ind(B,\brarrow)+\#\brarrow}\gr(\brarrow)\\
=\gr(\xgen)\cdot\lambda^{-1+(k-1)}\gr(a_1)\cdots\gr(a_{k-1}).\qedhere
\end{multline*}
\end{proof}

As with $\CSFD$, if we ignore $\spinc$--structures we can talk about the total invariant
\begin{equation*}
\CSFA(Y,\Gamma)=\bigoplus_{\s\in\spinc(Y,\Gamma)}\CSFA(Y,\Gamma,\s).
\end{equation*}

\begin{rmk}
As with $\CSFD$ the only nontrivial algebra action is by a single component of $\A(\Z)$.
In this case the action depends on occupied arcs.
Therefore if $(Y,\Gamma)$ is $p$--unbalanced,
then $\CSFA(Y,\Gamma)$ is an $\Ainf$--module over $\A(\Z,p)$ only.
\end{rmk}

\subsection{Invariants from nice diagrams}
For a nice diagram $\HH$, the invariants can be computed completely combinatorially,
avoiding all discussion of moduli spaces.

\begin{thm}
\label{thm:nice_d}
Let $\HH$ be a nice diagram. Then for any admissible almost complex structure $J$,
the type $D$ structure $\CSFD(\HH,J)$ can be computed as follows.
The map $\delta(\xgen)$ counts the following types of curves.
\begin{enumerate}
\item 
\label{cond:empty_bigon}
A source $S^\trr$ from $\xgen$ to $\ygen$, consisting of $g$ bigons with no $e$ punctures,
where all but one of the bigons are constant on $\Sigma$, while the remaining one embeds
as a convex bigon. The interior of the image contains none of the fixed points $\xgen\cap\ygen$.
Such curves contribute $I(\obar(\xgen))\otimes\ygen$ to $\delta(\xgen)$.

\item
\label{cond:empty_rectangle}
A source $S^\trr$ from $\xgen$ to $\ygen$, consisting of $g-2$ bigons, each of which has no $e$ punctures and is constant
on $\Sigma$, and a single quadrilateral with no $e$ punctures, which embeds as a convex rectangle.
The interior of the image contains none of the fixed points $\xgen\cap\ygen$.
Such curves contribute $I(\obar(\xgen))\otimes\ygen$ to $\delta(\xgen)$.

\item
\label{cond:bdy_rect}
A source $S^\trr$ from $\xgen$ to $\ygen$, consisting of $g-1$ bigons, each of which has no $e$ punctures and is constant
on $\Sigma$, and a single bigon with one $e$ punctures, which embeds as a convex rectangle, one of whose
sides is the Reeb chord $-\rho\in\Zmid$ labeling the puncture.
The interior of the image contains none of the fixed points $\xgen\cap\ygen$.
Such curves contribute $I(\obar(\xgen))a(\rho)I(\obar(\ygen))\otimes\ygen$ to $\delta(\xgen)$.
\end{enumerate}
\end{thm}

\begin{thm}
\label{thm:nice_a}
Let $\HH$ be a nice diagram. Then for any admissible almost complex structure $J$,
the $\CSFA(\HH,J)$ can be computed as follows.

The differential $m_1(\xgen)$ counts the following types of regions. (These are the same as
cases (\ref{cond:empty_bigon}) and (\ref{cond:empty_rectangle}) in theorem \ref{thm:nice_d}.)
\begin{enumerate}
\item A source $S^\trr$ from $\xgen$ to $\ygen$, consisting of $g$ bigons with no $e$ punctures,
where all but one of the bigons are constant on $\Sigma$, while the remaining one embeds
as a convex bigon. The interior of the image contains none of the fixed points $\xgen\cap\ygen$.
Such curves contribute $\ygen$ to $m_1(\xgen)$.

\item 
A source $S^\trr$ from $\xgen$ to $\ygen$, consisting of $g-2$ bigons, each of which has no $e$ punctures and is constant
on $\Sigma$, and a single quadrilateral with no $e$ punctures, which embeds as a convex rectangle.
The interior of the image contains none of the fixed points $\xgen\cap\ygen$.
\end{enumerate}

The algebra action $m_2(\xgen,\cdot\,)$ counts regions of the type below.
\begin{enumerate}
\item A source $S^\trr$ from $\xgen$ to $\ygen$, consisting of $g-p$ bigons, each of which has no $e$ punctures
and is constant
on $\Sigma$, and a collection of $p$ bigons, each of which has one $e$ puncture and which embeds as a convex rectangle,
one of whose sides is the Reeb chord $\rho_i\in\Zmid$. The height of all $e$ punctures is the same,
the interior of any image rectangle contains none of the fixed points $\xgen\cap\ygen$ and no other rectangles.
Such curves contribute $\ygen$ to the action $m_2(\xgen,I(o(\xgen))
\{a(\rho_1,\ldots,\rho_p\})I(o(\ygen))$.
\end{enumerate}

In addition, all actions $m_k$ for $k\geq 3$ are zero.

\end{thm}

\begin{proof}[Proof of theorems \ref{thm:nice_d} and \ref{thm:nice_a}]
The proofs follow the same steps as the ones for nice diagrams in bordered manifolds.
By looking at the index formula, and the restricted class of regions, one can
show that the only $B$, $S^\trr$, and $\parrow$ that have index
$\ind(B,S^\trr,\parrow)=1$ are of the following two types.
\begin{enumerate}
\item $S$ has no $e$ punctures, and consists of $g-1$ trivial components,
and one non-trivial bigon component, or $g-2$ trivial component and one non-trivial
rectangle component.
\item $S$ has several trivial components, and several bigons with a single $e$ puncture
each. Moreover, the partition $\parrow$ consists of only one set.
\end{enumerate}

The extra condition that the embedded index is also 1 (so the moduli space consists
of embedded curves), is equivalent to having no fixed points in the interior of a region,
and no region contained completely inside another.

For such curves, $\M_{\emb}^B(\xgen,\ygen,S^\trr,\parrow)$ has exactly one element, independent of the
almost complex structure $J$, using for example the Riemann mapping theorem.
\end{proof}

\subsection{Pairing theorem}
\label{sec:pairing}
In this section we describe the relationship between the sutured homology
of the gluing of two bordered sutured manifolds, and their bordered sutured invariants,
proving the second part of theorem~\ref{thm:intro_pairing}.

Recall that bordered sutured invariants are homotopy types of chain complexes, while
sutured Floer homology is usually regarded as an isomorphism type of homology groups.
However, one can also regard the underlying chain complex as an invariant up to
homotopy equivalence. To be precise, we will use $\SFH$ to denote sutured Floer
homology, and $\SFC$ to denote a representative chain complex defining that
homology.

\begin{thm}
\label{thm:pairing}
Suppose $(Y_1,\Gamma_1,\Z,\phi_1)$ and $(Y_2,\Gamma_2,-\Z,\phi_2)$ are two
bordered sutured manifolds that glue along $F=F(\Z)$ to form the sutured
manifold $(Y,\Gamma)$. Let $\s_i\in\spinc(Y_i,\del Y_i\setminus F)$ be relative
$\spinc$--structures for $i=1,2$. Then there is a graded chain homotopy equivalence
\begin{equation*}
\bigoplus_{\s|_{Y_i}=\s_i}\SFC(Y,\Gamma,\s)\simeq\CSFA(Y_1,\Gamma_1,\s_1)
\sqtens_{\A(\Z)}\CSFD(Y_2,\Gamma_2,\s_2),
\end{equation*}
provided that at least one of the modules on the right hand-side comes from an
admissible diagram.

To identify the gradings, we use the fact that the combined grading set
$\Grdn(Y_1\s_1)\times_{\Grdn(\Z)}\Grdn(Y_2,\s_2)$
distinguishes the individual $\spinc$--structures $\s\in\spinc(Y,\del Y)$ by its
homological component, while the Maslov component agrees with the $\SFH$ grading on each $\SFC(Y,\s)$.
\end{thm}

\begin{cor}
\label{cor:pairing}
In terms of modules and derived tensor products, the pairing theorem can be expressed
as
\begin{equation*}
\bigoplus_{\s|_{Y_i}=\s_i}\SFC(Y,\Gamma,\s)\simeq\CSFA(Y_1,\Gamma_1,\s_1)
\dtens\CSFDD(Y_2,\Gamma_2,\s_2).
\end{equation*}
\end{cor}

Corollary~\ref{cor:pairing} is a restatement of theorem~\ref{thm:pairing} in purely
$\Ainf$--module language. This allows us to dispose of type $D$ structures entirely.
However, in practice, the definition of the derived tensor product $\dtens$ involves
an infinitely generated chain complex, while that of $\sqtens$ only a finitely generated chain complex (assuming
both sides are finitely generated).

\begin{proof}[Proof of theorem~\ref{thm:pairing}]
We can prove the theorem using nice diagrams, similar to \cite[Chapter 8]{LOT:pairing}.

Suppose $\HH_1$ and $\HH_2$ are nice diagrams
for $Y_1$ and $Y_2$, respectively. If we glue them to get a diagram $\HH=\HH_1\cup_{\Zmid}\HH_2$
for $Y=Y_1\cup_F Y_2$, then $\HH$ is also a nice diagram.
Indeed, the only regions that change are boundary regions, which are irrelevant, and regions
adjacent to a Reeb chord. In the latter case, two rectangular regions in $\HH_1$ and $\HH_2$, that
border the same Reeb chord, glue to a single rectangular region in $\HH$.

Generators in $\G(\HH)$ correspond to pairs of generators in $\G(\HH_1)$ and $\G(\HH_2)$ that occupy
complementary sets of arcs. Provincial bigons and rectangles in $\HH_i$ are also bigons and rectangles
in $\HH$. The only other regions in $\HH$ that contribute to the differential $\del$ on $\SFC$
are rectangles that are split into two rectangles in $\HH_1$ and $\HH_2$, each of which is adjacent to
the same Reeb chord $\rho$ in $\Zmid$. Such rectangles contribute terms of
the form $(m_2\otimes\id_{\CSFD(\HH_2)})\circ(\id_{\CSFA(\HH_1)}\otimes\,\delta)$. Overall, terms in
$\del:\SFC(\HH)\to \SFC(\HH)$ are in a one-to-one correspondence with terms in
$\del:\CSFA(\HH_1)\sqtens\CSFD(\HH_2)\to\CSFA(\HH_1)\sqtens\CSFD(\HH_2)$.

This shows that there is an isomorphism of chain complexes
\begin{equation*}
\SFC(\HH)\cong \CSFA(\HH_1)\sqtens\CSFD(\HH_2).
\end{equation*}

The splitting into $\spinc$--structure and the grading equivalence follow from theorems~\ref{thm:grading_sfh}
and~\ref{thm:grading_gluing}, where the latter is applied to $(Y_1,\Gamma_1,-\varnothing\cup\Z)$ and
$(Y_2,\Gamma_2,-\Z\cup\varnothing)$.
\end{proof}

\section{Bimodule invariants}
\label{sec:bimodules}

As promised in the introduction, we will associate to a decorated sutured cobordism, a special type of
$\Ainf$--bimodule. We will sketch the construction, which closely parallels the discussion of bimodules
in \cite{LOT:bimodules}. The reader is encouraged to look there, especially for a careful discussion of
the algebra involved.

\subsection{Algebraic preliminaries}
The invariants we will define have the form of \emph{type $\DA$ structures}, which is a combination of
a type $D$ structure and an $\Ainf$--module.

\begin{defn}
Let $A$ and $B$ be differential graded algebras with differential and multiplication denoted
$\del_A$, $\del_B$, $\mu_A$, and $\mu_B$, respectively.
A \emph{type $DA$ structure} over $A$ and $B$ is a graded vector space $M$, together
with a collection of homogeneous operations $m_k\co M\otimes B^{\otimes (k-1)}\to A\otimes M[2-k]$, satisfying the
compatibility condition
\begin{multline*}
\sum_{p=1}^{k}(\mu_A\otimes\id_M)\circ(\id_A\otimes\,m_{k-p+1})\circ(m_p\otimes\id_{B^{\otimes(k-p)}})
+(\del_A\otimes\id_M)\circ m_k\\
+\sum_{p=0}^{k-2}m_k\circ(\id_M\otimes\id_{B^{\otimes p}}\otimes\,\del_B\otimes\id_{B^{\otimes(k-p-2)}})\\
+\sum_{p=0}^{k-3}m_k\circ(\id_M\otimes\id_{B^{\otimes p}}\otimes\,\mu_B\otimes\id_{B^{\otimes(k-p-3)}})=0,
\end{multline*}
for all $k\geq0$.
\end{defn}

We can also define $m_k^i\co M\otimes B^{\otimes(k-1)}\to A^{\otimes i}\otimes M[1+i-k]$,
such that $m_1^0=\id_M$, $m_k^0=0$ for $k>1$, $m_k^1=m_k$, and $m_k^i$ is obtained by iterating $m_*^1$:
\begin{equation*}
m_k^i=\sum_{j=0}^{k-1}(\id_{A^{\otimes(i-1)}}\otimes\, m_{j+1})\circ (m_{k-j}^{i-1}\otimes\id_{B^{\otimes\, j}}).
\end{equation*}

In the special case where $A$ is the trivial algebra $\{1\}$,
this is exactly the definition of a right $\Ainf$--module over $B$. In the case when $B$ is trivial, or we ignore
$m_k^i$ for $k\geq 2$, this is exactly the definition of a left type $D$ structure over $A$. In that case
$m_1^i$ corresponds to $\delta_i$.

We will use some notation from \cite{LOT:bimodules} and denote a type $\DA$ structure over $A$ and $B$ by
$\lu{A}M_B$. In the same vein, a type $D$ structure over $A$ is $\lu{A}M$, and a right $\Ainf$--module over $B$
is $M_B$. We can extend the tensor $\sqtens$ to type $\DA$ structures as follows.

\begin{defn}
Let $\lu{A} M_B$ and $\lu{B} N_C$ be two type $\DA$ structures, with operations $m_k^i$, and $n_l^j$, respectively.
Let $\lu{A}M_B\,\sqtens_B\mbox{}\lu{B}N_C$ denote the type $\DA$ structure $\lu{A}(M\otimes N)_C$, with operations
\begin{equation*}
(m\sqtens n)_k^i=\sum_{j\geq 1}(m^i_j\otimes\id_N)\circ(\id_M\otimes\,n_k^{j-1}).
\end{equation*}
\end{defn}

In the case when $A$ and $C$ are both trivial, this coincides with the standard operation
$M_B\sqtens\mbox{}\lu{B}N$.

The constructions generalize to mixed multi-modules of type
$\li{^{A_1,\ldots,A_i}_{B_1,\ldots,B_j}}M^{C_1,\ldots,C_k}_{D_1,\ldots,D_l}$. Such a module is left, respectively right,
type $D$ with respect to $A_p$, respectively $C_p$, and left, respectively right $\Ainf$--module with respect to
$B_p$, respectively $D_p$. 
The category of such modules is denoted
$\li{^{A_1,\ldots,A_i}_{B_1,\ldots,B_j}}\Mod^{C_1,\ldots,C_k}_{D_1,\ldots,D_l}$.
We can apply the tensor $\sqtens_X$ to any pair of such modules, as long as one of
them has $X$ as an upper (lower) right index, and the other has $X$ as a lower (upper) left index.

We will only use a few special cases of this construction. The most important one is to
associate to $\lu{A}M_B$ a canonical $A,B$ $\Ainf$--bimodule ${}_A(A\sqtens M)_B={}_A A_A\,\sqtens_A\mbox{}\lu{A}M_B$.
This allows us to bypass type $D$ and type $\DA$ structures. In particular,
\begin{equation*}
{}_A(A\sqtens M)_B\,\dtens_B\,\mbox{}_B(B\sqtens N)_C\simeq {}_A(A\sqtens (M\sqtens_B N))_C.
\end{equation*}

\subsection{\texorpdfstring{$\CSFD$ and $\CSFA$}{BSD and BSA} revisited}
Recall that the definition of $\CSFD$ counted a subset of the moduli spaces used to define $\CSFA$, and
interpreted them differently. This operation can in fact be described completely algebraically. For any arc
diagram $\Z$, there is a bimodule (or \emph{type $\tDD$ structure}) ${}^{\A(-\Z),\A(\Z)}\II$, such that
\begin{equation*}
{}^{\A(-\Z)}\CSFD(\HH,J)=\CSFA(\HH,J)_{\A(\Z)}\,\sqtens_{\A(\Z)}\,\mbox{}^{\A(-\Z),\A(\Z)}\II.
\end{equation*}

In fact, we could use this as the definition of $\CSFD$, and use the naturallity of $\sqtens$ to
prove that it is well-defined for $\HH$ and $J$, and its homotopy type is an invariant of the underling
bordered sutured manifold.

\subsection{Bimodule categories}
For two differential graded algebras $A$ and $B$, the notion of a left-left $A,B$--module is exactly the
same as that of a left $A\otimes B$--module. Similarly, a left type $D$ structure over $A$ and $B$ is exactly
the same as a left type $D$ structure over $A\otimes B$. In other words, we can interpret a module
$\lu{A,B}M$ as $\lu{A\otimes B}M$, and vice versa, and the categories $\lu{A,B}\Mod$ and
$\lu{A\otimes B}\Mod$ are canonically identified.

The situation is not as simple for $\Ainf$--modules.
The categories $\Mod_{A,B}$ and $\Mod_{A\otimes B}$ are \emph{not} the same, or even equivalent.
Fortunately, there is a canonical functor $\F\co \Mod_{A\otimes B}\to\Mod_{A,B}$ which induces an equivalence of the
derived categories. For this result, and the precise definition of $\F$ see \cite{LOT:bimodules}.

\subsection{\texorpdfstring{$\BSDA$ and $\BSDAM$}{BSDA and BSDAM}}
We will give two definitions of the bimodules. One is purely algebraic, and allows us to easily deduce that
the bimodules are well-defined and invariant, while the other is more analytic, but is more useful in practice.

\begin{defn}
\label{def:DA_algebraic}
Suppose $(Y,\Gamma,\Z_1\cup\Z_2,\phi)$ is a bordered sutured manifold---or equivalently, a decorated
sutured cobordism from $F(-\Z_1)$ to $F(\Z_2)$. Note that $\A(\Z_1\cup \Z_2)=\A(\Z_1)\otimes\A(\Z_2)$.
Define
\begin{multline*}
{}^{\A(-\Z_1)}\BSDA(Y,\Gamma,\s)_{\A(\Z_2)}=\\
\F(\CSFA(Y,\Gamma,\s))_{\A(\Z_1),\A(\Z_2)}
\,\sqtens_{\A(\Z_1)}\, {}^{\A(-\Z_1),\A(\Z_1)}\II.
\end{multline*}
\end{defn}

The invariance follows easily from the corresponding results for $\CSFA$ and naturallity.

As promised, below we give a more practical construction. Fix a provincially admissible diagram
$\HH=(\Sigma,\balpha,\bbeta,\Z_1\cup\Z_2,\psi)$, and an admissible almost complex structure $J$.

Recall that to define both $\CSFD$ and $\CSFA$, we looked at moduli
spaces $\M_{\emb}^B(\xgen,\ygen,S^\trr,\parrow)$, where $\parrow$ is a partition of the $e$ punctures
on the source $S^\trr$. In our case, we can
distinguish two sets of $e$ punctures---those labeled by Reeb chords in $\Z_1$, and those labeled by Reeb
chords in $\Z_2$. We denote
the two sets by $E_1$ and $E_2$, respectively. Any partition $\parrow$ restricts to partitions
$\parrow_i=\parrow|_{E_i}$ on the two sets.

\begin{defn} Define the moduli space
\begin{equation*}
\M_{\emb}^B(\xgen,\ygen,S^\trr,\parrow_1,\parrow_2)=\bigcup_{\parrow|_{E_i}=\parrow_i}
\M_{\emb}^B(\xgen,\ygen,S^\trr,\parrow),
\end{equation*}
with index
\begin{equation*}
\begin{split}
\ind(B,\brarrow_1,\brarrow_2)&=e(B)+n_x(B)+n_y(B)\\
&\quad{}+\#\brarrow_1+\#\brarrow_2+\iota(\brarrow_1)+\iota(\brarrow_2),
\end{split}
\end{equation*}
where $\brarrow_i=\brarrow(\parrow_i)$ is a sequence of sets of Reeb chords in $\Z_i$, for $i=1,2$.
\end{defn}

This has the effect of forgetting about the relative height of punctures in $E_1$ to those in $E_2$. Its algebraic
analogue is applying the functor $\F$, which combines the algebra actions
$m_3(x,a\otimes 1,1\otimes b)$ and $m_3(x,1\otimes b,a\otimes 1)$ into $m_{1,1,1}(x,a,b)$.

The general idea is to treat the $\Z_1$ part of the arc diagram as in $\CSFD$, and the $\Z_2$ part as in $\CSFA$.
First, to a generator $\xgen\in\G(\HH)$ we associate idempotents $I_1(\obar(\xgen))\in\I(-\Z_1)$ and
$I_2(o(\xgen))\in\I(\Z_2))$, corresponding to unoccupied arcs on the $\Z_1$ side, and occupied arcs on the
$\Z_2$ side, respectively.
Next, we will look at discrete partitions $\parrow_1=(\{q_1\},\ldots,\{q_i\})$ on the $\Z_1$ side,
while allowing arbitrary partitions $\parrow_2$ on the $\Z_2$ side.

If the punctures in $\parrow_1$ are labeled by the Reeb chords $(\rho_1,\ldots,\rho_i)$, set
\begin{equation*}
a_1(\xgen,\ygen,\parrow_1)=I_1(\obar(\xgen))\cdot a(-\rho_1)\cdots a(-\rho_i)\cdot I_1(\obar(\ygen))\in\A(-\Z_1).
\end{equation*}

If the sets of punctures in $\parrow_2$ are labeled by a sequence of sets of Reeb chords $(\brho_1,\ldots,\brho_j)$, set
\begin{equation*}
a_2(\xgen,\ygen,\parrow_2)=I_2(\xgen)\cdot a(\brho_1)\otimes\cdots\otimes a(\brho_j) \cdot I_2(\ygen)\in\A(\Z_2)^{\otimes j}.
\end{equation*}

\begin{defn}
\label{def:DA_analytic}
Fix $\HH$, $J$, and $\s$. Let $\BSDA(\HH,J,\s)$ be freely generated over $\ZZ/2$ by $\G(\HH,\s)$, with
$\I(-\Z_1)$ and $\I(\Z_2)$ actions
\begin{equation*}
I(s_1)\cdot\xgen\cdot I(s_2)=\begin{cases}
\xgen & \textrm{if $s_1=\obar(\xgen)$ and $s_2=o(\xgen)$,} \\
0 & \textrm{otherwise.}
\end{cases}
\end{equation*}

It has type $DA$ operations
\begin{multline*}
m_k(\xgen,b_1,\ldots,b_{k-1})=\\
\sum_{\substack{
\ind(B,\brarrow(\parrow_1),\brarrow(\parrow_2))=1\\
a_2(\xgen,\ygen,\parrow_2)=b_1\otimes\cdots\otimes b_{k-1}}}
\#\M_{\emb}^B(\xgen,\ygen,S^\trr,\parrow_1,\parrow_2)\cdot
a_1(\xgen,\ygen,\parrow_1)\otimes\ygen.
\end{multline*}
\end{defn}

It is easy to check that definitions~\ref{def:DA_analytic} and~\ref{def:DA_algebraic} are equivalent.
Considering pairs of partitions corresponds to the functor $\F$, while restricting to discrete partitions on
the $\Z_1$ side and multiplying the corresponding Reeb chords corresponds to the functor
$\cdot\,\sqtens_{\A(\Z_1)}\lu{\A(-\Z_1),\A(\Z_1)}\,\II$.

We can use either definition to define
\begin{multline*}
{}_{\A(-\Z_1)}\BSDAM(Y,\Gamma,\s)_{\A(\Z_2)}=\\
{}_{\A(-\Z_1)}\A(-\Z_1)_{\A(-\Z_1)}\,\sqtens_{\A(-\Z_1)}\,{}^{\A(-\Z_1)}\BSDA(Y,\Gamma,\s)_{\A(\Z_2)}.
\end{multline*}

As with the one-sided modules, there is a well-defined grading.

\begin{thm}
The grading $\grdn$ on $\BSDA(Y,\Gamma,\s)$ is well-defined with values in $\Grdn(Y,s)$, and makes it a graded
$\DA$--structure. In particular, whenever $b\otimes\ygen$ is a term in $m_k(\xgen,a_1,\ldots,a_{k-1})$, we have
\begin{equation*}
\grdn(b)\cdot\grdn(\ygen)=\lambda^{k-2}\cdot\grdn(\xgen)\cdot\grdn(a_1)\cdots\grdn(a_{k-1}).
\end{equation*}
\end{thm}
\begin{proof}
This is a straightforward combination of the arguments for the gradings on $\BSD$ and $\BSA$.
\end{proof}

\subsection{Nice diagrams and pairing}
The key results for bimodules allowing us to talk about a functor from the decorated sutured cobordism category
$\sdcat$ to the category $\D$ of differential graded algebras and $\Ainf$--bimodules
are the full version of theorem~\ref{thm:intro_pairing},
and theorem~\ref{thm:equivalence}. Below we give a more precise version of theorem~\ref{thm:intro_pairing}, in the
vein of theorem~\ref{thm:pairing}.

\begin{thm}
\label{thm:pairing_bimodules}
Given two bordered sutured manifolds $(Y_1,\Gamma_1,-\Z_1\cup\Z_2)$ and $(Y_2,\Gamma_2,-\Z_2\cup\Z_3)$, 
representing cobordisms from $\Z_1$ to $\Z_2$ and from $\Z_2$ to $\Z_3$, respectively, there are 
graded homotopy equivalences of bimodules
\begin{align*}
\bigoplus_{\s|Y_i=\s_i}\BSDA(Y_1\cup Y_2,\s)&
\simeq \BSDA(Y_1,\s_1)\,\sqtens_{\A(\Z_2)}\BSDA(Y_2,\s_2).\\
\bigoplus_{\s|Y_i=\s_i}\,\BSDAM(Y_1\cup Y_2,\s)&
\simeq \BSDAM(Y_1,\s_1)\dtens_{\A(\Z_2)}\BSDAM(Y_2,\s_2).
\end{align*}

The gradings are identified in the sense of theorem~\ref{thm:grading_gluing}.
\end{thm}

The proof is completely analogous to that of theorem~\ref{thm:pairing}. It relies on the combinatorial form of
$\BSDA$ from a nice diagram, and the fact that gluing two such diagrams also gives a nice diagram with direct
correspondence of domains. The actual result for nice diagrams is given below.

\begin{thm}
\label{thm:nice_da}
For any nice diagram $\HH=(\Sigma,\balpha,\bbeta,\Z_1\cup\Z_2,\psi)$ and any admissible almost complex structure $J$
the domains that contribute to $m_k$ are of the following types.
\begin{enumerate}
\item Provincial bigons and rectangles, which contribute terms of the form $I\otimes \ygen$ to $m_1(\xgen)$.
\item \label{cond:z1}
Rectangles hitting a Reeb chord at $\Z_1$, which contribute terms of the form
$a\otimes\ygen$ to $m_1(\xgen)$.
\item \label{cond:z2}
Collections of rectangles hitting Reeb chords at $\Z_2$, at the same height, which contribute terms of the form
$I\otimes\ygen$ to $m_2(\xgen,\ldots)$.
\end{enumerate}
\end{thm}
\begin{proof}
The proof is the same as those for $\CSFD$ and $\CSFA$. The only new step is showing that there are no mixed terms,
i.e. combinations of~(\ref{cond:z1}) and~(\ref{cond:z2}). In other words, the actions of $\A(-\Z_1)$ and $\A(\Z_2)$
commute for a nice diagram. The reason is that such a combined domain that hits both $\Z_1$ and
$\Z_2$ decomposes into two domains that hit only one side each. There is no constraint of the relative heights,
so such a domain will have index at least 2, and would not be counted.
\end{proof}

\subsection{Bimodule of the identity}
In this subsection we sketch the proof of theorem~\ref{thm:equivalence}. We prove a version for $\BSDA$, which implies
the original statement.

\begin{defn} Given an arc diagram $\Z$, define the bimodule $\lu{\A(\Z)}\,\II_{\A(\Z)}$, which as an $\I(\Z)$--bimodule
is isomorphic to $\I(\Z)$ itself, and whose nontrivial operations are
\begin{equation}
\label{eq:identity}
m_2(I_i,a)=a\otimes I_f,
\end{equation}
for all algebra elements $a\in\A(\Z)$ with initial and final idempotents $I_i$ and $I_f$, respectively.
It is absolutely graded by $\Grdn(\Z)$, as a subset of $\A(\Z)$, i.e. all elements are graded 0.
\end{defn}

It is easy to see that
$\lu{\A(\Z)}\,\II_{\A(\Z)}\,\sqtens\,\lu{\A(\Z)}M_{\A(\Z')}\cong\lu{\A(\Z)}M_{\A(\Z')}$ canonically, and that
$\A(\Z)\,\sqtens\,\lu{\A(\Z)}\,\II_{\A(\Z)}\simeq\A(\Z)$.

\begin{thm}
The identity decorated sutured cobordism $\id_{\Z}=(F(\Z)\times[0,1],\Lambda\times[0,1])$ from $\Z$ to $\Z$ has
a graded bimodule invariant
\begin{equation*}
\lu{\A(\Z)}\,\BSDA(\id_{\Z})_{\A(\Z)}\simeq\lu{\A(\Z)}\,\II_{\A(\Z)}.
\end{equation*}
\end{thm}
\begin{proof}[Proof (sketch)]
The proof is essentially the same as that of the corresponding statement for pure bordered identity cobordisms
in~\cite{LOT:bimodules}. First we look at an appropriate Heegaard diagram $\HH$ for $\id_{\Z}$. For any $\Z$ there
is a canonical diagram of the form in Fig.~\ref{subfig:HH_T}, only here we interpret the left side as the $-\Z$,
or type $D$, portion of the boundary, while the right side is the $+\Z$, or $\Ainf$--type, portion. Indeed,
choosing which of the right arcs are occupied in a generator determines it uniquely, and $\G(\HH)$ is a one-to-one
correspondence with elementary idempotents. Thus the underlying space for $\BSDA(\id_{\Z})$ is $\I(\Z)$.
For any Reeb chord $\rho$ of length one there is a convex octagonal domain in $\HH$ that makes Eq.~(\ref{eq:identity})
hold
for $a=a(\rho,s)$, for any such $\rho$, and any completion $s$.

The rest of the proof is algebraic. Any bimodule with underlying module $\I(\Z)$ corresponds to some $\Ainf$--algebra
morphism $\phi\co\A(\Z)\to\A(\Z)$. We compute the homology of $\A(\Z)$ and show it is Massey generated by length one
Reeb chords as above. Since Eq.~(\ref{eq:identity}) holds for such elements, $\phi$ is a quasi-isomorphism. By
theorem~\ref{thm:pairing_bimodules}, we know $\BSDA(\id_{\Z})$ squares to itself, and so does $\phi$, i.e.
$\phi\circ\phi\simeq\phi$. Since it is a quasi isomorphism, it is homotopic to the identity morphism, and
Eq.~(\ref{eq:identity}) holds for all $a$, up to homotopy equivalence.

Finally, for the grading, $\Grdn(-\Z\cup\Z)$ has two copies of $H_1(F(\Z))$, with opposite pairings. For all
$\spinc$--structures, there are obvious periodic domains, such that $\pi_2(\xgen,\xgen)=H_1(F)$. Taking the quotient
by the stabilizer subgroup identifies the subgroups $\Grdn(-\Z)$ and $\Grdn(\Z)$ by the canonical anti-isomorphism.
All domains have vanishing Maslov grading and canceling homological gradings, so in each $\spinc$--structure all
generators have the same relative grading. Thus, we can identify it with an absolute grading where all gradings are 0.
\end{proof}

\section{Examples}
\label{sec:examples}
To help the reader understand the definitions we give some
simple examples of bordered sutured manifolds and compute their invariants.

\subsection{Sutured surfaces and arc diagrams}
First we discuss some simple arc diagrams and their algebras, that parametrize the same sutured surfaces

\begin{exm}
\label{exm:ss_fn}
One of the simplest classes of examples is the following. Let $F_n$ be the sutured surface $(D^2,\Lambda_n)$, where
$\Lambda_n$ consists of $2n$ distinct points. That is, $F_n$ is a disc, whose boundary circle is divided into
$n$ positive and $n$ negative arcs.
\end{exm}

There are many different arc diagrams for $F_n$, especially for large $n$, but there are two special cases
which we will consider in detail.

\begin{exm}
\label{exm:ad_wn}
Let $\Zmid=\{Z_1,\ldots,Z_n\}$ be a collection of oriented arcs, and
$\amid=\{a_1,\ldots,a_{2n-2}\}$ be a collection of points, such that $a_1,\ldots,a_{n-1}\in Z_1$ are
in this order, and $a_{n+i-1}\in Z_{i+1}$ for $i=1,\ldots,n-1$. Let $M$ be the matching
$M(a_i)=M(a_{2n-i-1})=i$ for $i=1,\ldots,n-1$. The arc diagram
$\W_n=(\Zmid,\amid,M)$ parametrizes $F_n$, as in Fig.~\ref{subfig:wn}.
\end{exm}

\begin{prop}
For the arc diagram $\W_n$ from example~\ref{exm:ad_wn},
the algebra $\A(\W_n)$ satisfies $\A(\W_n,k)\cong\A(n-1,k)$ for all $k=0,\ldots,n-1$.
\end{prop}
\begin{proof}
The algebra $\A(\W_n)$ is a subalgebra of $\A(n-1,1,1,\ldots,1)\cong\A(n-1)\otimes\A(1)^{\otimes(n-1)}$.
But $\A(1)=\A(1,0)\oplus\A(1,1)$, where both summands are trivial.
The projection $\pi$ to $\A(n-1)\otimes\A(1,0)^{\otimes(n-1)}\cong\A(n-1)$ respects the algebra structure.
For each $\brho$ and completion $s$, the projection $\pi$ kills all summands in $a(\brho,s)$, except
the one corresponding to the unique section $S$ of $s$, where $S\subset\{1,\ldots,n-1\}$. Therefore
$\pi|_{\A(\W_n)}$ is an isomorphism.
\end{proof}

\begin{exm}
\label{exm:ad_vn}
Let $\Zmid=\{Z_1,\ldots,Z_n\}$ and $\amid=\{a_1,\ldots,a_{2n-2}\}$, again but set
$a_1\in Z_1$, $a_{2n-2}\in Z_n$, while $a_{2i},a_{2i+1}\in Z_{i+1}$ for $i=1,\ldots,n-2$. Set the matching $M$
to be $M(a_{2i-1})=M(a_{2i})=i$ for $i=1,\ldots,n-1$. The arc diagram $\V_n=(\Zmid,\amid,M)$ also parametrizes $F_n$, as
in Fig.~\ref{subfig:vn}
\end{exm}

\begin{prop}
For the arc diagram $\V_n$ from example~\ref{exm:ad_vn}, its associated algebra $\A(\V_n)$ has no differential.
\end{prop}
\begin{proof}
By definition $\A(\V_n)$ is a subalgebra of $\A(1)\otimes\A(2)^{\otimes(n-2)}\otimes\A(1)$.
It is trivial to check that neither $\A(1)$, nor $\A(2)$ have differentials. The differential on
their product is defined by the Leibniz rule, so it also vanishes.
\end{proof}

\begin{figure}
\begin{subfigure}[b]{.48\linewidth}
	\centering
	\includegraphics[scale=.7]{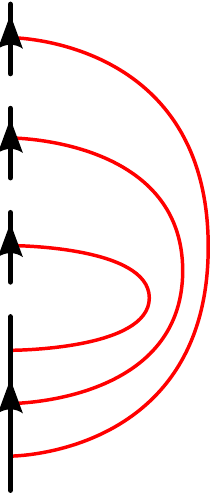}\qquad
	\labellist
	\small\hair 2pt
	\pinlabel $Z_1$ [t] at 80 80
	\pinlabel $Z_2$ [r] at 0 160
	\pinlabel $Z_n$ [l] at 160 160
	\pinlabel $\cdots$ [b] at 80 240
	\endlabellist
	\includegraphics[scale=.5]{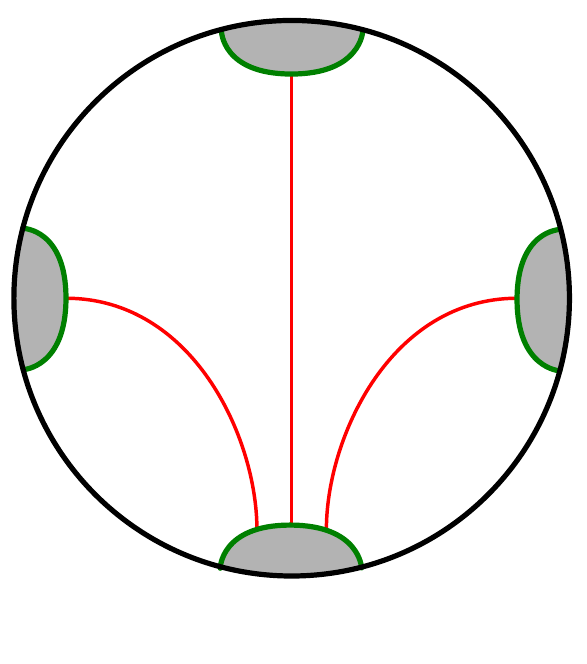}
	\caption{The arc diagram $\W_n$ for $F_n$ and corresponding parametrization.}
	\label{subfig:wn}
\end{subfigure}
\begin{subfigure}[b]{.48\linewidth}
	\centering
	\includegraphics[scale=.7]{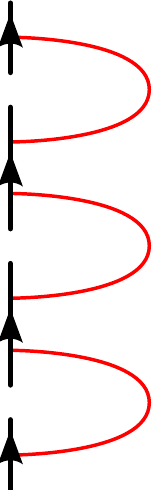}\qquad
	\labellist
	\small\hair 2pt
	\pinlabel $Z_1$ [t] at 80 80
	\pinlabel $Z_2$ [r] at 0 160
	\pinlabel $Z_n$ [l] at 160 160
	\pinlabel $\cdots$ [b] at 80 240
	\endlabellist
	\includegraphics[scale=.5]{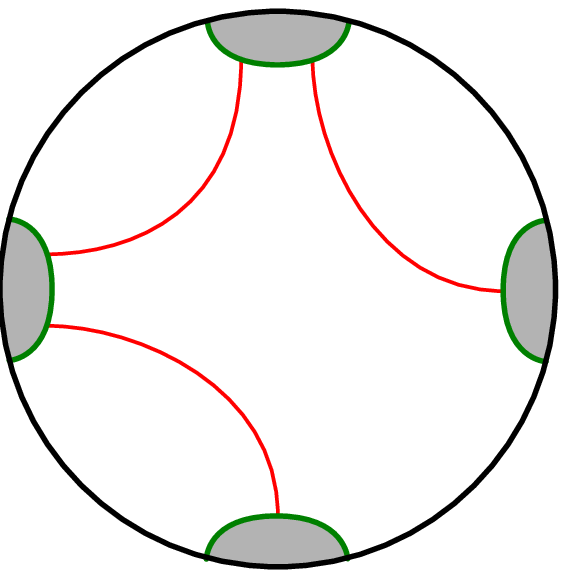}
	\caption{The arc diagram $\V_n$ for $F_n$ and corresponding parametrization.}
	\label{subfig:vn}
\end{subfigure}
\caption{Two parametrizations of $F_n$.}
\label{fig:fn_params}
\end{figure}

It will be useful for next section to compute the two algebras $\A(\W_4)$ and $\A(\V_4)$ explicitly.
Recall definition~\ref{def:a_rho_s}, which assigns to a collection $\brho$ of Reeb chords, corresponding to
moving strands, and a completion $s$, corresponding to stationary strands, an algebra element $a(\brho,s)$.
Abusing notation, we will denote the idempotent $I(\{1,2,4\})$ by $I_{124}$, etc.

In $\W_4$ there are three Reeb chords---$\rho_1$ from $a_1$ to $a_2$, $\rho_2$ from $a_2$ to $a_3$, and
their concatenation $\rho_{12}$ from $a_1$ to $a_3$. The algebra splits into 4 summands.
The 0-- and 3--summands $\A(\W_4,0)=\left<I_{\varnothing}\right>$ and $\A(\W_4,3)=\left<I_{123}\right>$ are trivial.
The 1--summand is $\A(\W_4,1)=\left<I_1,I_2,I_3,\rho'_1,\rho'_2,\rho'_{12}\right>$. It has three idempotents and three other generators
$\rho'_i=a(\{\rho_i\},\varnothing)$. It has no differential, and the only nontrivial product is
$\rho'_1\cdot\rho'_2=\rho'_{12}$. The 2--summand
$\A(\W_4,2)=\left<I_{12},I_{13},I_{23},\rho''_1,\rho''_2,\rho''_{12},\rho''_2\cdot\rho''_1\right>$ is the most interesting.
Here $\rho''_1=a(\{\rho_1\},\{3\})$, $\rho''_2=a(\{\rho_2\},\{1\})$, $\rho''_{12}=a(\{\rho_{12}\},\{2\})$, and
$\rho''_2\cdot\rho''_1=a(\{\rho_1,\rho_2\},\varnothing)$. There is a nontrivial differential
$\del\rho''_{12}=\rho''_2\cdot\rho''_1$, and one nontrivial product, which is clear from our notation.

In $\V_4$ there are two Reeb chords---$\sigma_1$ from $a_2$ to $a_3$, and $\sigma_2$ from $a_4$ to $a_5$. Again,
the summands $\A(\V_4,0)=\left<I_{\varnothing}\right>$ and $\A(\V_4,3)=\left<I_{123}\right>$ are trivial.
The 1--summand is $\A(\V_4,1)=\left<I_1,I_2,I_3,\sigma'_1,\sigma'_2\right>$, where $\sigma'_i=a(\{\sigma_i\},\varnothing)$. It has no
nontrivial differentials or products. The 2--summand is
$\A(\V_4,2)=\left<I_{12},I_{13},I_{13},\sigma''_1,\sigma''_2,\sigma''_2\cdot\sigma''_1\right>$,
where $\sigma''_1=a(\{\sigma_1\},\{3\})$, $\sigma''_2=a(\{\sigma_2\},\{1\})$, and
$\sigma''_2\cdot\sigma''_1=a(\{\sigma_1,\sigma_2\},\varnothing)$. There are no differentials and there is one nontrivial product.

\subsection{Bordered sutured manifolds}
We give three examples of bordered sutured manifolds. Topologically they are all very simple---in fact they are all
$D^2\times[0,1]$. They are, nonetheless, interesting and have nontrivial invariants. Bordered sutured manifolds
of this type are essential for the study of
what happens when we fill in a sutured surface with a chord diagram.

\begin{figure}
\begin{subfigure}[t]{.32\linewidth}
	\centering
	\includegraphics[scale=.45]{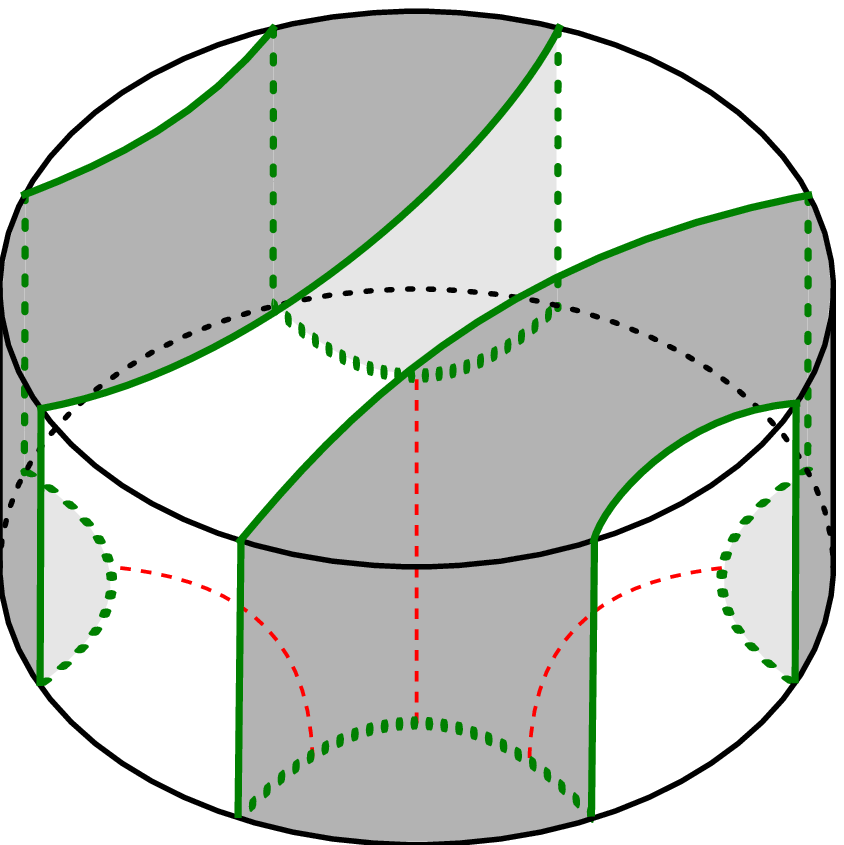}
	\caption{$M_1$ bordering $-\W_4$}
	\label{subfig:bs_mfld_1}
\end{subfigure}
\begin{subfigure}[t]{.32\linewidth}
	\centering
	\includegraphics[scale=.45]{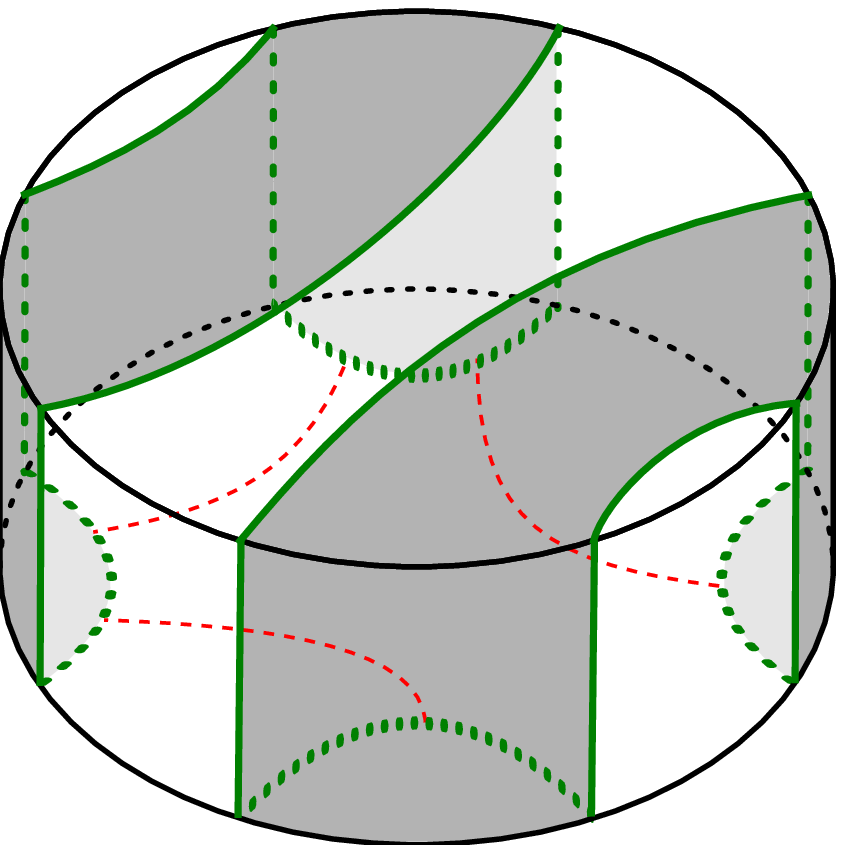}
	\caption{$M_2$ bordering $-\V_4$}
	\label{subfig:bs_mfld_2}
\end{subfigure}
\begin{subfigure}[t]{.32\linewidth}
	\centering
	\includegraphics[scale=.45]{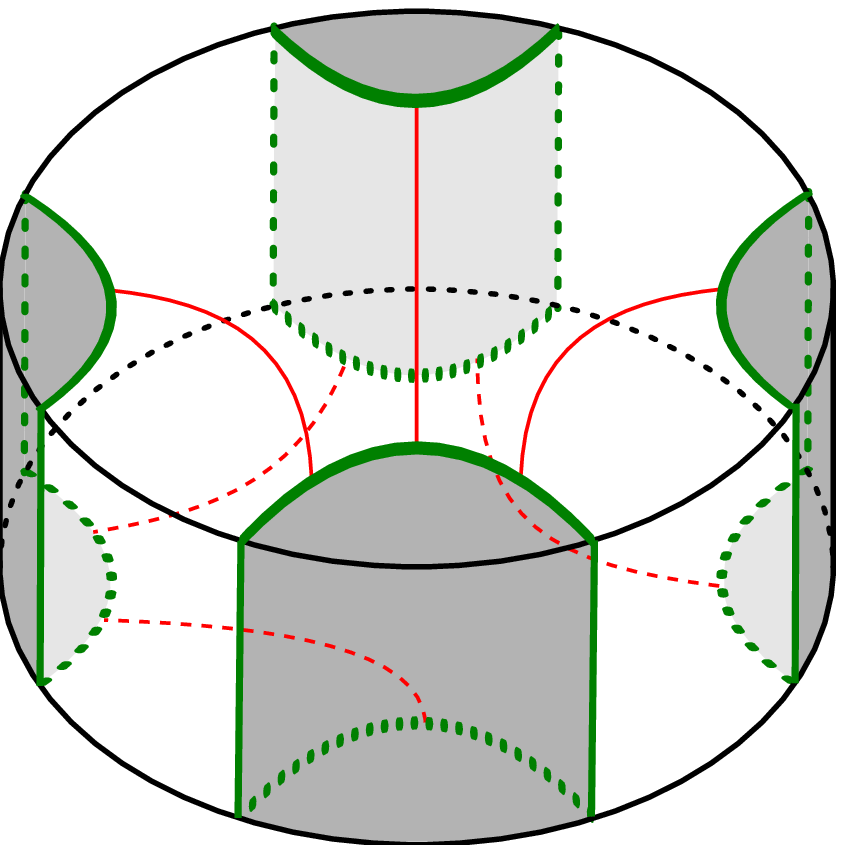}
	\caption{$M_3$ bordering\\$-\V_4\cup\W_4$}
	\label{subfig:bs_mfld_3}
\end{subfigure}
\begin{subfigure}[t]{.32\linewidth}
	\centering
	\labellist
	\small
	\pinlabel $A$ at 40 210
	\pinlabel \rotatebox{180}{\reflectbox{$A$}} at 40 75
	\small\hair 4pt
	\pinlabel $-\W_4$ [tr] at 0 360
	\hair 2pt
	\pinlabel $\rho_1$ [r] at 0 45
	\pinlabel $\rho_2$ [r] at 0 75
	\small
	\pinlabel $x$ [tr] at 20 60
	\pinlabel $y$ [br] at 20 90
	\endlabellist
	\includegraphics[scale=.6]{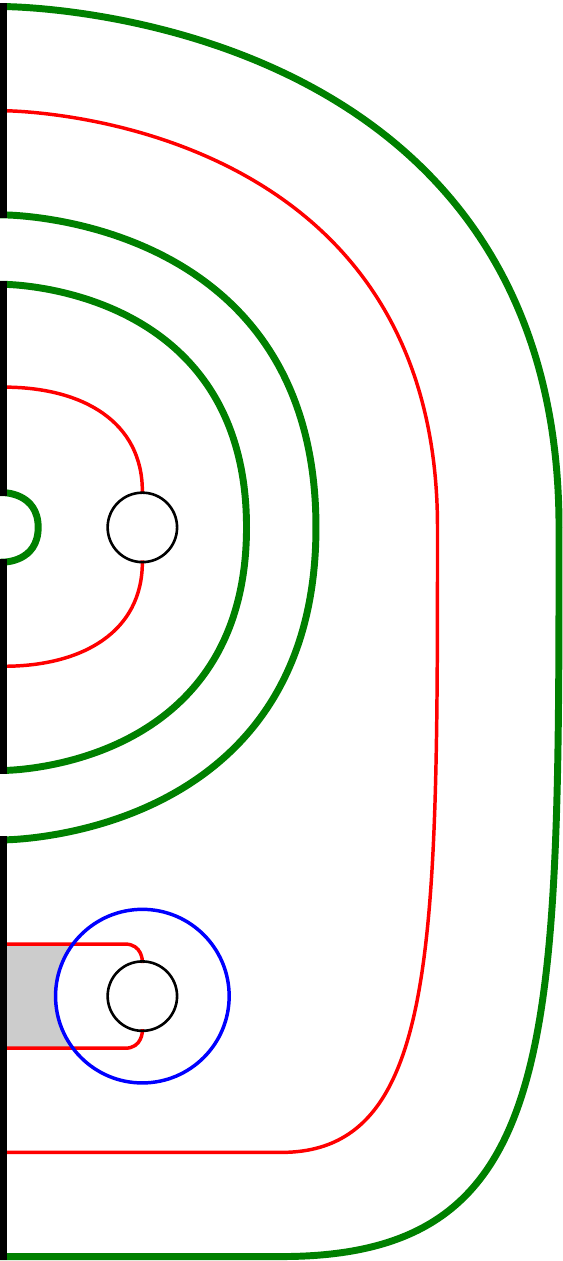}
	\caption{$\HH_{M_1}$}
	\label{subfig:bs_diag_1}
\end{subfigure}
\begin{subfigure}[t]{.32\linewidth}
	\centering
	\labellist
	\small
	\pinlabel $A$ at 50 210
	\pinlabel \rotatebox{180}{\reflectbox{$A$}} at 50 75
	\small\hair 4pt
	\pinlabel $-\V_4$ [tr] at 0 360
	\hair 2pt
	\pinlabel $\sigma_1$ [r] at 0 170
	\pinlabel $\sigma_2$ [r] at 0 250
	\small
	\pinlabel $u$ [tr] at 53 50
	\pinlabel $v$ [bl] at 45 100
	\endlabellist
	\includegraphics[scale=.6]{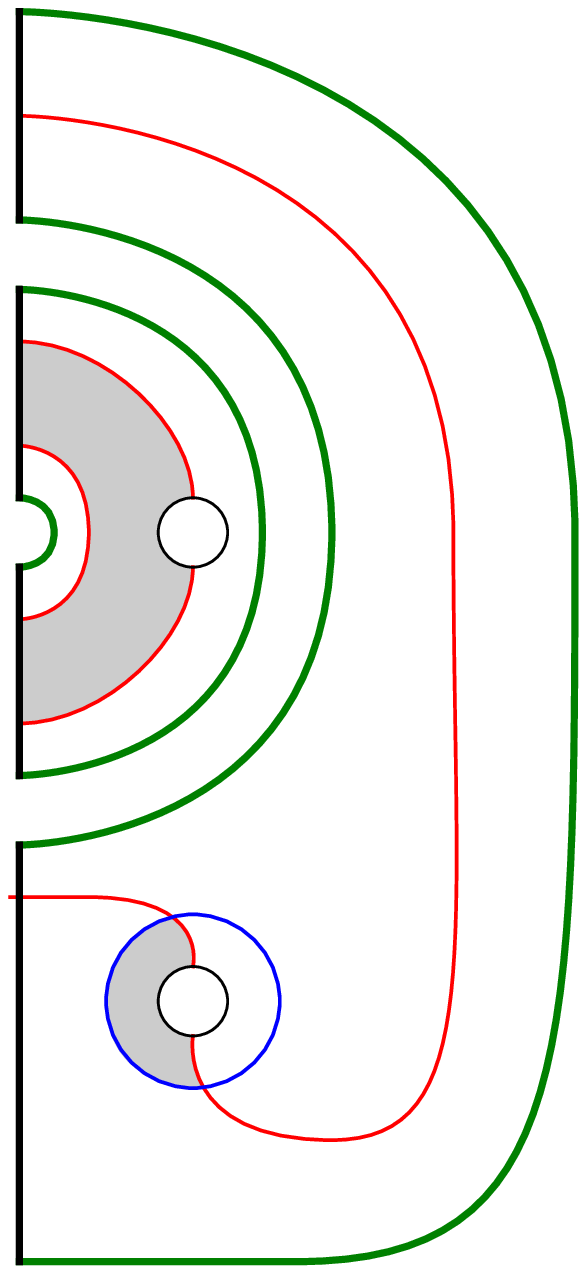}
	\caption{$\HH_{M_2}$}
	\label{subfig:bs_diag_2}
\end{subfigure}
\begin{subfigure}[t]{.32\linewidth}
	\centering
	\labellist
	\small
	\pinlabel $A$ at 55 -5
	\pinlabel \rotatebox{180}{\reflectbox{$A$}} at 55 -70
	\pinlabel $B$ at 100 75
	\pinlabel \rotatebox{180}{\reflectbox{$B$}} at 100 -100
	\pinlabel $C$ at 145 170
	\pinlabel \rotatebox{180}{\reflectbox{$C$}} at 145 -130
	\small\hair 4pt
	\pinlabel $-\V_4$ [tr] at 15 200
	\pinlabel $\W_4$ [tl] at 185 200
	\hair 2pt
	\pinlabel $\rho_1$ [l] at 185 -115
	\pinlabel $\rho_2$ [l] at 185 -85
	\pinlabel $\sigma_1$ [r] at 15 10
	\pinlabel $\sigma_2$ [r] at 15 90
	\small
	\pinlabel $a$ [tr] at 125 -130
	\pinlabel $b$ [tr] at 100 -120
	\pinlabel $c$ [tr] at 80 -100
	\pinlabel $d$ [tr] at 55 -90
	\pinlabel $e$ [tr] at 35 -70
	\pinlabel $f$ [bl] at 165 -130
	\pinlabel $g$ [bl] at 120 -100
	\pinlabel $h$ [bl] at 75 -70
	\endlabellist
	\includegraphics[scale=.6]{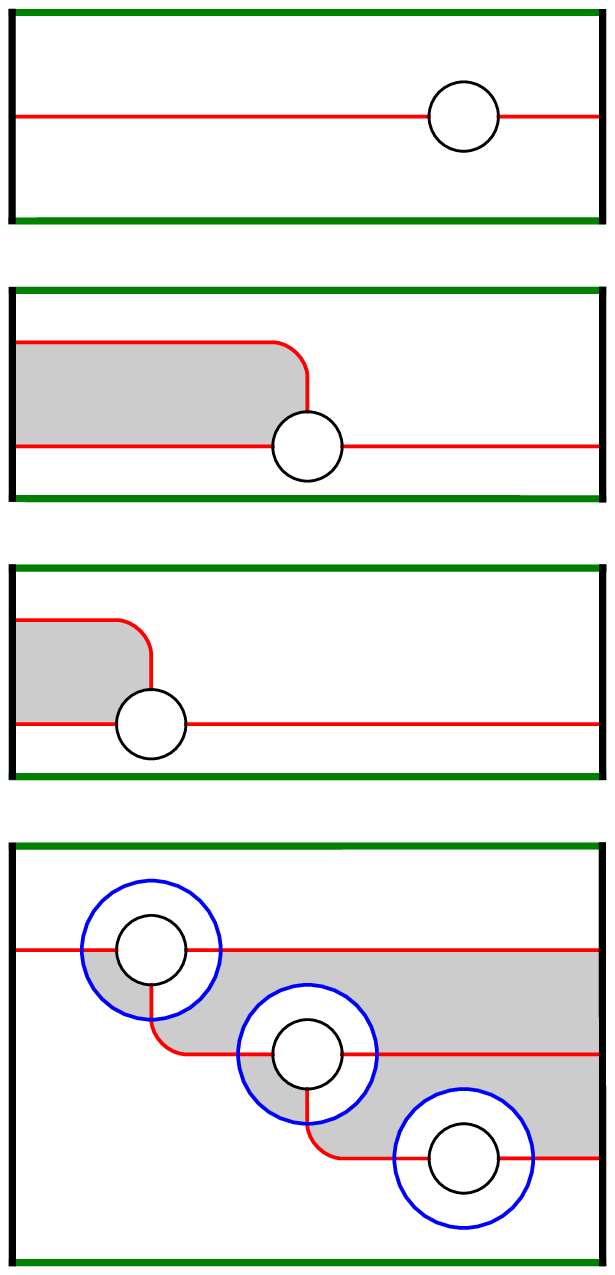}
	\caption{$\HH_{M_3}$}
	\label{subfig:bs_diag_3}
\end{subfigure}
\caption{Three examples of bordered sutured manifolds (top row), and their diagrams (bottom row).
Capital roman letters denote 1--handles, lower case roman letters denote intersection points,
and Greek letters denote Reeb chords (always oriented upward). All non-boundary elementary regions have been shaded.}
\label{fig:bs_examples}
\end{figure}

\begin{exm}
\label{exm:m1}
The first example is $M_1=(D^2\times[0,1],\Gamma_1,-\W_4)$, where $D^2\times\{0\}$ is parametrized by $-\W_n$, and the rest of the boundary
is divided into two positive and three negative regions (see Fig.~\ref{subfig:bs_mfld_1}). An admissible---and in fact nice---Heegaard diagram 
for $M_1$ is
given in Fig.~\ref{subfig:bs_mfld_1}. We will compute $\lu{\A(\W_4)}\,\CSFD(M_1)$.

First, notice that the relative $\spinc$--structures are in one-to-one correspondence with $H_1(D^2\times[0,1],D^2\times\{0\})=0$, so there is
a unique $\spinc$--structure. There are two generators $(x)$ and $(y)$, with idempotents $I_{13}\cdot(x)=(x)$, and $I_{12}\cdot(y)=(y)$
(both in $\I(\W_4,2)$).
There is a single region contributing to $\delta$. It corresponds to a source $S^\trr$ which is a bigon from $(y)$ to $(x)$, with one
$e$ puncture labeled $-\rho_2$. It contributes $a_2(\rho_2)\otimes(x)=\rho''_2\otimes(x)$ to $\delta(y)$.
Therefore, the only nontrivial term in $\delta$ is
\begin{equation*}
\delta((y))=\rho''_2\otimes(x).
\end{equation*}

If we want to compute $\CSFA(M_1)_{\A(-\W_4)}$, the same generators have idempotents $(x)\cdot I_2=(x)$ and $(y)\cdot I_3=(y)$, and
the same region contributes $(x)$ to $m_2((y),a(-\rho_2))$, instead. The only nontrivial term is
\begin{equation*}
m_2((y),-\rho'_2)=(x).
\end{equation*}
\end{exm}

\begin{exm}
\label{exm:m2}
The second example is $M_2=(D^2\times[0,1],\Gamma_1,-\V_4)$, which is the same as $M_1$, except for the different parametrization of
$D^2\times\{0\}$ (see Fig.~\ref{subfig:bs_mfld_2}). An admissible diagram for $M_2$ is given in Fig.~\ref{subfig:bs_diag_2}.

First, we compute ${}^{\A(\V_4)}\CSFD(M_2)$. It has two generators, with idempotents $I_{12}\cdot(u)=(u)$ and $I_{23}\cdot(v)=(v)$. There is
one region which is a bigon with two $e$ punctures labeled $-\sigma_2$ and $-\sigma_1$, at different heights. It contributes
$a_2(\sigma_2)a_2(\sigma_1)\otimes (v)=\sigma''_2\cdot\sigma''_1\otimes(v)$ to $\delta((u))$. Therefore the differential is
\begin{equation*}
\delta((u))=\sigma''_2\cdot\sigma''_1\otimes(v).
\end{equation*}

For $\CSFA(M_2)_{\A(-\V_4)}$, the idempotents are $(u)\cdot I_3=(u)$ and $(v)\cdot I_1=(v)$. The region contributes to $m_3$, yielding
\begin{equation*}
m_3((u),-\sigma'_2,-\sigma'_1)=(v).
\end{equation*}
\end{exm}

\begin{exm}
\label{exm:m3}
Our last---and richest---example is $M_3=(D^2\times[0,1],\Gamma_2,-\V_4\cup\W_4)$, where $-\V_4$ parametrizes $D^2\times\{0\}$, and
$\W_4$ parametrizes $D^2\times\{1\}$ (see Fig.~\ref{subfig:bs_mfld_3}). This is a decorated sutured cobordism from
$\V_4$ to $\W_4$, which is an isomorphism in the decorated category $\sdcat$.
An admissible diagram for $M_3$ is given in
Fig.~\ref{subfig:bs_diag_3}.

We will compute (part of) ${}^{\A(\V_4)}\BSDA(M_3)_{\A(\W_4)}$. In this case, since $H_1(D^2\times[0,1],D^2\times\{0,1\})=\ZZ$, there are multiple
$\spinc$--structures. As in the proof of theorem~\ref{thm:decomp}, the $\spinc$--structures correspond to how many $\alpha^a$ arcs are occupied on
the $\W_4$ side of $-\V_4\cup\W_4$. Let $\s_k$ be the $\spinc$--structure with $k$ arcs occupied. There are $3-k$ arcs occupied on the $-\V_4$ side
for each such generator, and therefore $\BSDA(M_3,\s_k)$ is a bimodule over $\A(\V_4,k)$ and $\A(\W_4,k)$. Moreover, only $k=0,1,2,3$ give
nonzero invariants.

It is easy to check that $\BSDA(M_3,\s_0)$ and $\BSDA(M_3,\s_3)$ have unique generators, $(ace)$ and $(fgh)$, respectively, with no
nontrivial actions $m_k$. We will leave $\BSDA(M_3,\s_1)$ as an exercise and compute $\BSDA(M_3,\s_2)$.
There are five generators, with idempotents as follows.
\begin{align*}
I_{12}\cdot(agh)\cdot I_{23}&=(agh) & I_{12}\cdot(fbh)\cdot I_{13}&=(fbh)\\
I_{13}\cdot(fch)\cdot I_{13}&=(fch) & I_{13}\cdot(fgd)\cdot I_{12}&=(fgd)\\
I_{23}\cdot(fge)\cdot I_{12}&=(fge)
\end{align*}

There are four elementary domains, each of which contributes one term to $m_1$ or $m_2$. Some of them also contribute to $m_1$ or $m_2$
for $\BSDA(M_3,\s_1)$, and there is a composite domain that also contributes in that case. The nontrivial operations for $\BSDA(M_3,\s_2)$
are listed below.
\begin{align*}
m_1((fgd))&=\sigma''_1\otimes(fge) & m_2((fgd),\rho''_2)&=I_{13}\otimes(fch)\\
m_1((fbh))&=\sigma''_2\otimes(fch) & m_2((fbh),\rho''_1)&=I_{12}\otimes(agh)
\end{align*}
\end{exm}

\subsection{Gluing}
Our final example is of gluing of bordered sutured manifolds and the corresponding operation on their invariants.

\begin{exm}
We will use the manifolds from examples~\ref{exm:m1}--\ref{exm:m3}. If we glue $M_1$ and $M_2$ along $F(\W_4)$ we obtain exactly $M_2$.
Treating
${}^{\A(\W_4)}\CSFD(M_1)$ as ${}^{\A(\W_4)}\BSDA(M_1)_{\A(\varnothing)}$, we can
compute the product
\begin{equation*}
\BSDA(M_3)\,\sqtens_{\A(\W_4)}\,\CSFD(M_1),
\end{equation*}
which is a type $D$ structure over $\A(\V_4)$.
Since the only relative $\spinc$--structure on $M_3$ which
extends over $M_1$ is $\s_2$, the product is equal to $\BSDA(M_3,\s_2)\sqtens\CSFD(M_1)$. Another way to see this is to notice that
if we decompose the product over $\sqtens_{\A(\W_4,k)}$, only the $k=2$ term is nonzero.

After taking the tensor product $\otimes_{\I(\W_4,2)}$ of the underlying modules, the generators and idempotents are
\begin{align*}
I_{13}\cdot(fch)\sqtens(x)&=(fch)\sqtens(x) & I_{12}\cdot(fbh)\sqtens(x)&=(fbh)\sqtens(x)\\
I_{23}\cdot(fge)\sqtens(y)&=(fge)\sqtens(y) & I_{13}\cdot(fgd)\sqtens(y)&=(fgd)\sqtens(y)
\end{align*}

The induced operations are
\begin{align*}
\delta((fgd)\sqtens(y))&=\sigma''_1\otimes((fge)\sqtens(y))+I_{13}\otimes((fch)\sqtens(x))\\
\delta((fbh)\sqtens(x))&=\sigma''_2\otimes((fch)\sqtens(x))
\end{align*}

There is one pure differential, from $(fgd)\sqtens(y)$ to $(fch)\sqtens(x)$. We can cancel it, and see that the complex is homotopy
equivalent to $\CSFD(M_2)$, as expected from the pairing theorem.
\end{exm}

\section{Applications}
\label{sec:applications}

In this section we describe some applications of the new invariants. First,
as a warm-up we describe how both sutured Floer homology and the regular 
bordered Floer homology appear as special cases of
bordered sutured homology. Then we describe how we can recover the sutured
Floer homology of a manifold with boundary from its bordered invariants.

Another application is a new proof for the surface decomposition formula
\cite[Theorem 1.3]{Juh:decompositions} of Juh\'asz.

\subsection{Sutured Floer homology as a special case}
We have already seen that the for a bordered sutured manifold $(Y,\Gamma,\varnothing)$,
the bordered sutured invariants coincide with the sutured ones. However, there are many
more cases when this happens. In fact, for any \emph{balanced} bordered sutured manifold,
the $\BSD$ and $\BSA$ invariants still reduce to $\SFH$, no matter what the arc diagram is.

\begin{thm}
Let $(Y,\Gamma)$ be a balanced sutured manifold, and
$\phi\co G(\Z)\to\del Y$ be a parametrization of any part of $(Y,\Gamma)$ by an arc
diagram $\Z$ with $k$ matched pairs. Let $(\SFC,\del)$ be the sutured chain complex
for $(Y,\Gamma)$.

The following statements hold.
\begin{enumerate}
\item $(\CSFA(Y,\Gamma,\Z,\phi),m_1)\simeq (\SFC(Y,\Gamma),\del)$, where $\A(\Z,0)=\{I(\varnothing)\}$ acts
by identity on $\CSFA$ and $\A(\Z,k)$ kills it for any $k>0$.
\item $\CSFD(Y,\Gamma,\Z,\phi)\cong \SFC(Y,\Gamma)$ as a set, with
\begin{equation*}
\delta(x)=I\otimes \del(x),
\end{equation*}
where $I=I(\{1,\ldots,k\})$ is the unique idempotent in $\A(-\Z,k)$.
\item $\CSFDD(Y,\Gamma,\Z,\phi)\simeq \A(-\Z,k) \otimes \SFC(Y,\Gamma)$ as a product of
chain complexes, with the standard action of $\A(-\Z)$ on $\A(-\Z,k)$ on the left.
\end{enumerate}
\end{thm}
\begin{proof}
Let $\HH$ be a provincially admissible Heegaard diagram for $(Y,\Gamma,\Z,\phi)$.
If we erase $\mathbf Z$ and
$\balpha^a$ from the diagram, we obtain an admissible sutured diagram $\HH'$
for $(Y,\Gamma)$.
(Indeed, any periodic domain for $\HH'$ is a provincial periodic domain for $\HH$.)

Remember that for a balanced, i.e. 0--unbalanced manifold, each generator occupies 0 arcs in $\balpha^a$.
In particular $\G(\HH)=\G(\HH')$.

Let $u\in\M^B(\xgen,\ygen,S^\trr,\parrow)$
be a strongly boundary monotonic curve.
Let $o_t(u)$ denote the set of $\alpha\in\balpha$, for which
$u^{-1}(\alpha\times\{1\}\times\{t\})$ is nonempty.
Since $\xgen$ occupies only $\alpha$ circles, $o_t(u)\subset\balpha^c$ for small $t$.
The only changes in $o_t(u)$ can occur at the heights of $e$ punctures.
But at an $e$ puncture, the boundary goes over a Reeb chord, so $o_t(u)$ can only change by
replacing some arc in $\balpha^a$ with another.
Therefore, $o_t(u)\subset\balpha^c$ for all $t\in\RR$, and $S^\trr$ has no
$e$ punctures. Thus, $u$ is a curve with no $e$ punctures and doesn't involve $\balpha^a$.
But these are exactly the curves from $\HH'$ counted in the definition of $\SFH$.

Therefore, the curves counted for the definitions of $\CSFD$ and $\CSFA$ from $\HH$ 
are in a one-to-one correspondence with curves counted for the definition of $\SFH$ from $\HH'$.
Moreover, in $\CSFD$ and $\CSFA$ these curves are all provincial.

Algebraically, in $\CSFD$ a provincial curve from $\xgen$ to $\ygen$ contributes $1 \otimes \ygen$ to
$\delta(\xgen)$. In $\CSFA$ it contributes $\ygen$ to $m_1(\xgen)$. Finally, in $\SFH$ it contributes $\ygen$
to $\del(\xgen)$. The first two statements follow. The last is a trivial consequence of the
definition of $\CSFDD$.
\end{proof}

In particular, the interesting behavior of the bordered sutured invariants occurs
when the underlying sutured manifold is unbalanced. In that case sutured Floer
homology is not defined, or is trivially set to 0.

\subsection{Bordered Floer homology as a special case}

The situation in this section is the opposite of that in the previous one. Here we show that if
we look at manifolds that are, in a sense, maximally unbalanced, the bordered sutured
invariants reduce to purely bordered invariants.

First we recall a basic result from \cite{Juh:SFH}.
\begin{prop}
Let $\C$ denote the class of all closed connected 3--man\-i\-folds, and $\C'$ denote the class
of all sutured 3--manifolds with one boundary component homeomorphic to $S^2$,
and a single suture on it.
The following statements hold.
\begin{enumerate}
\item $\C$ and $\C'$ are in a one-to-one correspondence given by the map
\begin{equation*}
\xi\co \C\to\C',
\end{equation*}
where $\xi(Y)$ is obtained by removing an open 3--ball from $Y$, and putting a single suture
on the boundary.
\item There is a natural homotopy equivalence
$\CFhat(Y)\simeq \SFC(\xi(Y))$.
\end{enumerate}
\end{prop}

The correspondence is most evident on the level of Heegaard diagrams, where a diagram for $\xi(Y)$ is obtained
from a diagram for $Y$ by cutting out a small disc around the basepoint.

There is a natural extension of this result to the bordered category.

\begin{thm}
\label{thm:bordered_special_case}
Let $\B$ denote the class of bordered manifolds with one boundary component,
and let $\B'$ denote the class of bordered sutured manifolds of the following form.
$(Y,\Gamma,\Z,\phi)\in\B'$ if and only if $D=\del Y\setminus F(\Z)$ is a single
disc $D$ and $\Gamma\cap D$ is a single arc.
The following statements hold.
\begin{enumerate}
\item \label{cond:correspondence}
$\B$ and $\B'$ are in a one-to-one correspondence given by the map
\begin{equation*}
\zeta\co \B\to\B',
\end{equation*}
which to a bordered manifold $Y$ parametrized by $\Z=(Z,\amid,M,z)$ associates
a bordered sutured manifold $\zeta(Y)=(Y,Z,\Z',\phi)$, pa\-ram\-e\-trized by
$\Z'=(Z\setminus D,\amid,M)$, where $D$ is a small neighborhood of $z$.
\item \label{cond:homotopy_equiv}
For any $Y\in \B$, we have
\begin{align*}
\CSFD(\zeta(Y))&\simeq\CFD(Y),\\
\CSFA(\zeta(Y))&\simeq\CFA(Y).
\end{align*}
\item \label{cond:gluing_equiv}
If $Y_1$ and $Y_2$ are bordered manifolds that glue to form a closed manifold $Y$,
then $\zeta(Y_1)$ and $\zeta(Y_2)$ glue to form $\xi(Y)$.
\end{enumerate}
\end{thm}
\begin{proof}
In the bordered setting the parametrization of $F(\Z)=\del Y$
means that there is a self-indexing Morse function $f$ on
$F$ with one index--0 critical point $p$, one index--2 critical point $q$, and $2k$ many
index--1 critical points $r_1,\ldots,r_{2k}$. The circle $Z$ is the level set $f^{-1}(3/2)$,
the basepoint $z$ is the intersection of the gradient flow from $p$ to $q$ with $Z$, and
the matched points $M^{-1}(i)\in\amid$ are the intersections of the flowlines from $r_i$ with $Z$.

Note that $F'=F\setminus D$ is a surface with boundary, parametrized by the arc diagram
$\Z'=(Z',\amid,M)$, where $Z'=Z\setminus D$. Indeed, if we take $D$ to be a neighborhood of the flowline
from $p$ to $q$, then $f|_{F'}$ is a self indexing Morse function for $F'$ with only
index 1--critical points, and their stable manifolds intersect the level set $Z'$ at the
matched points $\amid$.

Moreover, the circle $Z$ separates $F$ into two regions---a disc $R_+$ around the index--2
critical point $q$, and a genus $k$ surface $R_-$ with one boundary component.
Thus, $(Y,Z)$ is indeed sutured, and the arc $Z'$ embeds into the suture $Z$.
Since $D\cap Z$ is an arc, the manifold we get is indeed in $\B'$.

To see that the construction is reversible we need to check that
for any $(Y,\Gamma,\Z,\phi)\in\B'$ there is only one suture in $\Gamma$, $\Zmid$ has only one component, and $R_+$ is a disc.
Indeed, $\Zmid\cap\Gamma$ consists only of properly embedded arcs in $F(\Z)$.
But $\Gamma\cap\del F=\Gamma\cap \del D$ consists of two points, and therefore there is only
one arc. Now $\Gamma=(\Gamma\cap F)\cup(\Gamma\cap D)$ is a circle, and $R_+\cap F$ is half
a disc, so $R_+$ is a disc. This proves~(\ref{cond:correspondence}).

To see~(\ref{cond:homotopy_equiv}), we will investigate the correspondence on Heegaard diagrams.
Consider a boundary compatible Morse function $f$ on a bordered 3--manifold
$Y$. On the boundary it behaves as described in the first part of the proof. In the interior,
there are only index--1 and index--2 critical points. Let $B$ be a neighborhood in $Y$ of the
flowline from the index--0 to the index--3 critical point, which are the index--0 and index--2 critical
points on the surface. Then $D$ is precisely $B\cap\del Y$. Let $Y'=Y\setminus B$. Topologically,
passing from $Y$ to $Y'$ has no effect, except for canceling the two critical points. Now $f|_{Y'}$ is a boundary compatible
Morse function for the bordered sutured manifold $Y'=\xi(Y)$. One can verify this is the same
construction as above, except we have pushed $D$ slightly into the manifold.

Looking at the Heegaard diagrams $\HH=(\Sigma,\balpha,\bbeta)$ and $\HH'=(\Sigma',\balpha,\bbeta)$,
compatible with
$f$ and $f|_{Y'}$, respectively, one can see that the effect of removing $B$ on $\HH$
is that of removing a neighborhood of the basepoint $z\in\del\Sigma$. Now
$\Zmid=\del\Sigma\setminus{\nu(z)}$, the Reeb chords correspond, and
$\del\Sigma'\setminus\Zmid$ is a small arc in the region where $z$ used to be.

Recall that the definitions of $\CFD$ and $\CFA$ on one side, and $\CSFD$ and $\CSFA$ on the other, are
the same, except that $\del\Sigma'\setminus\Zmid$ in the latter plays the role of $z$ in the former.
Therefore the corresponding moduli spaces $\M$ exactly coincide, and for these particular diagrams
there is actual equality of the invariants, proving~(\ref{cond:homotopy_equiv}).

For~(\ref{cond:gluing_equiv}), it is enough to notice that $Y=Y_1\cup_{F}Y_2$, while
$\zeta(Y_1)\cup_{F\setminus D}\zeta(Y_2)=Y_1\cup_{F\setminus D}Y_2$, which is $Y$ minus a ball.
\end{proof}

\subsection{From bordered to sutured homology}
\label{sec:bordered_to_sutured}

In the current section we prove theorem~\ref{thm:intro_bordered_to_sutured}, which
was the original motivation for developing the theory
of bordered sutured manifolds and their invariants. Recall that it states that
for any set of sutures on a bordered manifold, the sutured homology can be obtained from
the bordered homology in a functorial way. A refined version is given below.

\begin{thm}
Let $F$ be a closed connected surface parametrized by some pointed matched circle $Z$.
Let $\Gamma$ be any set of sutures on $F$, i.e. an oriented multi curve in $F$ that divides
it into positive and negative regions $R_+$ and $R_-$.

There is a (non unique) left type $D$ structure $\CFD(\Gamma)$ over $\A(\Z)$, with the
following property. If $Y$ is any 3--manifold, such that $\del Y$ is identified with $F$,
making $(Y,\Gamma)$ a sutured manifold, then
\begin{equation*}
\SFC(Y,\Gamma)\simeq\CFA(Y)\sqtens\CFD(\Gamma).
\end{equation*}

Similarly, there is a (non unique) right $\Ainf$--module $\CFA(\Gamma)$ over $\A(-\Z)$,
such that
\begin{equation*}
\SFC(Y,\Gamma)\simeq\CFA(\Gamma)\sqtens\CFD(Y).
\end{equation*}
\end{thm}

Before we begin the proof, we will note that although $\CFD(\Gamma)$ and $\CFA(\Gamma)$
are not unique (not even up to homotopy equivalence), they can be easily made so by fixing
some extra data. The exact details will become clear below.

\begin{proof}
Fix the surface $F$, pointed matched circle $\Z=(\Zmid,\amid,M)$, and the sutures $\Gamma$.
Repeating the discussion in the proof of theorem~\ref{thm:bordered_special_case}, 
the parametrization of $F$ means that there is a self-indexing
Morse function $f$ on
$F$ with exactly one index--0 critical point, and exactly one index--2 critical point, where the circle $Z$
is the level set $f^{-1}(3/2)$.

The choice that breaks uniqueness is the following.
Isotope $\Gamma$ along $F$ until one of the sutures $\gamma$ becomes tangent to $Z$ at the
basepoint $z$, and so that the orientations of $Z$ and $\gamma$ agree. Let
$D$ be a disc neighborhood of $z$ in $F$. We can further isotope $\gamma$ until $\gamma\cap D=Z \cap D$.
We will refer to this operation as \emph{picking a basepoint, with direction, on $\Gamma$}.

Let $F'=F\setminus D$, and let $P$ be the 3--manifold
$F'\times[0,1]$. Let $\Delta$ be a set of sutures on
$P$, such that
\begin{align*}
(F'\times\{1\})\cap\Delta&=(F'\cap\Gamma)\times\{1\},\\
(F'\times\{0\})\cap\Delta&=(F'\cap Z)\times\{0\},\\
(\del D\times[0,1])\cap\Delta&=(\Gamma\cap\del D)\times[0,1].
\end{align*}

We orient $\Delta$ so that on the ``top'' surface $F'\times\{1\}$ 
its orientation agrees with $\Gamma$, its orientation on the ``bottom'' is
opposite from $Z$, and on $\del D\times[0,1]$ the two segments are oriented
opposite from each other.

As in theorem~\ref{thm:bordered_special_case},
$F'$ is parametrized by the arc diagram
$\Z'=(Z\setminus D,\amid,M)$.
Therefore the ``bottom'' of $P$, i.e. $F'\times\{0\}$ is parametrized by $-\Z'$.
(Indeed $-(Z\setminus D)$ is part of $\Delta$.) This makes $(P,\Delta)$ into a bordered sutured manifold,
parametrized by $-\Z'$.

Isotopies of $\Gamma$ outside of $D$ have no effect on $P$, except for an isotopy of
$\Delta$ in the non parametrized part of $\del P$.
Therefore the bordered sutured manifold $P$ is an invariant of $F$, $\Gamma$, and
the choice of basepoint on $\Gamma$.

Define
\begin{align*}
\CFD(\Gamma)&=\CSFD(P,\Delta),\\
\CFA(\Gamma)&=\CSFA(P,\Delta).
\end{align*}

It is clear that their homotopy types are invariants of $\Gamma$ and the choice of basepoint (with
direction). Since $\A(\Z')=\A(\Z)$, they are indeed modules over $\A(\Z)$ and $\A(-\Z)$,
respectively.

To prove the rest of the theorem, consider any manifold $Y$ with boundary $\del Y=F$.
By the construction in theorem~\ref{thm:bordered_special_case}, 
$\zeta(Y)$ is the sutured manifold $(Y,Z)$, where $F'$ is parametrized by $\Z'$. 

If we glue $\zeta(Y)$ and $P$ along $F'$, we get the sutured manifold
\begin{equation*}
(Y\cup F'\times[0,1],(Z\setminus F')\cup(\Delta\setminus F'\times\{0\})).
\end{equation*}

The sutures consist of  $Z\setminus F'=Z\cap D=\Gamma\cap D$,
$\Delta\cap(\del D\times[0,1])=(\Gamma\cap\del D)\times[0,1]$, and
$\Delta\cap(F'\times\{1\})=(\Gamma\setminus D)\times\{1\}$.
Up to homeomorphism, $Y\cup F'\times[0,1]=Y$, and under that homeomorphism
the sutures get collapsed to $\Gamma\subset F$. Therefore,
$\zeta(Y)\cup_{F'}P$ is precisely $(Y,\Gamma)$.

Using theorem~\ref{thm:bordered_special_case},
$\CSFD(\zeta(Y))\simeq\CFD(Y)$, and 
$\CSFA(\zeta(Y))\simeq\CFA(Y)$. By therorem~\ref{thm:pairing},
\begin{align*}
\SFC(Y,\Gamma)\simeq \CSFA(\zeta(Y))\sqtens\CSFD(P)&\simeq \CFA(Y)\sqtens\CFD(\Gamma),\\
\SFC(Y,\Gamma)\simeq \CSFA(P)\sqtens\CSFD(\zeta(Y))&\simeq \CFA(\Gamma)\sqtens\CFD(Y).
\qedhere
\end{align*}
\end{proof}

\subsection{Surface decompositions}
\label{sec:surface_decompositions}
The final application we will show is a new proof of the surface decomposition theorem
of Juh\'asz proved in \cite{Juh:decompositions}.

More precisely we prove the following statement.

\begin{thm}
\label{thm:decomp}
Let $(Y,\Gamma)$ be a balanced sutured manifold. Let $S$ be a properly embedded surface in $Y$
with the following properties. $S$ has no closed components, and each component of $\del S$
intersects both $R_-$ and $R_+$. (Juh\'asz calls such a surface a \emph{good decomposing surface}.)

A $\spinc$ structure $\s\in\spinc(Y,\Gamma)$ is \emph{outer} with respect to $S$ if it is represented by
a vector field $v$ which is nowhere tangent to a normal vector to $-S$ (with respect to some metric).

Let $(Y',\Gamma')$ be the sutured manifold, obtained by \emph{decomposing $Y$ along $S$}. More precisely,
$Y'$ is $Y$ cut along $S$, such that $\del Y'=\del Y\cup+S\cup-S$, and the sutures
$\Gamma'$ are chosen so that $R_-(\Gamma')=\overline{R_-(\Gamma)\cup-S}$, and
$R_+(\Gamma')=\overline{R_+(\Gamma)\cup+S}$. Here $+S$ (respectively $-S$) is the copy of $S$ on $\del Y'$, whose
orientation agrees (respectively disagrees) with $S$.

Then the following statement holds.
\begin{equation*}
\SFH(Y',\Gamma')\cong\bigoplus_{\s~\textrm{outward to}~S}\SFH(Y,\Gamma,\s).
\end{equation*}
\end{thm}

\begin{proof}
We will consider three bordered sutured manifolds. Let $T=S\times[-2,2]\subset Y$ be a neighborhood
of $S$ in $Y$ (so the positive normal of $S$ is in the $+$ direction).
Let $W=\overline{Y\setminus T}$, and let $P=S\times([-2,-1]\cup[1,2])\subset T$.
We can assume that $\Gamma\cap\del T$ consists of arcs parallel to the $[-2,2]$ factor.

Put sutures on $T$, $W$ and $P$ in the following way. First, notice that $R_+(\Gamma)\cap\del S$ consists of
several arcs $a=\{a_1,\ldots,a_n\}$. Let $A_+\subset S$ be a collection of disjoint discs, such that
$A_+\cap\del S=a$.

On $T$ put sutures $\Gamma_T$, such that
\begin{align*}
R_+(\Gamma_T)\cap\del Y&=R_+(\Gamma)\cap\del T,\\
R_+(\Gamma_T)\cap(S\times\{\pm 2\})&=A_+\times\{\pm 2\}.
\end{align*}

On $W$ put sutures $\Gamma_W$, such that
\begin{align*}
R_+(\Gamma_W)\cap\del Y&=R_+(\Gamma)\cap\del W,\\
R_+(\Gamma_W)\cap(S\times\{\pm 2\})&=A_+\times\{\pm 2\}.
\end{align*}

On $P$ put sutures $\Gamma_P$, such that
\begin{align*}
R_+(\Gamma_P)\cap\del Y&=R_+(\Gamma)\cap\del P,\\
R_+(\Gamma_P)\cap(S\times\{\pm 2\})&=A_+\times\{\pm 2\},\\
R_+(\Gamma_P)\cap(S\times\{-1\})&=S\times\{-1\},\\
R_+(\Gamma_P)\cap(S\times\{1\})&=\varnothing.
\end{align*}

Fix a parametrization of $S$ by a balanced arc diagram $\Z_S$ with $k$ many arcs, such that
the positive region of $S$ is $A_+$. This is possible, since $S$ has no closed components,
and the arcs $a$ hit every boundary component.

Parametrize the surfaces $S\times\{\pm 2\}$ in each of $T$, $W$, and $P$ by $\pm\Z_S$,
depending on orientation.
If we set $U=S\times\{\pm 2\}\subset W$, with the boundary orientation from $W$, then $U$
is parametrized by $\Z=\Z_1\cup\Z_2$, where $\Z_1\cong\Z_S$ parametrizes $S\times\{-2\}$, and
$\Z_2\cong-\Z_S$ parametrizes $S\times\{2\}$.
Thus, $W$ is a bordered sutured manifold parametrized by $\Z$, while
$T$ and $P$ are parametrized by $-\Z$ (see Fig.~\ref{fig:decomp_manifolds}).
Moreover, gluing along the
parametrization,
\begin{align*}
W\cup_U T&=(Y,\Gamma),\\
W\cup_U P&=(Y',\Gamma').
\end{align*}

\begin{figure}
\begin{subfigure}[b]{\linewidth}
	\centering
	\labellist
	\small\hair 2pt
	\pinlabel {$S\times\{-2\}$} [b] at 200 120
	\pinlabel {$S\times\{2\}$} [b] at 375 120
	\endlabellist
	\includegraphics[scale=.5]{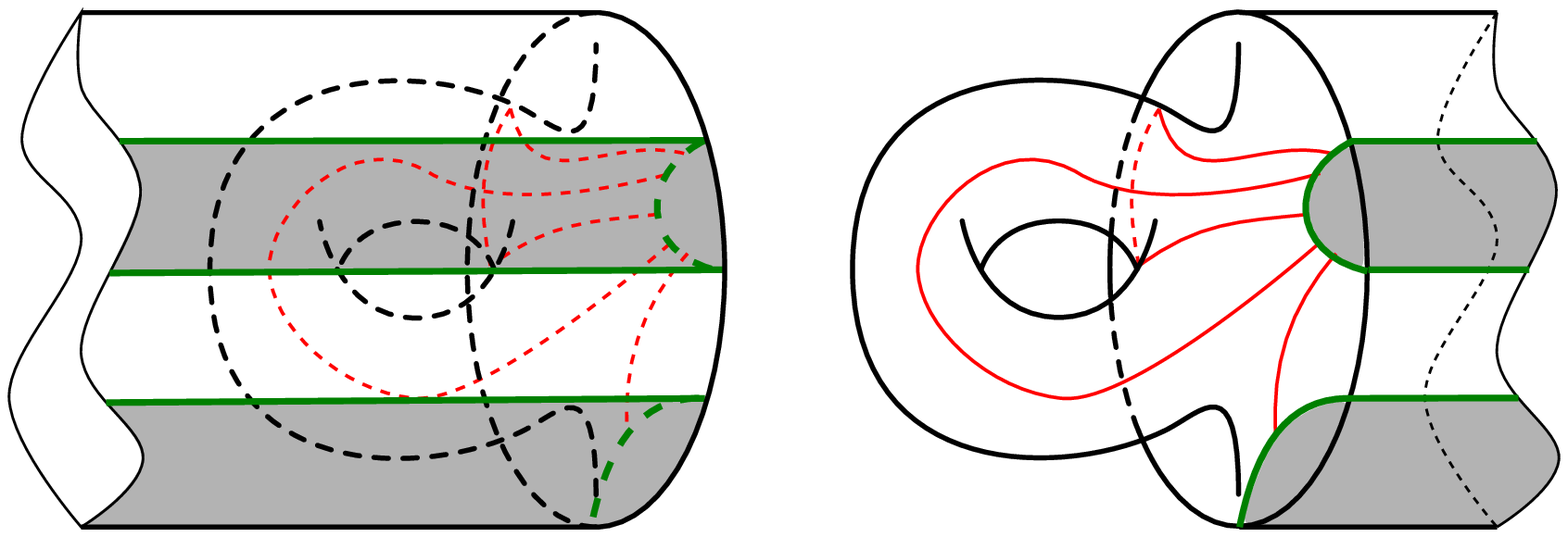}
	\caption{$W$ parametrized by $\Z$.}
\end{subfigure}
\begin{subfigure}[b]{\linewidth}
	\centering
	\labellist
	\small\hair 2pt
	\pinlabel {$S\times\{-2\}$} [b] at 25 120
	\pinlabel {$S\times\{2\}$} [b] at 200 120
	\endlabellist
	\includegraphics[scale=.5]{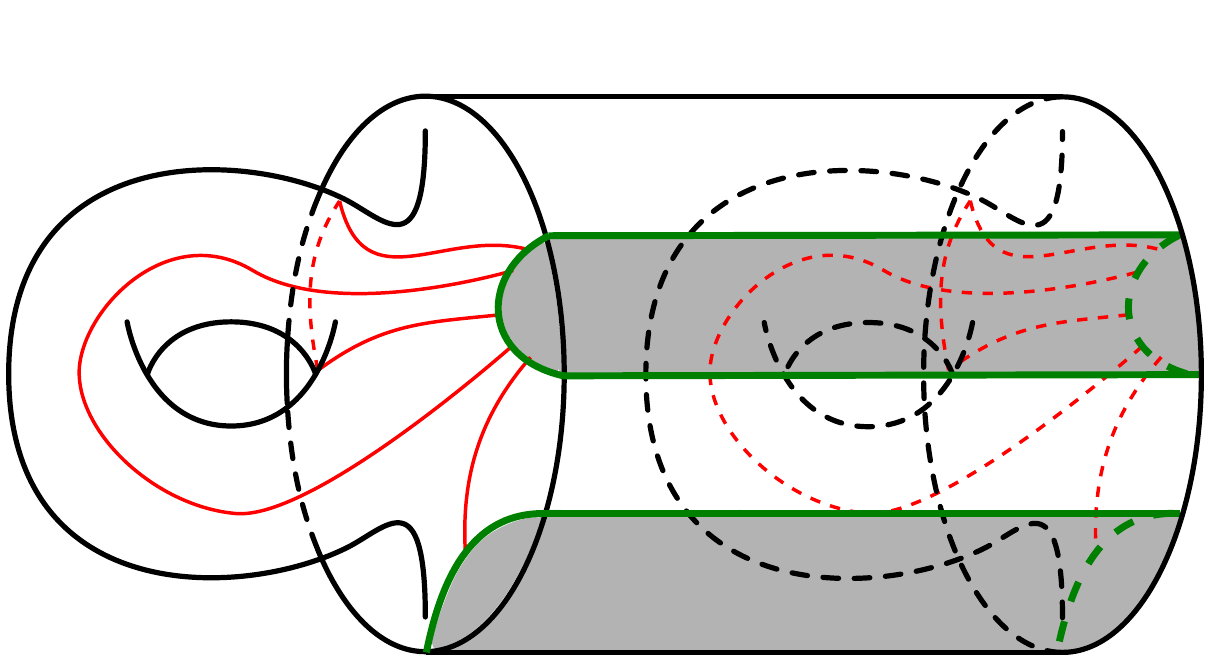}
	\caption{$T$ parametrized by $-\Z$.}
\end{subfigure}
\begin{subfigure}[b]{\linewidth}
	\centering
	\labellist
	\small\hair 2pt
	\pinlabel {$S\times\{-2\}$} [b] at 15 120
	\pinlabel {$S\times\{-1\}$} [b] at 200 120
	\pinlabel {$S\times\{1\}$} [b] at 375 120
	\pinlabel {$S\times\{2\}$} [b] at 560 120
	\endlabellist
	\includegraphics[scale=.5]{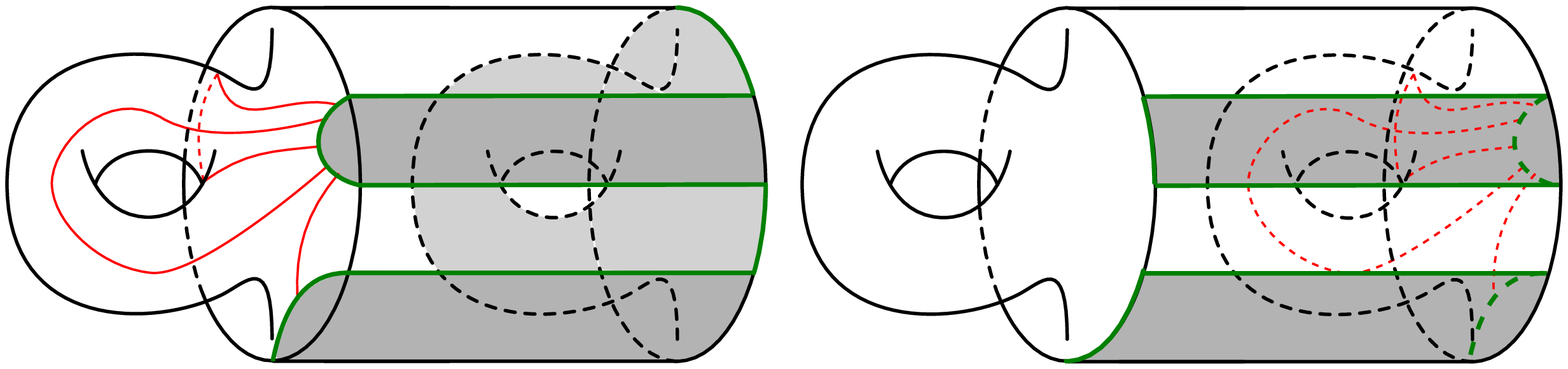}
	\caption{$P$ parametrized by $-\Z$.}
\end{subfigure}
\caption{Bordered sutured decomposition of $(Y,\Gamma)$ and $(Y',\Gamma')$.}
\label{fig:decomp_manifolds}
\end{figure}

We will look at the relationship between $\CSFD(T)$ and $\CSFD(P)$.
For simplicity we will assume $S$ is connected. The argument easily generalizes to multiple
connected components. Alternatively, it follows by induction.
The Heegaard diagrams $\HH_T$ and $\HH_P$ for $T$ and $P$ are shown in figure~\ref{fig:decomp_HH}.
Since all regions $D$ have nonzero $\del^\del D$, the diagrams are automatically
provincially admissible.

\begin{figure}
\begin{subfigure}[b]{.45\linewidth}
	\centering
	\labellist
	\small\hair 2pt
	\pinlabel $A$ at 215 200
	\pinlabel \rotatebox{180}{\reflectbox{$A$}} at 215 140
	\pinlabel $B$ at 180 170
	\pinlabel \rotatebox{180}{\reflectbox{$B$}} at 180 110
	\pinlabel $C$ at 145 80
	\pinlabel \rotatebox{180}{\reflectbox{$C$}} at 145 -10
	\small\hair 5pt
	\pinlabel {-$\mathcal{Z}_1$} [r] at 110 95
	\pinlabel {-$\mathcal{Z}_2$} [l] at 240 95
	\endlabellist
	\includegraphics[scale=.65]{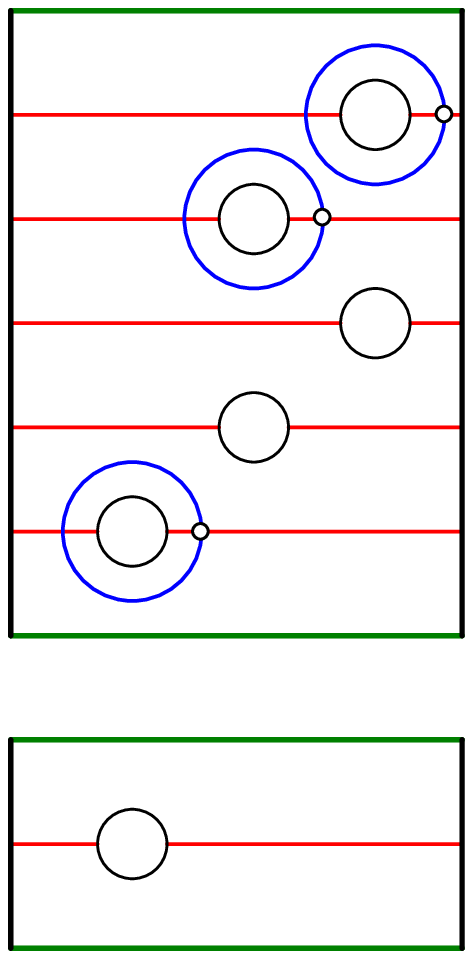}
	\caption{$\mathcal{H}_T$ with $\xgen_T$ marked.}
	\label{subfig:HH_T}
\end{subfigure}
\begin{subfigure}[b]{.45\linewidth}
	\centering
	\labellist
	\small\hair 2pt
	\pinlabel $A$ at 215 200
	\pinlabel \rotatebox{180}{\reflectbox{$A$}} at 215 140
	\pinlabel $B$ at 180 170
	\pinlabel \rotatebox{180}{\reflectbox{$B$}} at 180 110
	\pinlabel $C$ at 145 80
	\pinlabel \rotatebox{180}{\reflectbox{$C$}} at 145 -10
	\small\hair 5pt
	\pinlabel {-$\mathcal{Z}_1$} [r] at 0 95
	\pinlabel {-$\mathcal{Z}_2$} [l] at 240 95
	\endlabellist
	\includegraphics[scale=.65]{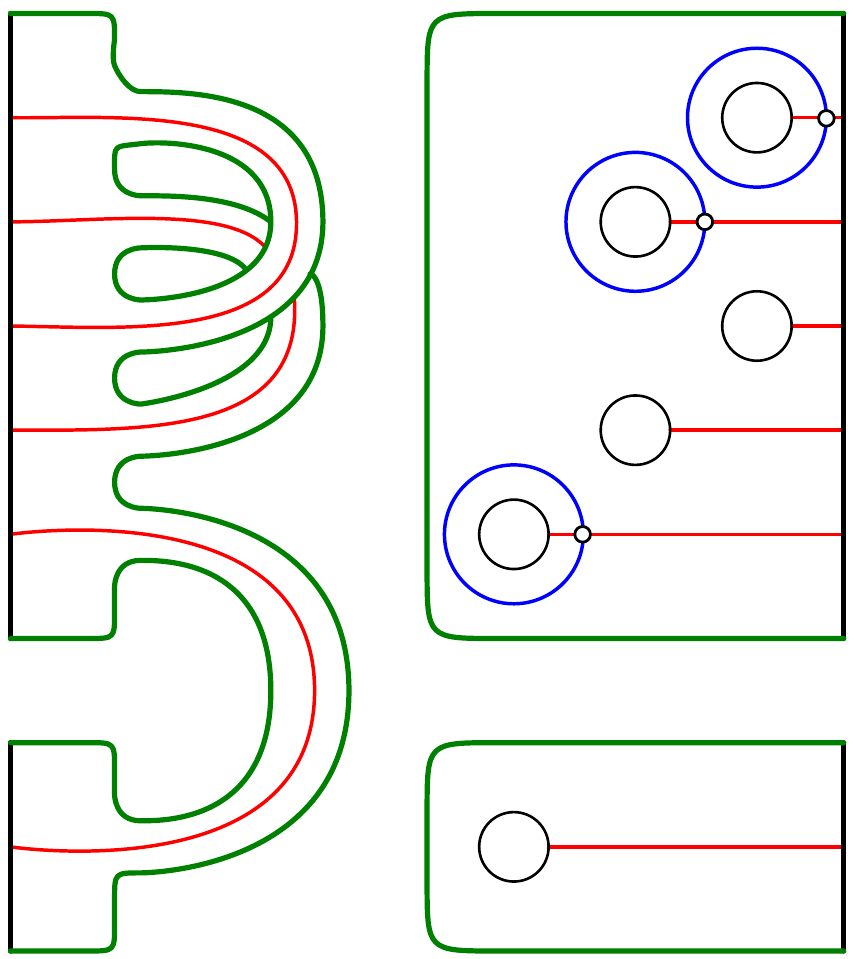}
	\caption{$\mathcal{H}_P$ with $\xgen_P$ marked.}
	\label{subfig:HH_P}
\end{subfigure}
\caption{Heegaard diagrams for $P$ and $T$.
Here $A$, $B$, and $C$ denote 1--handles.}
\label{fig:decomp_HH}
\end{figure}

The algebra $\A(\Z)$ splits as $\A(\Z_1)\otimes\Z(\Z_2)$, and each idempotent $I\in\I(\Z)$
splits as the product $I=I_1\otimes I_2$, where $I_1\in\I(\Z_1)$ and $I_2\in\I(\Z_2)$.
Moreover, $I_1$ is in a summand $\I(\Z,l)$ for some $l=0,\ldots,k$. Denote
this number by $l(I)$. Intuitively, $l(I)$ means ``how many arcs on the $\Z_1$ portion of $\Z$
does $I$ occupy''. Similarly, for a generator $\xgen$ we can define $l(\xgen)=l(I(\obar(\xgen))$.

Notice that $\HH_P$ has a unique generator $\xgen_{P}$,
such that $l(\xgen_P)=k$. 
Moreover, there are only two regions in the diagram, and both of them are boundary regions.
Therefore, no curves contribute to $\delta$. Thus, $\CSFD(P)$ has a unique generator $\xgen_P$,
with $\delta(\xgen_P)=0$.

Now, consider $\HH_T$. Every $\alpha^a$ arc intersects a unique $\beta$ curve, and any $\beta$ curve
intersects a unique pair of $\alpha^a$ arcs, that correspond in $-\Z_1$ and $-\Z_2\cong\Z_1$.
Therefore for any $s\subset\{1,\ldots,k\}$ there is a unique generator $\xgen_s\in\G(\HH_t)$, such that
$I(\obar(\xgen))=I_1(s)\otimes I_2(\sbar)$, and $l(\xgen_s)=\#s$. These are all the elements of $\G(\HH_P)$.

Consider all the $\spinc$--structures in $\spinc(T,\del T\setminus S\times\{\pm 2\})$. By
Poincar\'e duality they are an affine space over 
$H_1(T,S\times\{\pm 2\})=H_1(S\times[-2,2],S\times\{\pm 2\})\cong\ZZ$, generated by an arc
$\mu=\{p\}\times[-2,2]$ for any $p\in S$.
It is easy to see that $\epsilon(\xgen,\ygen)=(l(\xgen)-l(\ygen))\cdot[\mu]$. Thus, for any
$\xgen\in\G(\HH_T)$, its $\spinc$--structure $\s(\xgen)$ depends only on $l(\xgen)$.
In particular, there is a unique generator $\xgen_T$, in the $\spinc$--structure
$\s_k=\s(\xgen_T)$ which corresponds to $l=k$.

Since $l(\xgen_T)=k$, any class
$B\in\pi_2(\xgen,\xgen)$ that contributes to $\delta$ could not hit any Reeb chords on the $\Z_2$ side, and
$\del^{\del}B\cap\Zmid_2$ should be empty. But any elementary region in the diagram
hits Reeb chords on both sides. Therefore any such $B$ should be $0$, and $\delta(\xgen_T)=0$.

Notice that $\G(\HH_P)=\{\xgen_P\}\cong\{\xgen_T\}=\G(\HH_T,\s_k)$,
$I(\obar(\xgen_P))=I(\obar(\xgen_T))=I_1(\{1,\ldots,k\})\otimes I_2(\varnothing)$, and $\delta(\xgen_P)=\delta(\xgen_T)=0$.
Therefore $\CSFD(\HH_P)\cong\CSFD(\HH_T,\s_k)$, and $\CSFD(P)\simeq\CSFD(T,\s_k)$, as
type $D$ structures over $\A(\Z)$.

By the pairing theorem,
\begin{equation*}
\begin{split}
\SFC(Y',\Gamma')&\simeq\CSFA(W)\sqtens\CSFD(P)\\
&\simeq\CSFA(W)\sqtens\CSFD(T,\s_k)
\simeq\bigoplus_{\s|_T=\s_k}\SFC(Y,\Gamma,\s).
\end{split}
\end{equation*}

To finish the proof, we need to check that $\s\in\spinc(Y,\del Y)$ is outward to $S$ if and only if
$\s|_T=\s_k$. This follows from the fact that being outward to $S$ is a local condition. In $T=S\times[-2,2]$
the existence of an outward vector field representing $\s_l$ is equivalent to $l=k$.
\end{proof}

In fact, using bimodules the proof carries through even when $W$ has another bordered component $\Z'$. Thus we get a
somewhat stronger version of the formula.

\begin{thm}
If $(Y,\Gamma,\Z,\phi)$ is a bordered sutured manifold, and $S$ is a nice decomposing surface, where
$\del S\subset\del Y\setminus F(\Z)$, and $(Y',\Gamma',\Z,\phi)$ is obtained by decomposing along $S$, then
the following formulas hold.
\begin{align*}
\CSFD(Y',\Gamma')&\simeq\bigoplus_{\textrm{$\s$ outward to $S$}}\CSFD(Y,\Gamma,\s),\\
\CSFA(Y',\Gamma')&\simeq\bigoplus_{\textrm{$\s$ outward to $S$}}\CSFA(Y,\Gamma,\s).
\end{align*}
\end{thm}
\begin{proof}
The first statement follows as in theorem~\ref{thm:decomp}, using $\BSDA(W)$.
The second follows analogously, replacing the argument for $\CSFD(T)$ and $\CSFD(P)$ with
one for $\CSFA(T)$ and $\CSFA(P)$.
\end{proof}

\bibliographystyle{hamsplain2}
\bibliography{/home/rumen/papers/bibliography/all}

\end{document}